\documentclass{amsart}
\usepackage{amsaddr}

\usepackage{amsmath, amssymb, amsfonts, amsbsy}
\usepackage[margin=3cm]{geometry}
\usepackage[usenames,dvipsnames]{color}
\usepackage{hyperref}
\usepackage{xcolor}
\usepackage{todonotes}
\usepackage{verbatim}

\def\tc{\textcolor}

\newcommand{\ja}[1]{\tc{magenta}{#1}}

\hypersetup{
	colorlinks=true,
	linkcolor= blue,
	citecolor=blue
}
\usepackage{stmaryrd}
\usepackage[all]{xy}

\usepackage[shortlabels]{enumitem}
\setlist[enumerate]{label={\bf\arabic*}., itemsep=1ex,leftmargin=1.0cm}
\setlist[enumerate,2]{label={\bf\alph*)}, itemsep=1ex,leftmargin=1.0cm}
\setlist[enumerate,3]{label={\bf\roman*)}, itemsep=1ex,leftmargin=1.0cm}
\usepackage[T1]{fontenc}
\usepackage[utf8]{inputenc}

\usepackage{cite}

\newtheorem{thm}{\protect\theoremname}[section]
\newtheorem*{thm*}{\protect\theoremname}
\newtheorem{prop}[thm]{\protect\propositionname}
\newtheorem{fact}[thm]{\protect\factname}
\newtheorem{lem}[thm]{\protect\lemmaname}
\newtheorem{obs}[thm]{\protect\observationname}
\newtheorem{cor}[thm]{\protect\corollaryname}

\newtheorem{claim}[thm]{\protect\claimname}
\newtheorem*{claim*}{\protect\claimname}
\newtheorem{question}{\protect\questionname}
\newtheorem{conjecture}[question]{\protect\conjecturename}
\newtheorem{problem}[question]{\protect\problemname}

\theoremstyle{definition}
\newtheorem{defn}[thm]{\protect\definitionname}

\newtheorem{rem}[thm]{\protect\remarkname}

\newcommand{\fraisse}{Fra\"iss\'e } 

\def\Ind#1#2{#1\setbox0=\hbox{$#1x$}\kern\wd0\hbox to 0pt{\hss$#1\mid$\hss}
\lower.9\ht0\hbox to 0pt{\hss$#1\smile$\hss}\kern\wd0}
\def\ind{\mathop{\mathpalette\Ind{}}}
\def\Notind#1#2{#1\setbox0=\hbox{$#1x$}\kern\wd0\hbox to 0pt{\mathchardef
\nn="3236\hss$#1\nn$\kern1.4\wd0\hss}\hbox to
0pt{\hss$#1\mid$\hss}\lower.9\ht0
\hbox to 0pt{\hss$#1\smile$\hss}\kern\wd0}

\providecommand{\theoremname}{Theorem}
\providecommand{\propositionname}{Proposition}
\providecommand{\factname}{Fact}
\providecommand{\lemmaname}{Lemma}
\providecommand{\observationname}{Observation}
\providecommand{\corollaryname}{Corollary}
\providecommand{\definitionname}{Definition}
\providecommand{\consname}{Construction}
\providecommand{\remarkname}{Remark}
\providecommand{\examplename}{Example}
\providecommand{\claimname}{Claim}
\providecommand{\questionname}{Question}
\providecommand{\conjecturename}{Conjecture}
\providecommand{\problemname}{Problem}

\global\long\def\fl#1{\mathrm{Flim}\left(#1\right)}

\title[Group Topologies on auto. groups of homogeneous str.]{Group topologies on automorphism groups of homogeneous structures}

\author{Zaniar Ghadernezhad \textsuperscript{$\dagger$}}\thanks{\noindent\textsuperscript{$\dagger$} Supported by the Leverhulme Grant RPG-2017-179}
\address[A1]{Department of Mathematics, Imperial College London}
\author{Javier de la Nuez Gonz\'alez \textsuperscript{$\dagger$}}
\thanks{\noindent\textsuperscript{$\dagger$}Supported by the European Research Council under the European Unions Seventh Framework Programme (FP7/2007- 2013)/ERC Grant Agreements No. 291111 and No. 336983 and from the Basque Government Grant IT974-16.}
\address[A2]{Korea Institute for Advanced Study (KIAS)}

\email[A1,A2]{z.ghadernezhad@imperial.ac.uk, javier.delanuez@gmail.com}

\subjclass{22A05,22F50(primary),03C15,03C45(secondary)}

\global\long\def\dim#1{\mathrm{d}\left(#1\right)}

\global\long\def\cl#1{\mathrm{cl}\left(#1\right)}
\global\long\def\diam#1{\mathrm{diam}\left(#1\right)}

\global\long\def\gcl{\mathrm{gcl}}
\global\long\def\acl{\mathrm{acl}}
\global\long\def\typ{\mathrm{tp}}

\global\long\def\Aut#1{\mathrm{Aut}\left(#1\right)}
\global\long\def\id{\mathrm{id}}

\global\long\def\Isom{\mathrm{Isom}}

\renewcommand\leq{\leqslant}
\renewcommand\geq{\geqslant}

\newcommand{\frp}[0]{\ast}
\newcommand{\subg}[1]{\langle #1 \rangle}

\newcommand{\N}[0]{\mathbb{N}}

\newcommand{\U}[0]{\mathcal{U}}
\newcommand{\Ba}[0]{\mathcal{B}}

\newcommand{\Q}[0]{\mathbb{Q}}
\newcommand{\nin}[0]{\notin}
\newcommand{\sst}[1]{\substack{#1}}
\newcommand{\idn}[0]{\mathcal{N}_{\tau}(1)}
\newcommand{\nb}[0]{N^{sp}}

\newcommand{\op}[0]{\oplus}

\newcommand{\mov}[0]{\mathrm {Supp}}

\newcommand{\dom}[0]{\mathrm {dom}}

\newcommand{\R}[0]{\mathbb{R}}

\newtheorem{thmx}{Theorem}

\newcommand{\RR}[0]{(R,0,\leq,\oplus)}
\newcommand{\rr}[0]{\mathcal{R}}
\newcommand{\M}[0]{\mathcal{M}}

\newcommand{\fps}[1]{[#1]^{<\omega}}      
\newcommand{\X}[0]{\mathfrak{X}}

\begin{document}

\maketitle

\begin{abstract}
We classify all group topologies coarser than the topology of stabilizers of finite sets in the case of automorphism groups of countable free-homogeneous structures, Urysohn space and Urysohn sphere.
\end{abstract}

\hypersetup{linkcolor=black}
\hypersetup{linkcolor=blue}

\section{Introduction}
  
  \subsection*{Minimality}
    A topological group  $(G,\tau)$ consists of a group $(G,\cdot)$ and a topology $\tau$ on $G$ such that the map  $\rho: G\times G\rightarrow G$ where $\rho(g,h)= gh^{-1}$ is jointly continuous.
    \begin{defn}
    A Hausdorff topological group $G$ is called \emph{minimal} if every bijective continuous homomorphism from $G$ to another Hausdorff topological group is a homeomorphism. The group $G$ is \emph{totally minimal} if every continuous surjective homomorphism to a Hausdorff topological group is open.
    \end{defn}
    In fact $(G,\tau)$ is minimal if and only if $G$ does not admit a strictly coarser Hausdorff group topology than $\tau$. Furthermore, it is also clear that every totally minimal group is minimal.
    
    The notion of minimality for topological groups was introduced as back as $1971$ as a generalization of compactness. In fact it is easy to see that any compact Hausdorff topological group is minimal. For more information about minimality, we refer the reader to the survey by Dikranjan and Megrelishvili \cite{DiMeg}.

    Given a group $G$ of permutations of some set $\Omega$ and $A\subseteq\Omega$, let $G_{A}=\{g\in G\,|\,\forall a\in A,\,\,ga=a\}$. Let $\fps{\Omega}$ be the set of all finite subsets of $\Omega$. The collection $\{G_{A}\,|\,A\in \fps{\Omega}\}$
    is a base of neighbourhoods at the identity of a group topology which we call the {\it standard topology} and denote by $\tau_{st}$. More generally for each $G$-invariant $X\subseteq\Omega$ there is an associated group topology $\tau^{X}_{st}$ generated by $\{G_{A}\,|\,A\in \fps{X}\}$.
    
    One of the earliest results on minimality due to Gaughan \cite{Gua} states that $(S_{\infty},\tau_{st})$ is totally minimal where $S_\infty$ denotes the group of all permutations of a countable set $\Omega$.
    
    \newcommand{\uni}[0]{M}   
    Given a countable first order structure $\M $ with universe $\uni$, the automorphism group of $\M$ is a $\tau_{st}$-closed subgroup of $S_{\infty}=S(\uni)$ and vice versa: any closed subgroup of $S(\uni)$ is the automorphism group of some countable structure on $\uni$.
    The interplay between the dynamical properties of $\Aut{\M}$ and the logical and combinatorial  properties of $\M $ has been widely studied in the literature, beginning with the characterization due to Engeler, Ryll-Nardzewski, Svenonius and others of oligomorphic subgroups of $S_{\infty}$ as the automorphism groups of $\omega$-categorical countable structures. Recall that an oligomorphic group is a closed subgroup of
    $S_{\infty}$ whose diagonal action on $\uni^{n}$ has finitely many orbits, for each $n\in \mathbb N$.
    
    In this context $\tau_{st}$ is often referred to in the literature as the point-wise convergence topology, in implicit reference to the discrete metric on $\uni$. When discussing isometry groups it will be important for us distinguishing between $\tau_{st}$ as above and the point-wise convergence topology relative to the metric in question which we will denote by $\tau_{m}$, so we will avoid this practice.
    
    In light of the above the following is thus a natural question, already asked in \cite{DiMeg}
    \begin{problem}
    \label{minimality problem}Let $\M$ be a countable $\omega$-categorical ($\omega$-saturated, sufficiently nice) first order structure and $G=\Aut{\M}$. When is $(G,\tau_{st})$ (totally) minimal?
    \end{problem}
    
    A deep result in this direction appeared in recent work by Ben Yaacov and Tsankov \cite{BenTs}, where the authors show that automorphism groups of countable $\omega$-categorical, stable continuous structures are totally minimal with respect to the point-wise convergence topology. This specializes to the result that the automorphism groups of classical $\omega$-categorical stable structures are totally minimal with respect to $\tau_{st}$.
    
    Not all oligomorphic groups are minimal with respect to $\tau_{st}$. As pointed out in \cite{BenTs},
    an example of this is $\Aut{\Q,<}$ (see Theorem \ref{distality} for a generalization). However even in those cases it is possible to formulate the following more general question:
    
    \begin{problem}
    \label{classification problem}Let $\M$ be a countable $\omega$-categorical (or sufficiently nice) first order structure and $G=\Aut{\M}$. Describe the lattice of all Hausdorff group topologies on $G$ coarser than $\tau_{st}$.
    \end{problem}
    
    This work was mainly motivated by \cite{BenTs} and is meant as a preliminary exploration of Problems $1$ and $2$ in the classical setting outside the stability constraint.
    
    In its broadest lines the strategy followed by \cite{BenTs} goes back to \cite{Usp}, where it was shown by Uspenskij that the isometry group of the Urysohn sphere is totally minimal with the point-wise convergence topology.
    Both proofs rely on the assumption that the group in question is Roelcke precompact and
    use a well behaved independece relation among (small) subsets of the structure to endow the Roelcke precompletion of the group with a topological semigroup structure. Information on the topological quotients of the original group is then recovered from the latter via the functoriality of Roelcke compactification and Ellis lemma.
    Recall that a topological group $(G,\tau)$ is {\it Roelcke precompact} if for any neighbourhood $W$ of $1$ there exists a finite $F\subset G$ such that $WFW=G$. For closed subgroups of $S_{\infty}$ this is equivalent to being oligomorphic.
    
    In contrast, our methods for obtaining (partial) minimality results are completely elementary. There are drawbacks to this lack of sophistication: for instance, we are not able to recover the result in \cite{BenTs} for classical structures. On the other hand we do not rely on assumptions of Roelcke pre-compactness (except for certain residual assumptions in some cases). In particular, we are able to answer in the positive the question about the minimality of the isometry group of the (unbounded) Urysohn space posed in \cite{Usp} (Theorem \ref{Urysohn thm 1}).
    It is worth emphasizing that in some cases we manage to obtain complete classifications of continuous homomorphic images of topological groups which are neither Roelcke precompact nor separable (see Theorem \ref{3 topologies theorem}).
    
    A section by section summary with our main results can be found below.
    
  \subsection*{Free amalgamation and one basedness}
    
    Section \ref{sec-relative-min} provides a simple technical criterion (Lemma \ref{main minimality proposition}) of (relative) minimality for $\tau_{st}$ from which more concrete applications
    are derived in Section \ref{sec-ind-free-1based}.
    
    Recall that the {\it free amalgam} of two relational structures $A,B$ over a common substructure $C$ is the structure resulting from taking unrelated copies of $A$ and $B$ and then gluing together the two copies of $C$ without adding any extra relations. A {\it free amalgamation} class $\mathcal{K}$ is a collection of finite structures closed under substructures and free amalgams. Associated with any such $\mathcal{K}$ there is a unique {\it \fraisse limit}: a countable structure in which every $A\in\mathcal{K}$ embeds and which is ultra-homogeneous, i.e., any finite partial isomorphism extends to an automorphism of the structure.
    
    \newcommand{\textfreedomtheorem}[0]{
    Let $\M$ be the Fra\"iss\'e limit of a free amalgamation class in a countable relational structure. Let $G=\Aut{\M}$. Then any group topology $\tau\subseteq\tau_{st}$ on $G$ is of the form $\tau_{st}^{X}$, where $X\subseteq M$ is some $G$-invariant set. In particular, if the action of $G$ on $M$ is transitive, then $(G,\tau_{st})$ is totally minimal.
    }
    
    \begin{thmx}\
    \label{freedom theorem} \textfreedomtheorem
    \end{thmx}
    
    Simple structures (i.e. theories) occupy an important place in classification theory. We refer the reader to \cite{TentZBook}, \cite{WagnerSimple} and  \cite{kim1997simple} for the definition of simple theories, forking and canonical bases. We say that a simple theory $T$ is {\it one-based} if $Cb(a/A)\subseteq bdd(a)$ for any hyperimagianry element $a$ and small subset $A$ of the monster model.
    
    \newcommand{\textsimpletheorem}[0]{
    Let $\M $ be a simple, $\omega$-saturated countable structure with elimination of hyperimaginaries, locally finite algebraic closure and weak elimination of imaginaries.
    Assume furthermore that $Th(\M)$ is one-based.
    Let $G=\Aut{\M}$. Then
    \begin{enumerate}
    \item If $G$ acts transitively on $M$, then $(G,\tau_{st})$ is minimal.
    \item If all singletons are algebraically closed, then any group topology $\tau$ on $G$ coarser than $\tau_{st}$ is of the form
    $\tau_{st}^{X}$ for some $G$-invariant $X\subseteq M$.
    \end{enumerate}
    }
    \begin{thmx}
    \label{simple structures theorem}
    \textsimpletheorem
    \end{thmx}
    
    By an \emph{independence relation} is usually meant some ternary relation $\ind$ on (some) collection of sets of parameters of a structure such that $A\ind_{C}B$ captures the intuitive idea that $B$ does not contain any information about $A$ not already contained in $C$.
    The paradigmatic example is that of forking independence. The connections between the existence of an independence relations on a homogeneous structure satisfying certain axioms and the properties of the automorphism group goes back to \cite{TentZUry} (see also \cite{EGT}). Of particular relevance to us is the freedom axiom, explored in detail in \cite{conant2017axiomatic}. We explain Theorems \ref{freedom theorem} and \ref{simple structures theorem} in terms of the existence of an independence relation satisfying certain sets of axioms. The roles played by stationarity and the freedom axiom in \ref{freedom theorem} are replaced by one basedness and the independence property respectively in \ref{simple structures theorem}.

  \subsection*{Generalized universal metric spaces}
    \newcommand{\UR}[0]{\mathcal{U}_{\R}}
    \newcommand{\US}[0]{\mathcal{U}_{[0,1]}}
    
    Urysohn universal space $\UR$ is a homogeneous space that contains all separable metric spaces due to Urysohn.
    It is both {\it $\omega$-universal}, i.e. it contains any finite metric space as a subspace and {\it $\omega$-homogeneous}, i.e. any partial isometry between finite subspaces of $\U$ extends to some global isometry.
    Associated with the class of metric spaces with diameter at most $1$ there is an object with similar properties $\US$, known as the {\it Urysohn sphere}.
    The isometry groups $\Isom(\UR)$ and $\Isom(\US)$ endowed with the point-wise convergence topology $\tau_{m}$ (`m' is for `metric')
    are Polish groups whose algebraic and dynamical properties have been widely studied. It is known, for instance, that any Polish group is isomorphic to a closed subgroup of $\Isom(\UR)$ and that $\Isom(\UR)$ is extremely amenable.
    
    It is shown in \cite{Usp} that any continouous quotient of $(\Isom(\US),\tau_{m})$ is either trivial or a homeomorphism.
    Theorem \ref{Urysohn thm 1} below extends this result to $\Isom(\UR)$.
    
    We work in the framework of generalized metric spaces introduced by Conant in \cite{conant2017axiomatic}.
    A {\it distance monoid} is an abelian monoid endowed with a compatible linear order (see Subsection \ref{subsec-setting-Ury} for more details).
    Given  a distance monoid $\rr=(R,0,\oplus,\leq)$ an $\rr$-metric space is a set $X$ endowed with a map
    $d:X^{2}\to R$ satisfying the obvious generalization of the axioms for a metric space.
    In our terminology an $\rr$-Urysohn space $\U$ will be an $\rr$-metric space satisfying the obvious generalization of $\omega$-homogeneity and $\omega$-universality to this setting, ignoring any separability and cardinality considerations.
    If $\rr\models \forall x\neq 0\,\,\exists y\neq 0\,\,y\oplus y\leq x$, then the collection
    $\{N_{u}(\epsilon)\,|\,u\in\U,\epsilon\in R\setminus\{0\}\}$, where $N_{u}(\epsilon)=\{g\in G\,|\,d(gu,u)\leq\epsilon\}$ generates a group topology on the isometry group of $\U$. For plain metric spaces the result is the point-wise convergence topology so we keep the notation $\tau_{m}$ in the general case.
    
    We say that a distance monoid $\rr$ as above is {\it archimedean} if for any $r,s\in R\setminus\{0\}$ there exists some $m\in\N$ such that $m\cdot s:=s\oplus s\oplus\dots\oplus s \text{ ($m$ times)}\geq r$.
    
    \newcommand{\texturysohnthmone}[0]{
    Let $\rr=\RR$ be an archimedean distance monoid, $\U$ a
    $\rr$-Urysohn space, $G=\Isom(\U)$ and let $\tau_0$ be either:
    \begin{itemize}
    \item $\tau_{m}$ in case for any $r\in R\setminus\{0\}$ there exists $s\in R\setminus\{0\}$ with $s\oplus s \leq r$; or,
    \item $\tau_{st}$ otherwise.
    \end{itemize}
    Then $\tau_0$ is the coarsest non-trivial group topology on $G$ coarser than the stabilizer topology $\tau_{st}$. In particular, $(G,\tau_0)$ is totally minimal.
    }
    
    \begin{thmx}
    \label{Urysohn thm 1}
    \texturysohnthmone
    \end{thmx}
    
    Given some $S\subseteq\R$ closed under addition and $b\in S_{>0}\cup\{\infty\}$, we let $\mathcal{S}_{b}$ be the distance monoid given by the tuple $\{\{r\in S\,|\,0\leq r\leq b\},0,\leq,+_{b}\}$, where
    $x+_{b}y=\min\{x+y,b\}$.
    
    \newcommand{\textthreetopologiestheorem}[0]{
    Let $S$ be a dense subgroup of $\R$ and $b\in S_{>0}\cup\{\infty\}$, $\U$ an $\mathcal{S}_{b}$-Urysohn space and $G=\Isom(\U)$. Then there are exactly $4$ group topologies on $G$ coarser than $\tau_{st}$:
    $$
    \tau_{st}\supsetneq\tau_{0^{+},0}\supsetneq\tau_{m}\supsetneq\{\emptyset,G\}
    $$
    where $\tau_{0^{+},0}$ is the topology generated at the identity
    by the collection
    \begin{align*}
    \{\{g\in G\,|\,d(gu,v)\leq d(u,v)\}\,|\,u,v\in\U,d(u,v)>0\}.
    \end{align*}
    }
    
    \newcommand{\textthreetopologiestheoremold}[0]{
    Let $\rr$ be a standard archimedean distance monoid with no minimal positive element, $\U$ an $\rr$-Urysohn space and $G=\Isom(\U)$. Then there are exactly $4$ group topologies on $G$ coarser than $\tau_{st}$:
    $$
    \tau_{st}\supsetneq\tau_{0^{+},0}\supsetneq\tau_{m}\supsetneq\{\emptyset,G\}
    $$
    where $\tau_{0^{+},0}$ is the topology whose system of neighbourhoods of the identity is generated
    by the collection $\{\mathcal N^{sp}_{u,v}\,|\,u,v\in\U,d(u,v)>0\}$, where $\mathcal N_{u,v}^{sp}:=\{g\in G\,|\,d(gu,v)=d(u,v)\}$ for any $u,v\in\U$ with $d(u,v)>0$.
    }
    \begin{thmx}
    \label{3 topologies theorem}
    \textthreetopologiestheorem
    \end{thmx}
    
    In Section \ref{Urysohn parametrization section} we describe a general family of group topologies on
    the isometry group of an $\rr$-Urysohn space $\U$ that includes all the topologies involved in the two results above. Theorems \ref{Urysohn thm 1} and \ref{3 topologies theorem}, can be seen as evidence for the much more general conjecture that these are in fact all group topologies coarser than $\tau_{st}$ on $\Isom(\U)$.

  \subsection*{Algebraic minimality: the Zariski topology}
    Given a group $G$ the {\it Zariski topology} $\tau_Z$, is generated by the subbase consisting of the sets $\{ x \in G\,|\, x^{\epsilon_1}g_1 x^{\epsilon_2} g_2 \cdots x^{\epsilon_n} g_n \neq 1 \}$, where $n\in \mathbb N$, $g_1,\dots,g_n\in G$, and $\epsilon_1,\dots,\epsilon_n \in \{-1,1\}$. According to the result of Gaughan in \cite{Gua} for the group $S_\infty$ the Zariski topology $\tau_Z$ and $\tau_{st}$ coincide. In Section \ref{sec-Zariski} we investigate the following general question.
    \begin{question}
    For which (sufficiently homogeneous) structures is it true that $\tau_{Z}=\tau_{st}$. For which of them is the Zariski topology a group topology?
    \end{question}
    
    First we provide a variety of Fra\"iss\'e limits for which the question above has a negative answer. In all cases this follows via Lemma \ref{using meagerness} from the property that over $\Aut{\M}$ of non-trivial equations in one variable have meager sets of solutions. The latter is in turn from the criterion formulated in Lemma \ref{lem-comeager}, according to which the conclusion holds whenever any $\alpha\in \Aut{\M}\setminus\{1\}$ is what we call {\it strongly unbounded} (Definition \ref{def-st-ub}). Intuitively, the latter means that points `largely displaced' by $\alpha$ are in some sense  dense in $M$. Theorem \ref{thm-main-zariski} below collects some miscellaneous results found  in  Corollary \ref{zar-free-AP} (for \ref{zar fr am}),  \ref{cor-zar-ury} (for \ref{zar ury}), \ref{cor-zar-rant} (for \ref{zar rant}) and Corollary \ref{cor-products} (for \ref{zar prod}) below.
    \newcommand{\textthmmainzariski}[0]{
    The Zarisski topology $\tau_Z$ on $\Aut{\M}$ is not a group topology if $\M=\fl{\mathcal{K}}$ for a \fraisse class $\mathcal{K}$ in a relational language $\mathfrak L$ in each of the following cases:
    \begin{enumerate}
    \item \label{zar fr am}$\mathcal{K}$ is a non-trivial free amalgamation class and the action of $G$ on $M$ is transitive;
    \item \label{zar ury}$\M$ is the rational Urysohn space;
    \item \label{zar rant}$\M$ is the random tournament;
    \item \label{zar prod}$\mathcal{K}$ is of the form $\mathcal{K}_{1}\otimes \mathcal{K}_{2}$ (see Definition \ref{def: tensor product of fraisse classes}) for strong amalgamation classes $\mathcal{K}_{1}$ and $\mathcal{K}_{2}$ where:
    \begin{itemize}
    \item   $\mathcal{K}_{1}$ is non-trivial
    and either it is as in \ref{zar ury} or the action of $\Aut{\fl{\mathcal{K}_{1}}}$ on the set $\uni^{2}\setminus\{(a,a)\}_{a\in M}$ is transitive.
    \item $\fl{\mathcal{K}_{2}}$ is the countable dense meet tree, the cyclic tournament $S(2)$ or $(\mathbb{Q},<)$.
    \end{itemize}
    \end{enumerate}
    }
    \begin{thmx}
    \label{thm-main-zariski}
    \textthmmainzariski
    \end{thmx}
    Here we say that $\mathcal{K}$ is {\it trivial} if the equality type of a tuple from $M$ determines its type or, equivalently, if $\Aut{\M}$ is the full symmetric group.
    Additionally in Corollary \ref{Hrus-zar-main-hh} we can prove the following:
    \begin{thmx}
    \label{thm-zar-Hru}
    Suppose $\M^\eta$ is the  Hrushovski generic structure that is obtained from a pre-dimension function with the coefficient $\eta\in (0,1]$. Then the Zariski topology for $\Aut
    {\M^\eta}$ is not a group topology.
    \end{thmx}
    On the flip side there is the following positive result:
    \newcommand{\textnicezariskitheorem}[0]{
    The Zariski topology $\tau_{Z}$ on $\Aut{\M }$ is a group topology in case $\M $ is one of the following:
    \begin{itemize}
    \item Some reduct of $(\mathbb{Q},<)$;
    \item A countable dense meet-tree or the lexicographically ordered dense meet-tree, in which case $\tau_{Z}=\tau_{st}$;
    \item The cyclic tournament $S(2)$.
    \end{itemize}
    }
    \begin{thmx}
    \label{nice zariski theorem}
    \textnicezariskitheorem
    \end{thmx}
    
  \subsection*{Topologies and partial types}
    Finally in Section \ref{sec-typ}, we present a natural variant of ideas of \cite{Usp} and \cite{BenTs} in the context of automorphism groups of first order structures. Given a structure $\M$ with group of automorphisms $G$, we describe a semi-group of partial types $R^{pa}(\M)$ containing $G$ consisting of partial types and show that any idempotent in $R^{pa}(\M)$ which is invariant under the involution and the action of $G$ can be associated to a group topology on $G$ coarser than $\tau_{st}$.
\section{Review of some classical constructions of homogeneous structures}
  \subsection{Fra\"iss\'e construction}
    \label{subsec-Fra}
    Let us briefly review the Fra\"iss\'e construction method in a relational language. For a more detailed and general introduction see Chapter 6 in \cite{Hodgbo}.
    
    Let $\mathfrak L$ be a relational signature and $\mathcal K$ be a countable class of finite $\mathfrak L$-structures closed under isomorphisms. Suppose $A,B\in \mathcal K$ by $A\subseteq B$ we mean $A$ is an $\mathfrak L$-substructure of $B$. We say $\mathcal K$ is a {\it Fra\"iss\'e class} if it satisfies the following properties:
    \begin{itemize}
    \item (HP) It is closed under substructures;
    \item (JEP) For any $A,B\in\mathcal{K}$ there is $C$ in $\mathcal K$ such that $A,B\subseteq C$;
    \item (AP) Given $A_{1},A_{2},B\in\mathcal{K}$ and isometric embeddings $g_{i}:B\to A_{i}$, $i=1,2$  there exists $C\in\mathcal{K}$ and isometric embeddings $h_{i}:A_{i}\to C$ such that $h_{1}\circ g_{1}=h_{2}\circ g_{2}$.
    \end{itemize}
    We say that a \fraisse class $\mathcal{K}$ has {\it strong amalgamation} if in (AP) we might assume that
    $h_{1}(A_{1})\cap h_{2}(A_{2})=h_{1}(B)$.
    
    According to a theorem of Fra\"iss\'e  for any Fra\"iss\'e class $\mathcal{K}$ there is a unique countable structure $\M$ called the {\it Fra\"iss\'e limit} of $\mathcal K$ and denoted by $\fl{\mathcal K}$, such that:
    \begin{itemize}
    \item $\M$ is {\it ultrahomogeneous}, i.e. every finite partial isomorphism between substructures of $\M$ extends to an automorphism of $\M$;
    \item $Age(\M)$, the collection of all finite substructures of $\M$, coincides with $\mathcal{K}$.
    \end{itemize}
    Classical examples of  Fra\"iss\'e limit structures are $(\mathbb Q,<)$ and the random graph. If $\mathfrak L$ is empty, then $\mathcal{K}$ is the class of finite sets and $\fl{\mathcal{K}}$ an infinite countable set. More in general, we say that $\mathcal{K}$ is {\it trivial} if the equality type of a finite tuple of elements from $M$ determines its type (equivalently, if $\Aut{\M}$ is the full permutation group of $M$).
    
    Suppose $A,B$ and $C$ are structures in some relational language $\mathfrak{L}$ with $A\subseteq B,C$. By the {\it free-amalgam} of $B$ and $C$ over $A$, denoted by $B\otimes_{A}C$, we mean the structure with domain $B\coprod_{A}C$ in which a relation holds if and only if it already did in either $B$ or $C$.
    
    By a {\it free amalgamation class} we mean a class $\mathcal{K}$ of finite structures in a relational language satisfying (HP) and such that $B\otimes_{A}C\in\mathcal{K}$ for any $A,B,C\in\mathcal{K}$ such that $A\subseteq B,C$.   Note this is automatically a Fra\"iss\'e class with strong amalgamation. We write $B\ind^{fr}_{A}C$ if and only if the structure generated by $ABC$ is isomorphic (with the right identifications) with the free amalgam $B\otimes_{A}C$. If $B\ind^{fr}_{\emptyset} C$ we write $B\ind^{fr} C$ and say $B$ and $C$ are {\it free} from each other.

  \subsection{Hrushovski's pre-dimension construction}
    
    \label{sec-Hrush-new}
    Originally Hrushovski's pre-dimension  construction was introduced as a means of producing countable strongly minimal structures which are not field-like or vector-space like. There are many variants of the method, but to fix notation, we consider the following basic case and
    later focus on a version that produces $\omega$-categorical structures. We refer  readers to \cite{WagnerHRCONS}, \cite{BaldShi} and \cite{EGT} for most of the properties that are mentioned here about Hrushovski constructions and some of their variations.
    
    Suppose $s\geq  2$ and $\eta\in (0,1]$ . We work with the class $\mathcal C$ of finite $s$-uniform
    hypergraphs, that is, structures in a language with a single $s$-ary relation symbol $R(x_1,\dots,x_s)$ whose interpretation is invariant under permutation of coordinates and satisfies $R(x_1,\dots,x_s) \rightarrow \bigwedge_{i<j} (x_i\neq x_j)$.
    
    To each
    $B\in \mathcal C$ we assign the predimension
    $$\delta(B)= |B|-\eta|R[B]|;$$
    where $R[B]$ denotes the set of hyperedges on $B$. For $A\subseteq  B$, we
    define $A \leqslant B$ iff for all $A \subseteq B' \subseteq  B$ we have $\delta(A)\leq \delta(B')$, and let $\mathcal C_\eta: = \{B \in \mathcal  C \,|\, \emptyset \leqslant B\}$. The following is
    standard
    \begin{lem}
    \label{lem-predim}
    Suppose $A,B \subseteq C \in \mathcal  C_\eta$. Then:
    \begin{enumerate}
    \item $\delta(AB) \leq  \delta(A)+ \delta(B)- \delta(A \cap B)$;
    \item If $A \leqslant B$ and $X \subseteq B$, then $A \cap X \leqslant X$;
    \item If $A \leqslant B \leqslant C$, then $A \leqslant C$.
    
    \end{enumerate}
    
    \end{lem}
    If $A,B \subseteq C \in  \mathcal C_\eta$ then we define $\delta(A/B)= \delta(AB)- \delta(B)$. Note that this is equal to
    $|A\backslash B|- \eta|R[AB] \backslash R[B]|$. Then $B \leqslant AB$ iff $\delta(A'/B) \geq 0$ for all $A'\subseteq A$.  Moreover, if $N$ is an infinite $\mathfrak L$-structure such that $A\subseteq N$, we write $A\leqslant N$ whenever $A\leqslant B$ for every finite substructure $B$ of $N$ that contains $A$.
    One can show $\mathcal C_\eta$ has the $\leqslant$-free amalgamation property (cf. Lemma 4.8 in \cite{BaldShi}), by which we mean free amalgamation with $\leqslant$ inclusions.  An analogue of Fra\"iss\'e's theorem holds in this situation:
    \begin{prop}
    \label{prop:sgen}  There is a unique countable structure ${\M^\eta}$, up to isomorphism, satisfying:
    \begin{enumerate}
    \item The set of all finite substructures of $\M^\eta$, up to isomorphism, is precisely  $\mathcal{C}_\eta$;\item $\M^\eta=\bigcup_{i\in\omega}A_{i}$ where $\left(A_{i}:i\in\omega\right)$ is a chain of $\leqslant$-closed finite sets;
    \item If $A\leqslant\M^\eta$ and $A\leqslant B\in\mathcal{C}_{\eta}$, then there is an embedding $f:B\longrightarrow\M^\eta$ with $f\upharpoonright_{A}=\id_{A}$ and $f\left(B\right)\leqslant\M^\eta$.
    \end{enumerate}
    \end{prop}
    The structure $\M^\eta$, that is obtained in the above proposition, is called the
    {\it Hrushovski generic} structure.

    \subsubsection{$\omega$-categorical case}
      
      Here we briefly discuss a variation on the Hrushovski's pre-dimension construction method as a way to generate $\omega$-categorical structures. The original version of this is used to provide a counterexample to Lachlan's conjecture, where it is used to construct a stable $\omega$-categorical pseudoplane (see Section 5 in \cite{WagnerHRCONS}). Here we follow similar setting used Section 5.2. in
      \cite{EGT}.

      Suppose $\eta =\frac{m}{n}\in (0,1]$ where $gcd(m,n)=1$. Consider the same setting of the previous subsection for $\mathfrak L$ and $\mathcal C_\eta$. Choosing an unbounded convex increasing function $f: \R^{\geq 0} \to \R^{\geq 0}$ which is ``good'' enough one can consider $$\mathcal C^f_\eta:=\{A \in \mathcal C_\eta \,|\, \delta(X) \geq f(\vert X \vert)\,\, \forall X \subseteq A\};$$ where $(\mathcal C^f_\eta,\leqslant_d)$ has the $\leqslant_d$-free amalgamation property and  $\leqslant_d $ is defined as follows:  $A\leqslant_d B$ when $\delta (A'/A)>0$, for each $A\subsetneq A'\subseteq B $.
      In this case we have an associated countable \textit{generic structure} $\M^f_\eta$ which is $\omega$-categorical.
      
      \begin{rem}\label{frem}
      As a good function  we can take some piecewise smooth $f$ where its right derivative $f'$ satisfies $f'(x) \leq 1/x$ and is non-increasing, for $x \geq 1$. The latter condition implies that $f(x+y) \leq f(x) + yf'(x)$ (for $y \geq 0$). It can be shown that under these conditions, $\mathcal C^f_\eta$ has the free $\leqslant_d$-amalgamation property.
      \end{rem}
      We assume that $f$ is a good function.  We will assume that $f(0) = 0$ and $f(1) > 0$, and in this case the $\leqslant$-closure of empty set is empty. We shall also assume that $f(1) = n$ and one can show $\Aut{\M^f_\eta}$ acts  transitively on $M^f_\eta$. See  examples 5.11 and 5.12. in Section 5.2. \cite{EGT} for details.

\section{A relative minimality criterion for $\tau_{st}$}

  \label{sec-relative-min}
  Given a topological group $(G,\tau)$ and $g\in G$ we denote by $\mathcal{N}_{\tau}(g)$ the filter of neighbourhoods of $g$ in $\tau$.
  Since  $\mathcal{N}_{\tau}(g)=g\mathcal N_{\tau}(1)=\mathcal N_{\tau}(1)g$ for any $g\in G$, any group topology $\tau$ is uniquely determined by $\mathcal{N}_{\tau}(1)$. Given a filter $\mathcal V$ on $G$ at $1$ such that
  \begin{itemize}
  \item For every $U\in \mathcal V$ there is $V\in \mathcal V$ such that $ V^{-1}\subseteq U$;
  \item For every $U\in \mathcal V$ there is $V\in \mathcal V$ such that $ VV\subseteq U$;
  \item $U^{g}\in\mathcal{V}$ for every $U\in \mathcal V$ and $g\in G$;
  \end{itemize}
  then there is a unique group topology $\tau$ on $G$ such that $\mathcal{V}=\mathcal{N}_{\tau}(1)$. Given a family $\mathcal{Y}$ of subsets of $G$ containing $1$, we say that $\mathcal{Y}$ generates a group topology $\tau$ at the identity if $\mathcal{Y}$ generates $\mathcal{N}_{\tau}(1)$ as a filter.

  \newcommand{\orq}[0]{\cong^{\scriptscriptstyle{G}}}
  \newcommand{\clg}[0]{\mathrm{acl}^{G}}
  \newcommand{\ax}[0]{(\X)} 
  
  Given a set $X$ we let $[X]^{<\omega}$ stand for the collection of all finite subsets of $X$. Our setting consists of an infinite set $\Omega$ and some $G\leq S(\Omega)$, where $S(\Omega)$ is the group of permutations of $\Omega$.
  It is easy to see using the criterion above that the collection $\{G_{A}\,|\,A\in\fps{\Omega}\}$ is a base of neighbourhoods of the identity of a unique group topology $\tau_{st}$, which we will refer to as the standard topology. We are mainly interested in the case in which $\Omega$ is countable, in which case $S(\Omega)$, abbreviated as $S_\infty$, is a Polish group.
  
  By a closure operator on $\fps{\Omega}$ we mean a map $cl:\fps{\Omega}\to\fps{\Omega}$ that preserves inclusion and satisfies $A\subseteq\cl{A}=\cl{\cl{A}}$, for each $A\in\fps{\Omega}$. There is a bijective correspondence between ($G$-equivariant) closure operators $cl$ and ($G$-invariant) families $\X\subseteq\fps{\Omega}$ closed under intersections. Each $\X$ gives a closure operator $\cl{-}$ by taking as $\cl{A}$ the smallest set in $\X$ containing $A$. In the opposite direction we associate $cl$ with the class of $cl$-closed sets: $\X=\{A\in\fps{\Omega}|\cl{A}=A\}$.
  
  Given tuples $A,B,C$ of elements from $\Omega$ we write $A\orq B$ if there exists some $g\in G$ such that $gA=B$
  and given an additional $C$ we write $A\orq_{C}B$ if there is $g\in G_{C}$ such that $gA=B$.
  Given $A\subset\Omega$ we let $\clg(A)$ stand for the union of all elements of $\Omega$ whose orbit under
  $G_{A}$ is finite. We say $\clg(-)$ is {\it locally finite} if $\clg(A)$ is finite whenever $A$ is. In that case the restriction of  $\clg$ to $\fps{\Omega}$ is a closure operator on $\fps{\Omega}$. We write $\X^{G}=\{A\in[\Omega]^{<\omega}\,|\,\clg(A)=A\}$ and we say that $\clg$ is {\it trivial} if $\X^{G}=\fps{\Omega}$.
  
  Given a family $\X$ of subsets of a set $\Omega$, denote by $\ax$ the collection of all tuples of elements whose coordinates enumerate some member of $\X$. As is customary, the same letter will be used to refer to either a tuple or the corresponding set depending on the context. In particular we might use an expression such as $BC$ to denote the union of the ranges of $B$ and $C$.
  
  Let $G$ be the group of automorphisms of some structure $\M $ with universe $M$.
  Recall that if $\M$ is $\omega$-saturated, then for finite $A$ we have that $\clg(A)$ coincides with the algebraic closure of $A$.
  If $\M$ is $\omega$-saturated and countable, then in particular it is $\omega$-homogeneous, i.e. $A\orq B\Leftrightarrow \typ(A)=\typ(B)$ (alt. $A\equiv B$) for any $A,B\in[M]^{<\omega}$. One says $\M $ is ultra-homogeneous if the stronger equivalence  $A\orq B\Leftrightarrow A\cong B$ holds for any $A,B\in \fps{M}$.
  \newcommand{\XG}[0]{\X^{G}}

  \begin{lem}
  \label{jump lemma} Let $G$ be a group of permutations of a set $\Omega$ for which $\clg(-)$ is  locally finite. Suppose we are given some $G$-invariant $X\subseteq\Omega$ and another group topology $\tau^{*}\subset\tau_{st}^{X}$ such that for some constant $K\in \N$ the following property holds:
  \begin{enumerate}
  \item[$(\diamond)$] \label{intersection} For any 
  $A,B\in\XG$ and $U\in\mathcal{N}_{\tau^{*}}(1)$ there exists $U'\in\mathcal{N}_{\tau^{*}}(1)$ such that $((G_{A}\cap U)G_{B})^{K}=G_{A\cap B}\cap U'$.
  \end{enumerate}
  Then any group topology $\tau\subseteq\tau_{st}^{X}$ must satisfy at least one of the following two conditions:
  \begin{enumerate}
  \item \label{case 1 intersection lemma} Given $x\in X$ there exists $W\in\mathcal{N}_{\tau}(1)$
  such that $gx\in \clg(x)$, for each $g\in W$; or, 
  \item \label{case 2 intersection lemma} There exists some $G$-invariant $X'\subsetneq X$
  such that for all $W\in\mathcal{N}_{\tau}(1)$ there is
  $U'\in\mathcal{N}_{\tau^{*}}(1)$ and $U''\in\mathcal{N}_{\tau_{st}^{X'}}(1)$ such that $U'\cap U''\subseteq W$.
  
  \end{enumerate}
  \end{lem}
  \begin{proof}
  Assume the first alternative does not hold. Then there is $x_{0}\in X$  such that
  for any $W\in \mathcal N_{\tau}(1)$ there exists $g\in W$ such that $g(x_{0})\nin \clg(x_{0})$.
  Let $X'=X\setminus G\cdot x_{0}$. Our goal is to show point \ref{case 2 intersection lemma}, that is, that any neighbourhood $W$ of $1$ in $\tau$ is also a neighbourhood of the identity in any topology containing $\tau^{*}$ and $\tau_{st}^{X'}$.
  
  \begin{obs}
  \label{obs: out of finite}For any $a\in G\cdot x_{0}$, any finite $B\subset\Omega$ and any $W\in\mathcal{N}_{\tau}(1)$ there exists  some $g\in W$ such that $ga\nin B$.
  \end{obs}
  \begin{proof}
  Suppose the condition above fails for some $a$, $B$, and $W$. By Neumann's lemma there exists some $h\in G_{a}$ such that $h(B)\cap B\subseteq \clg(a)$. This means that any $g$ in $W\cap W^{h^{-1}}\in\mathcal{N}_{\tau}(1)$ must take $a$ to a point in $\clg(a)$, a contradiction.
  \end{proof}
  
  The following observation follows from $(\diamond)$ by an induction argument.
  \begin{obs}
  \label{iteration intersection}There is a function $\mu:\N\to\N$ such that given any finite collection $\{B_{j}\}_{j=1}^{r}\subset \X^{G}$, $U\in\mathcal{N}_{\tau^{*}}(1)$ and $W\subseteq G$ containing $U\cap\bigcup_{j=1}^{r}G_{B_{j}}$
  there exists $U'\in\mathcal{N}_{\tau^{*}}(1)$ such that $G_{\bigcap_{j=1}^{r}B_{j}}\cap U'\subseteq W^{\mu(r)}$.
  \end{obs}
  
  Fix some arbitrary $W\in\idn$. Pick $W_{0}=W_{0}^{-1}\in\mathcal{N}_{\tau}(1)$ such that
  $W_{0}^{2K}\subseteq W$.
  Since $\tau\subseteq\tau_{st}^{X}$, there exists some finite $A\subset X$ such that $G_{A}\subseteq W_{0}$. By local finiteness we may assume $A=\clg(A)$.
  Let $\{a_{j}\}_{j=1}^{r}:=A\cap (G\cdot x_{0})$.
  
  Pick $W_{1}=W_{1}^{-1}\in\mathcal{N}_{\tau}(1)$ such that $W^{3\mu(r)}_{1}\subseteq W_{0}$, where $\mu$ is the function given by Observation \ref{iteration intersection}.
  Let $B\subset\Omega$ be a finite subset such that $G_{B}\subset W_{1}$. We may assume again $B\in\XG$. By Observation \ref{obs: out of finite} for any $1\leq j\leq r$ there exists some $g_{j}\in W_{1}$ such that $g_{j}a_{j}\nin B$ or, equivalently, $a_{j}\nin B_{j}:=g_{j}^{-1}B$. Notice that $G_{B_{j}}=G_{B}^{g_{j}}\subseteq W_{1}^{3}$.
  
  Let $C=\bigcap_{j=1}^{r}B_{j}$. According to \ref{iteration intersection} (for $U=G$) there is $U'\in\mathcal{N}_{\tau^{*}}(1)$ such that $G_{C}\cap U'\subset (W_{1}^{3})^{\mu(r)}\subseteq W_{0}$. A final direct application of $(\diamond)$ yields some $U''\in\mathcal{N}_{\tau^{*}}(1)$ such that
  \begin{align*}
  U''\cap G_{C\cap A}\subseteq (G_{C}G_{A})^{K}\subseteq W_{0}^{2K}\subseteq W.
  \end{align*}
  By construction $C\cap A\subseteq X'$ so we are done.
  \end{proof}
  Here is another instance of the same idea.
  \begin{lem}
  \label{intersection of standard topologies}Let $G$ be a group of permutations of a set $\Omega$, $\{X_{j}\}_{j\in J}$ some collection of $G$-invariant subsets of $\Omega$ and $Z=\bigcap_{j\in J}X_{j}$. Assume that $\clg(x)=x$ for any $x\in\Omega$  and that there exists $K>0$ such that for any finite
  $A,B\subset\Omega$ we have $(G_{A}G_{B})^{K}=G_{A\cap B}$. Then $\tau_{st}^{Z}=\bigcap_{j\in J}\tau^{X_{j}}_{st}$.
  \end{lem}
  \begin{proof}
  Let $\tau_{0}=\bigcap_{j\in J}\tau^{X_{j}}_{st}$. The inclusion $\tau^{Z}_{st}\subseteq\tau_{0}$ is clear. Take now any $W\in \mathcal{N}_{\tau_{0}}(1)$. Fix $j_{0}\in J$. Since $W\in \tau_{st}^{X_{j_{0}}}$, there exists
  some finite $A\subseteq X_{j_{0}}$ such that $G_{A}\subseteq W$. Let $\{a_{j}\}_{j=1}^{r}:=A\setminus Z$.
  Just as in Observation \ref{iteration intersection} one can show by induction:
  \begin{claim*}
  There exists a function $\mu:\N\to\N$ such that for any finite collection
  $\{B_{l}\}_{l=1}^{r}\subseteq\fps{\Omega}$ and any $V\subseteq G$ containing
  $G_{B_{l}}$ for all $1\leq l\leq r$ we have $G_{\bigcap_{l=1}^{r}B_{l}}\subseteq V^{\mu(r)}$.
  \end{claim*}
  Pick $W_{0}=W_{0}^{-1}\in\mathcal{N}_{\tau_{0}}(1)$ such that $W_{0}^{\mu(r+1)}\subseteq W$. For each $1\leq l\leq r$ choose some $j_{l}\in J$ such that $a_{l}\nin X_{j_{l}}$ and then some finite $B_{l}\subseteq X_{j_{l}}$ such that $G_{B_{l}}\subseteq W_{0}$. The Claim and the choice of $W_{0}$ implies $G_{C}\subseteq W$ where
  $C=A\cap\bigcap_{l=1}^{r}B_{l}$. Since $C\subseteq Z$ we are done.
  \end{proof}
  
  \begin{lem}
  \label{dealing with acl}Let $G$ be the automorphism group of some structure $\M$ endowed with a $G$-invariant locally finite closure operator $\cl{-}$ on $M$ and a group topology $\tau$ coarser than $\tau_{st}$. Assume that the action of $G$ is transitive and there is some $W\in\mathcal{N}_{\tau}(1)$ and $a\in M$ such that $ga\in \cl{a}$, for each $g\in W$.
  Then either $\tau$ is not Hausdorff or $\tau=\tau_{st}$.
  \end{lem}
  
  \begin{proof}
  Notice that by the transitivity of the action of $G$ on $M$ and continuity of the inverse operation for every $a\in M$ there are $U_{a},W_{a}\in \mathcal N_{\tau}(1)$ such that $f(a)\in \cl{a}$ for any $f\in W_{a}$ and $g^{-1}(a)\in\cl{a}$ for any $g\in U_{a}$. For a finite tuple $A$ in $M$ we write $W_{A}=\bigcap_{a\in A}W_{a}$. Given $a,b\in M$, we say that $a\sim b$ if $a\in \cl{b}$ and $b\in \cl{a}$. This is clearly an equivalence relation.
  If we let $W'_{a}=W_{a}\cap\bigcap_{z\in \cl{b}}U_{z}$, then any $f\in W_{a}'$ must preserve the class $[a]\in M/\sim$ set-wise, that is $W'_{a}\subset G_{[a]}$.
  
  For any $V\in \mathcal N_{\tau}(1)$ and any finite $\sim$-closed $A\subset M$ consider the set
  \begin{align*}
  Y_{V}^{A}=\{f:A\to A\,|\,\exists g\in V\,g\restriction_{A}=f,\,\forall a\in A\,\,g([a])=[a]\}.
  \end{align*}
  Notice that this set is finite, and that given $\sim$-closed $A\subset B\subset M$ and $f\in Y_{V}^{B}$ we have $f\restriction_{A}\in Y_{V}^{A}$. Invariance should be clear from the fact that $A$ is $\sim$-closed and the definition of $Y^{A}_{V}$.
  
  \begin{claim}
  \label{cl: alternative}Either $Y_{V}^{A}=\{\id_{A}\}$ for some $V\in \mathcal N_\tau(1)$ and $\sim$-closed $A$ or there exists $f\in G$ such that for all $\sim$-closed $A\subset M$ and all $V\in \mathcal N_{\tau}(1)$ we have $f\restriction_{A}\in Y_{V}^{A}$.
  \end{claim}
  \begin{proof} [of Claim]
  Recall that according to the assumption the closure is locally finite. If the first alternative is not the case, then from Observation \ref{iteration intersection} and  K\"onig's lemma follows that there is a function $f:M\to M$ such that $f\restriction_{A}\in Y_{V}^{A}$ for any $\sim$-closed $A$ and $V\in\mathcal  N_{\tau}(1)$. The fact that
  $f\restriction_{A}$ is a type-preserving bijection of $A$ for any such $A$ implies $f\in G$.
  \end{proof}
  
  If the first possibility in Claim \ref{cl: alternative} holds true, then $G_{A}$ contains $W'_{A}\cap V$ and is thus a neighbourhood of the identity in $\tau$, which implies that  $\tau=\tau_{st}$. We claim that if the second possibility is satisfied the resulting $f\in G\setminus1$ satisfies $f\in\bigcap_{V\in \mathcal  N_{\tau}(1)}V$, so $\tau$ is not Hausdorff.
  Given any $V\in \mathcal N_{\tau}(1)$, the closure in $\tau_{st}$ of any symmetric $W\in \mathcal N_{\tau}(1)\cap\tau_{st}$ satisfying $W^{2}\subset V$ is itself contained in $V$.
  Hence, $\mathcal N_{\tau}(1)$ admits a basis consisting entirely of $\tau_{st}$-closed neighbourhoods of the identity.  It is thus enough to show that $f$ belongs to the closure of $V$ in $\tau_{st}$ for any $V\in \mathcal N_{\tau}(1)$, which is immediate from the definition of $Y_{V}^{A}$.
  \end{proof}

  The following ubiquitous observation is crucial for the application of the results above. We provide a proof for the sake of completeness.
        \begin{lem}\label{lem-zig-used}
        Let $G$ be a group of permutations of a set $\Omega$ and $A,B$ tuples of elements from $\Omega$ for which there is a chain $A=A_0, B_0, \dots, B_{n-1},A_n=g(A)$ such that $A_{i}B_{i}\orq A_{i+1}B_{i}\orq AB$ for $0\leq i<n$. Then $g \in (G_{A}G_{B})^{n}G_{A}$.
        \end{lem}
        \begin{proof}
        The proof is by induction on $n$. In the base case $n=0$ we have $A=g(A)$ that is $g\in G_{A}$.
        Assume now $n>0$. Since $AB_{0}\orq AB$, there exists $h\in G_{A}$ such that $h(B_{0})=B$.
        Now $A_{1}B_{0}\orq AB$ implies $h(A_{1})B=h(A_{1})h(B_{0})\orq A_{1}B_{0}\orq AB$, which implies
        that there exits $h'\in G_{B}$ such that
        $h'(h(A_{1}))=A$. Applying induction to the sequence $(A'_{i},B'_{i})_{i=0}^{n-1}$ given by $A'_{i}=h'h(A_{i+1})$, $B'_{i}=h'h(B_{i+1})$  yields that $h'hg\in (G_{A}G_{B})^{n-1}G_{A}$, from which it follows that $g\in (G_{A}G_{B})^{n}G_{A}$ as desired.
        \end{proof}

        \begin{defn}
        \label{def-zigzag}Suppose we are given a group $G$ of permutations of a set $\Omega$, and
        $\X$ a $G$-invariant family of subsets of $\Omega$ closed under intersection. We say $\X$ has the {\it $n$-zigzag} property (with respect to the action of $G$) if for every $A,B\in \ax$ and any  $A'$ with
        $A\orq_{A\cap B}A'$ there are $A_0,\dots, A_n$ and $B_0,\dots, B_{n-1}$ such that
        \begin{enumerate}
        \item $A_0:=A$ and $A_{n}=A'$;
        \item $A_{i}B_{i}\orq A_{i+1}B_{i}\orq AB$ for $0\leq i\leq n-1$.
        \end{enumerate}
        We will refer to the sequence $A_{0},B_{0},A_{1},\dots, A_{n}$ above as an $(n,B)${\it -zigzag path} from $A$ to $A'$.
        \end{defn}

        \begin{obs}
        \label{zizagobs} Given an $n$-zigzag path as above it is easy to show by induction that if we write $C=A\cap B$ then $C\subseteq A_{i}B_{i}\cong_{C}A_{i+1}B_{i}\cong_{C}AB$ for all $0\leq i\leq n-1$. In particular, $A_{i}\cap B_{i}=A_{i+1}\cap B_{i}=C$.
        \end{obs}

        Notice that for fixed $A,B$ and $n$ the existence of a $(n,B)$-zigzag path from $A$ to $A'$ depends only on
        the orbit of $A'$ under $G_{A}$.
        
        \begin{prop}
        \label{main minimality proposition} Suppose $\M$ is a countable first order structure and $G=\Aut \M$. Assume $\clg(-)$ is locally finite and the corresponding
        $\X^G$ has the $n$-zigzag property for some $n$. Then:
        \begin{enumerate}
        \item \label{min with closure}If the action of $G$ on $M$ is transitive, then $(G,\tau_{st})$ is minimal.
        \item \label{min without closure} If $\clg{(x)}=x$ for any $x\in M$, then any group topology $\tau\subseteq\tau_{st}$ is of the form
        $\tau_{st}^{X}$ for some $G$-invariant $X\subseteq M$.
        \end{enumerate}
        \end{prop}
        \begin{proof}
        Let us show \ref{min with closure} first. Let $\tau$ be a group topology on $G$ coarser than $\tau_{st}$. By Lemma \ref{lem-zig-used} it is possible to apply \ref{jump lemma} with $\tau^{*}=\{\emptyset,G\}$. If the first alternative of \ref{jump lemma} holds, then by Lemma \ref{dealing with acl} either $\tau$ is not Hausdorff or $\tau=\tau_{st}$. Since by assumption the only invariant subsets of $M$ are $\emptyset$ and $M$, the second alternative
        implies that $\tau=\{\emptyset,G\}$.
        
        Let us now show \ref{min without closure}. Let $\tau$ be a group topology on $G$ coarser than $\tau_{st}$. By Lemma \ref{intersection of standard topologies} there exists some unique minimal $G$-invariant set $X$ such that $\tau\subseteq\tau_{st}^{X}$.
        Apply Lemma \ref{jump lemma} with $\tau^{*}=\{\emptyset,G\}$. The second alternative produces some
        $G$-invariant $X'\subsetneq X$ such that $\tau\subseteq\tau_{st}^{X'}$, in contradiction with the choice of $X$. Since we assume $\clg$ to be trivial, the first alternative implies $\tau=\tau_{st}^{X}$.
        \end{proof}

\section{Minimality and independence}
  \label{sec-ind-free-1based}
  \subsection{Independence}
    Throughout this section we work in the following setting: $\Omega$ is a set, $G$ is a permutation group of $\Omega$, $\cl{-}$ a $G$-equivariant closure operator on $\fps{\Omega}$ and $\X=\{\cl{A}\,|\,A\in\fps{\Omega}\}$ the associated family of closed sets. Our goal is to derive concrete applications from the results of the previous section to the case where $\Omega$ is the underlying set of a first order structure $\M$ and $G =\Aut{\M}$.

    \begin{defn}
    
    Given $\cl{-}$ and $\X$ as above and a ternary relation $\ind$ between members of $\fps{\Omega}$ we say that $(cl,\ind)$ (alternatively, $(\X,\ind)$)
    is a {\it compatible pair} if for all $A,B,C,D\in\fps{\Omega}$ the following properties are satisfied:
    \begin{itemize}
    \item (compatibility) $A\ind_{C}B$ if and only if $A\ind_{\cl{C}}B$ if and only if $\cl{AC}\ind_{C}\cl{BC}$.
    
    \item (invariance) If $g\in G$ and $A\ind_B C$ then $gA\ind_{gB} gC$.
    \item (weak monotonicity) If $A\ind_{B} CD$ or $AD\ind_{B}C$ then $A\ind_{B} C$.
    \item (anti-reflexivity) If $A\ind_{C}B$, then $A\cap B\subseteq \cl{C}$.
    \end{itemize}
    
    We write $A\ind B$ as an abbreviation of $A\ind_{\emptyset}B$.
    \end{defn}

    \begin{defn}
    \label{properties of independence} We define some additional properties for a compatible pair $(\X,\ind)$:
    \begin{itemize}
    \item (transitivity) If $A\ind_B C$ and $A\ind_{BC} D$, then $A\ind_B CD$.
    \item (symmetry) If $A\ind_B C$ then $C\ind_B A$.
    \item (existence) For any $A,B,C$ there is $g\in G_B$ such that $gA\ind_B C$.
    \item (independence) Suppose we are given $A,B_{1},B_{2},C_{1},C_{2}\in\ax$ such that $B_{1}\ind_{A}B_{2}$, $A\subseteq B_{i}$ and $C_{i}\ind_{A}B_{i}$ for $i=1,2$ and $C_{1}\orq_{A}C_{2}$. Then there exists $D\in\X$ such that
    $D\orq_{B_{i}} C_{i}$ for $i=1,2$ and $D\ind_{A}B_{1}B_{2}$.
    \item (stationarity) If $B\in\X$ and $A_i\ind_B C$ for $i=1,2$, then $A_{1}\orq_{B} A_{2}$ implies $A_{1}\orq_{BC} A_{2}$.
    
    \end{itemize}
    Additionally, we consider
    \begin{itemize}
    \item (freedom) $\X=\fps{\Omega}$ and moreover if $A\ind_CB$ and $C\cap AB\subseteq D\subseteq C$, then $A\ind_D B$.
    \item (one-basedness) $A\ind_{A\cap B} B$ for every $A,B\in \X$.
    \end{itemize}
    \end{defn}
    
    The one-basedness property admits the following generalization:
    \begin{defn}
    Given $k\geq 1$, we say that $(\X,\ind)$ satisfies {\it $k$-narrowness} if for any $A\in\fps{\Omega}$ and any $C,A_{0},A_{1},\dots, A_{k}$ in $\X$ the conditions
    \begin{itemize}
    \item $A_{i}\cap A_{i+1}=C$, for each $0\leq i\leq k-1$;
    \item  $A_{i+1}\ind_{A_{i}}A_{i-1}\dots A_{0}$, for each $1\leq i\leq k-1$;
    \end{itemize}
    imply that $A_{0}\ind_{C}A_{k}$ (notice that for $k=1$ we recover the one basedness property).
    \end{defn}

    \begin{lem}
    \label{getting zigzag} Let   $(\X,\ind)$ be a compatible pair that satisfies existence. Then
    \begin{enumerate}
    \item \label{free part} If it satisfies freedom or one-basedness, then for any $A,B\in\mathfrak{X}$  there is $A'\in\mathfrak{X}$ such that $A'\orq_{B}A$, $A'\cap A=A\cap B$ and $A\ind_{A\cap B}A'$;
    \item \label{narrowness part} If it satisfies transitivity, symmetry and $2m$-narrowness, then for any $A,B\in\mathfrak{X}$ there is $A'\in\mathfrak{X}$ such that an $(m,B)$-zigzag path from $A$ to $A'$ exists, $A'\cap A=A'\cap B$ and $A\ind_{A\cap B}A'$.
    \end{enumerate}
    \end{lem}
    
    \begin{proof}
    Existence yields $A'\in\X$ such that $A'\orq_{B}A$ and $A'\ind_{B}A$. Anti-reflexivity implies that
    $A'\cap A\subseteq B$, i.e. $A'\cap A\subseteq A\cap B$. On the other hand $A'\orq_{B}A$ implies
    $A\cap B= A'\cap B$.
    
    If we assume the freedom axiom, then $A'\ind_{A\cap B}A$ follows from $A'\ind_{B}A$ and $B\cap (A'\cup A)=(B\cap A')\cup(B\cap A)=B\cap A$. Alternatively, the same conclusion follows directly from one-basedness.
    
    Let $C=A\cap B$. For \ref{narrowness part} construct sequences $B_{0}=B,B_{1},\dots B_{m-1}$ and $A_{0}=A,A_{1},\dots A_{m}$ as follows.
    Assuming we have already taken $(A_{i},B_{i})_{i=0}^{k}$, existence provides
    $A_{k+1}\cong_{B_{k}}A_{k}$ with $A_{k+1}\ind_{B_{k}}A_{0}B_{0}\dots A_{k}B_{k}$. If $k\leq m$ we we can use the same argument to choose $B_{k+1}\cong_{A_{k+1}}B_{k}$ with $B_{k+1}\ind_{A_{k}}A_{0}B_{0}\dots A_{k+1}$. It is clear that this yields an $(m,B)$-zigzag path from $A$ to $A_{m}$.
    
    By transitivity we have $A_{j}\ind_{B_{j-1}}A_{l}$ for any $0\leq l\leq j-1$ so that $A_{j}\cap A_{l}\subseteq A_{j}\cap B_{j-1}$. Since $A_{j}\cap B_{j-1}=C$ and $C\subset A_{j}\cap A_{l}$ by \ref{zizagobs} we conclude that $A_{j}\cap A_{l}=C$. Arguing in a similar manner one can show that $A_{j}\cap B_{l}=C$ for any $0\leq j\leq m$ and $0\leq l\leq m-1$. This establishes that the sequence $A_{0},B_{0}\dots B_{m-1}A_{m}$ satisfies the first property of the condition in the definition of $2m$-narrowness, while the second follows by transitivity and construction. If we let $A'=A_{m}$ we then get $A'\ind_{C}A$ and $A\ind_{C}A'$ by symmetry, while the sequence above is an $(m,B)$-zigzag path from $A$ to $A'$.
    \end{proof}
    
    \begin{lem}
    \label{reachability of independence}Let $(\X,\ind)$ be a compatible pair satisfying symmetry existence and transitivity and assume that for any $A,B\in\X$ there exists an $(m,B)$-zigzag path from $A$ to some $A_{1}$
    such that $A_{1}\ind_{A\cap B}A$. Then
    \begin{enumerate}
    \item \label{sym-sta} If stationarity holds, then $\X$ has the $2m$-zigzag property;
    \item \label{indep} If independence holds, then $\X$ has the $4m$-zigzag property.
    \end{enumerate}
    \end{lem}
    \begin{proof}
    Let $A,A',B\in\X$ with  $A'\orq_{A\cap B}A$. Let $C:=A\cap B$. In both cases using the assumption we start by choosing $A_{1}\in\X$ for which there is an $m$-zigzag path from $A$ to $A_{1}$ and $A_{1}\ind_{C}A$.
    
    Consider case \ref{sym-sta} first.
    By extension there is $A_{2}$ such that $A_{2}\orq_{A}A_{1}$ and $A_{2}\ind_{A}A'A$.
    The first implies the existence of an $(m,B)$-zigzag path from $A$ to $A_{2}$. The second, together with $A_{2}\ind_{C}A$ implies $A_{2}\ind_{C}A'A$ by right transitivity. By weak monotonicity we get
    $A_{2}\ind_{C}A'$ and by symmetry $A\ind_{C}A_{2}$ and $A'\ind_{C}A_{2}$. Stationarity yields
    $A\orq_{A_2} A'$. Thus, there is also an $(m,B')$-zigzag path from $A_2$ to $A'$, where $A'B'\orq AB$ and combining both paths we get a $(2m,B)$-zigzag path from $A$ to $A'$.
    
    We move on to case \ref{indep}. By invariance and existence there is
    $A'_{1}$ such that $A'_{1}A'\orq A_{1}A$ (so that by invariance $A'_{1}\ind_{C}A'$) and $A'_{1}\ind_{A'}A'A_{1}$. Transitivity and monotonicity then imply $A'_{1}\ind_{C}A_{1}$.
    
    Independence applied to the tuple $C,A_{1},A'_{1},A,A'$ in place of the $A,B_{1},B_{2},C_{1},C_{2}$ of the definition implies the existence of some $D$ such that $DA_{1}\orq AA_{1}$ and $DA'_{1}\orq AA_{1}$. This witnesses the existence of a $(4m,B)$-zigzag path from $A$ to $A'$. Notice that symmetry is required in order to get $A'\ind_{C}A'_{1}$.
    \end{proof}

    \begin{thm*}[\ref{freedom theorem}]
    \textfreedomtheorem
    \end{thm*}
    
    \begin{proof}
    If we let $\X=\fps{M}$ where $M$ is the underlying set of $\M$ and $\ind=\ind^{fr}$, then part \ref{free part} of Lemma \ref{getting zigzag} and part \ref{sym-sta} of Lemma \ref{reachability of independence} apply to the pair $(\X,\ind)$. Together, they imply $\X$ has the $2$-zigzag property with respect to the action of $G$. The result then follows from an application of Proposition \ref{main minimality proposition}.
    \end{proof}

    \begin{thm*}[\ref{simple structures theorem}]
    \textsimpletheorem
    \end{thm*}
    \begin{proof}
    As  $cl$ we take the algebraic closure  $\acl$ and $\ind$ the forking independence. We claim part \ref{free part} of Lemma \ref{getting zigzag} and part \ref{indep} of Lemma \ref{reachability of independence} both apply to $(\X,\ind)$.
    
    The pair clearly satisfies invariance, weak monotonicity, transitivity and symmetry. Existence follows from the fact that $M$ is $\omega$-saturated, so it is left to check  one-basedness and independence in sense of Definition \ref{properties of independence}.
    
    Take $A,B\in\X$. The fact that the theory is one-based in the sense of simplicity theory and has elimination of hyperimaginaries implies $A\ind_{\acl^{eq}(A)\cap \acl^{eq}(B)}B$. The relation $A\ind_{A\cap B}B$ follows then from weak elimination of imaginaries.
    
    Lastly, elimination of hyperimaginaries and weak elimination of imaginaries imply that the type of a tuple over a finite real closed set determines its Lascar strong type over that same set. Hence,  Kim and Pillay's independence theorem \cite{kim1998stability} (see also Chapter 2.3 and Theorem 2.3.1 in \cite{kim1997simple})
    translates into abstract independence (amalgamation of types) for  $(\acl,\ind)$.
    \end{proof}
    
    It is known that simple one-based $\omega$-categorical structures have elimination of hyperimaginaries. This follows from the fact that $\omega$-categorical theories are {\it small} and simple one-based theories admit {\it finite coding}. See section 6 and Proposition 6.1.21. in \cite{WagnerSimple} for definitions and details.
    For stable theories the notion of being $k$-ample (for some $k\geq 1$) generalizes the negation of one-basedness.
    See \cite{evans2003ample} for details. In the absence of algebraic closure being not $k$-ample translates into $(\acl,\ind^{f})$ being $k$-narrow where $\ind^{f}$ is the forking independence. From an argument similar to the one in the two theorems above we can deduce the following result:
    \newcommand{\stablestructurestheoremtext}[0]{
    Let $\M$ be a countable $\omega$-saturated stable structure such that $Th(\M)$ has trivial  algebraic closure, weak elimination of imaginaries and is not $k$-ample for some $k\geq 1$. Then any group topology on $G=\Aut{\M}$ coaraser than $\tau_{st}$ is of the form $\tau_{st}^{X}$ for some $G$-invariant $X\subseteq M$.
    }
    \begin{thm}
    \stablestructurestheoremtext
    \end{thm}
  \subsection*{Example: total minimality is not preserved under taking open finite index subgroups}
    
    \renewcommand{\L}[0]{\mathfrak L}
    Consider the relational language $\L_{1}=\left\{E^{(2)}, P^{(1)}\right\}$ and let
    $\mathcal{K}_{1}$ be the class of all finite $\L_{1}$-structures in which $E$ is interpreted as the edge relation of a bipartite graph with with edges only between the domain of the unary predicate $P$ and its complement.
    Consider also the class $\mathcal{K}_{2}$ in the language
    $\L_{2}=\{E^{(2)},F^{(2)}\}$ consisting of all finite $\L$-structures in which $F$ is interpreted as an equivalence relation with at most $2$ classes and $E$ as the edge relation of a bipartite graph with edges only among vertices that belong to distinct $F$-classes.
    
    Let $\M _{i}=\fl{\mathcal{K}_{i}}$ and $G_{i}=\Aut{\M _{i}}$. Clearly $\M _{2}$ is a reduct of $\M _{1}$, so that $G_{1}\lhd G_{2}$ and in fact $[G_{2}:G_{1}]=2$. It is easy to check that $\mathcal{K}_{1}$ has free amalgamation and then by
    Theorem \ref{freedom theorem} there are exactly two group topologies on $G_{1}$ strictly coarser than $\tau_{st}$, namely $\tau_{st}^{P(\M_{1})}$ and $\tau_{ts}^{\neg P(\M_{1})}$.
    Notice that both are Hausdorff, since no automorphism of $\M _{1}$ can fix $P(\M_{1})$ or its complement (given any two points $a,b$, there exists $c$ in $P$ (resp $\neg P$) such that $\typ(c,a)\neq \typ(c,b)$) so $(G_{1},\tau_{st})$ is not minimal.
    
    In this case we have an additional non-Hausdorff
    group topology, $\tau^*=\{\emptyset,G_{1}\}$.
    Apply Lemma \ref{jump lemma} to conclude that any group topology on $G_1$ strictly contained in  $\tau_{st}$ is contained in $\tau^{*}$.
    
    On the other hand, it follows from Theorem \ref{simple structures theorem} that $(G_{2},\tau_{st})$ is minimal.
  \subsection{Simple non-modular Hrushovski structures} In Subsection \ref{sec-Hrush-new} we discussed in more detail some instances of the Hrushovski construction, in particular the  $\omega$-categorical version (see section 5.2. in \cite{EGT} for more). Here is a brief reminder of the setting:
    
    Choosing an unbounded convex function $f$ which is ``good'' enough, one can consider $\mathcal C^f_\eta$, a subclass of $\mathcal C _\eta$, with the free amalgamation property where the limit structure $\M^f_\eta$ is $\omega$-categorical and such that $\Aut{\M^f_\eta}$ acts  transitively on $M^f_\eta$ (underlying set of $\M^f_\eta$).

    It is shown in Lemma 5.7 in \cite{EGT} that there is an independence relation defined for the class of $\leqslant$-closed subsets of $\M^f_\eta$ that satisfy all the properties of part 1. in Lemma \ref{reachability of independence}. Then using Proposition  \ref{main minimality proposition} we conclude the following.
    \begin{cor}
    Let $\M^f_\eta$ be an $\omega$-categorical Hrushovski generic structure such that $G:=\Aut{\M^f_\eta}$ acts transitively on $M^f_\eta$. Then $(G,\tau_{st})$ is a minimal topological group.
    \end{cor}

\section{Generalized Urysohn spaces}
  \label{subsec-setting-Ury}
  We start by recalling some notions from \cite{conant2015model}.
  A \emph{distance magma} $\rr=(R,\leq,\oplus,0)$ is a set $R$ endowed with a linear order $\leq$ and an operation $\oplus$ such that the following axioms are satisfied:
  \begin{itemize}
  \item $\forall s\, \ s\oplus 0=s$;
  \item $\forall s, r\,\  r\leq r\oplus s $;
  \item $\forall s, t\,\ s\oplus t=t\oplus s$;
  \item $\forall s, s', t\ \ s\leq s'\rightarrow s\oplus t\leq s'\oplus t$.
  \end{itemize}
  When referring to a monoid $\rr$, unless anything to the contrary is said, it will be implicit in the notation that $R$ is its underlying set, and so forth.
  We say that $\rr$ is a \emph{distance monoid} if additionally $\oplus$ satisfies associativity, i.e.:
  \begin{itemize}
  \item $\forall r,s,t\,\ r\oplus (s\oplus t)=(r\op s)\oplus t$.
  \end{itemize}

  Given some additively closed subset $S$ of some ordered abelian group $(\Lambda,+,\leq)$ and $b\in\{S_{>0},\infty\}$ the structure $\mathcal{S}_{b}=\{\{r\in S\,|\,0\leq r\leq b\},0,\leq,+_{b}\}$ is a distance monoid, where $x+_{b}y=\min\{b,x+y\}$ for $b\in S$ and $x+_{\infty}y=x+y$.
  We write $\mathcal{S}$ for $\mathcal{S}_{\infty}$ and $\mathcal{Q}_{b}=\mathcal{S}_{b}$ in case $S=\mathbb{Q}^{\geq 0}$.
  We will refer to any distance monoid  $\rr$ of the form $\mathcal{S}_{b}$ as \emph{basic}.
  If additionally $S$ is a subgroup of $\Lambda$
  with no minimal element
  then we will say $\rr$ is \emph{standard}.
  When talking about a standard distance monoid we may use the symbols $+$ (as opposed to $\oplus$) and $-$ to refer to the operations in the ambient group $\Lambda$ without explicitly referencing $\Lambda$.
  Notice that in the case of basic archimedean distance monoids we can always assume  $\Lambda=\R$.
  
  Given $m\in\mathbb{N}$ and $r\in R$ we will write $m\cdot r$ for the $\oplus$ addition of $r$ with itself $m$ times.
  Given two elements $r,s\in R$, we write $r\sim s$ if there exists some positive integer $n$ such that $n\cdot r\geq s$ and $n\cdot s\geq r$. We refer to the $\sim$-class $[r]$ of $r$ as its {\it archimedean class}. A distance monoid with a single archimedean class of non-zero elements will be called \emph{archimedean}. We write $[r]<[t]$ if $r'<t'$ for all $r'\sim r$ and $t'\sim t$.  \\
  
  Fix a distance magma $\rr:=(R,\leq,\oplus)$. An $\rr$-metric space $(X,d)$ consists of a set $X$ together with  a  map $d:X^{2}\to R$ such that for all $x,y,z\in X$:
  \begin{enumerate}
  \item $d(x,y)=0\leftrightarrow x=y$;
  \item $d(x,y)=d(y,x)$;
  \item $d(x,z)\leq d(x,y)\oplus d(y,z)$.
  \end{enumerate}
  Notice that if $(R,\leq)$ is a substructure of $(\R^{\geq 0},\leq)$ and $r\oplus s\leq r+s$ for all $r,s\in R$, then an $\rr$-metric spaces are just a particular class of metric spaces. In particular, this holds for standard archimedean distance monoids.
  
  An isometric embedding of an $\rr$-metric spaces $(X,d)$ into another $(X',d')$ is a map
  $f:(X,d)\to(X,d')$ such that $d'(f(x),f(y))=d(x,y)$, for each $x,y\in X$. A surjective isometric embedding is called an isometry. Given an $\rr$-metric space $(X,d)$, we let $\Isom(X,d)$ stand for the group of isometries from $(X,d)$ to itself. We will use the symbol $\cong$ to denote the existence of an isometry between two tuples in $\rr$-metric spaces.
  
  In the same spirit, given finite tuples $A=(a_{i})_{i=1}^{k}$ and $A'=(a'_{i})_{i=1}^{k}$ inside an $\rr$-metric space we will write $A\cong_{B} A'$ if there is a partial isometry fixing $B$ and sending each $a_{i}$ to $a'_{i}$, that is, if for any $1\leq i\leq k$ and $b\in B$ we have $d(a_{i},b)=d(a_{i}',b)$ and $d(a_{i},a_{j})=d(a'_{i},a'_{j})$ for distinct $i$ and $j$.\\
  
  We will say that an $\rr$-metric space $\U$ is an $\rr$-{\it Urysohn space} if it satisfies:
  \begin{itemize}
  \item [(U)]\label{universality} Any finite $\rr$-metric space embeds in $\U$; and,
  \item [(H)]\label{ultrahomogeneity} Any isometry between finite subspaces of $\U$ extends to an isometry of $\U$.
  \end{itemize}
  The following strengthening of (U) is implied by the conjunction of (U) and (H) and equivalent to it under the assumption that $\U$ is countable.
  \begin{itemize}
  \item [(EP)]\label{amalgamation property} For any finite $\rr$-metric space $B$ and $A\subseteq B$ any isometric embedding
  $h:A\to\U$ extends to some $\bar{h}:B\to\U$.
  \end{itemize}
  
  Assume $(A,d_{A})$ and $(B,d_{B})$ are two finite $\rr$-metric spaces where $C:=A\cap B\neq \emptyset$ and let
  $D$ be the disjoint union of $A$ and $B$ over $C$. Let
  $\bar{d}:D^{2}\to R$ be given as follows:
  \begin{itemize}
  \item $\bar{d}$ restricts to $d_{A}$ and $d_{B}$ on $A\times A$ and $B\times B$; respectively,
  \item $d(a,b)=\min\{d(a,c)\oplus d(c,b)\,|\,c\in C\}$ for any $a\in A\setminus C$ and any $b\in B\setminus C$.
  \end{itemize}
  
  It can be shown that if  $\rr$ is a distance monoid, then the $(D,\bar{d})$ above is itself an $\rr$-metric space, which we will denote as $A\otimes_{C}B$.
  
  This implies that the class $\mathcal{K}$ of all finite $\rr$-metric spaces has the joint embedding and amalgamation properties. See 2.7 in \cite{conant2015model} for a more precise result (here we are only interested in $S=R$).
  Therefore if $\rr$ is countable, then $\mathcal{K}$ determines a unique countable Fra\"iss\'e limit structure $\U_\rr=\fl{\mathcal{K}}$. This is a countable $\rr$-metric space satisfying property (EP) above and thus an $\rr$-Urysohn space (see Theorem 2.7.7 in  \cite{conant2015model}). An object satisfying the two properties above might exist even if $\rr$ is not countable. The classical Urysohn space and Urysohn sphere are examples of this for $\rr=(\R^{\geq 0},0,\leq,+)$ and $\rr=([0,1],0,\leq,+_{1})$ respectively.
  
  Given finite sets $A,B$ of an $\rr$-metric space $(X,d)$ we define
  $\diam {A}:=\max\{d(a,a')\,|\,a,a'\in A\}$ and
  $d(A,B):=\min \{d(a,b)\,|\, a\in A,b\in B\}$.
  
  Given finite subsets $A,B,C$ of $X$ such that $C\subset A\cap B$ we say that $A\ind_{C}B$ if the subspace $A\cup B$ of is isomorphic to $A\otimes_{C}B$. We generalize this notation to the case in which $C$ is not a common subsets of $A$ and $B$ by letting $A\ind_{C}B$ if and only if $AC\ind_{C}BC$.

\section{Isometry groups of archimedean Urysohn spaces}
  
  The goal of this section is to prove Theorem \ref{Urysohn thm 1} of the introduction. We start with three preliminary lemmas in the following general setting: $\rr=(R,0,\leq,\oplus)$ is a distance monoid and $\U$ an $\rr$-Urysohn space.
  
  \begin{lem}\label{equal distance lemma} 
  Suppose $A$ and $B$ are finite subsets of $\U$ and $r\in R$ such that  $\diam{A}\leq r\leq 2\cdot d(A,B)$. Then there is $A'\subseteq \U$ such that $A'\cong_B A$ and $d(a,a')=r$ for all $a\in A$ and $a'\in A'$.
  \end{lem}
  
  \begin{proof}
  Consider the set $D=A'\coprod_{B} A''$ which is the amalgamated union of two copies $A',A''$ of $A$ over $B$.
  We define an $R$-valued distance function on $D$ as follows. On $A'B$ and $A''B$ the distance between two points equals the distance between the corresponding pair in $\U$, while we set  $d(a',a'')=r$ for any $a'\in A'$ and $a''\in A''$. In order to show that the resulting function satisfies the triangle inequality it suffices to check triples $\{u,v,w\}$ with $u\in A'$, $v\in B$ and $w\in A''$. We have $d(u,w)=r\leq 2\cdot d(A,B)\leq d(u,v)\oplus d(v,w)$. Without loss of generality assume $d(u,v)\leq d(v,w)$. In that case  $d(v,w)\leq \diam{A}\oplus d(v,u)\leq r\oplus d(u,v)=d(u,w)\oplus d(u,v)$. Then the result follows from (EP).
  \end{proof}

  \begin{lem}
  \label{new distancing lemma}
  Suppose we are given finite $A,B,C\subseteq \U$ with $A,B\neq\emptyset$ and $d(A,B)\neq 0$. Then for each $n\in \mathbb N$ there is $g\in G_{B}(G_{A}G_{B})^{n}$ such that $d(C,g(A))\geq (2n+1)\cdot d(A,B)$.
  \end{lem}
  \begin{proof}
  Take $A'\cong_{B} A$ with $A'\ind_{B}C$. Construct a sequence $A_{i},B_{i}$, for $i\geq 0$ as follows. We start by taking $A_{0}=A'$ and $B_{0}=B$. For any $0\leq i<n$ let $C_{i}=(CB_{j}A_{j})_{j\leq i}$ and take $B_{i+1}\cong_{A_{i}}B_{i}$ with $B_{i+1}\ind_{A_{i}}C_{i}$.
  Then, take $D_{i}=(CB_{i+1}B_{j}A_{j})_{j\leq i}$ and let $A_{i+1}\cong_{B_{i+1}}A_{i}$ with $A_{i+1}\ind_{B_{i+1}}D_{i}$.
  The independence $A'\ind_{B}C$ implies $d(A',C)\geq d(A,B)$.
  From Lemma \ref{lem-zig-used} we know $A_{n}$ is of the form $g(A)$ for some $g\in G_{B}(G_{A}G_{B})^{n}$.
  Since by construction $A_{i}B_{i+1}\cong A_{i}B_{i}\cong AB$, independence
  implies that for any
  $c\in C_{i}$ we have $d(c,B_{i+1})\geq d(c,A_{i})\oplus d(A,B)$. Similarly, for any $c\in D_{i}$ we have
  $d(c,A_{i+1})\geq d(c,B_{i+1})\oplus d(A,B)$. The result follows by an easy induction argument.
  \end{proof}

  \begin{lem}
  \label{separation lemma}
  Let $A=\bigcup_{i=1}^{k}A_{i}\subset \U$ be a finite set and $r\in R\setminus\{0\}$. Assume $B_i\subset \U$ is a finite set such that $d(A_i,B_i)\geq r$ for all $1\leq i\leq k$. Then there is a finite $C\subset \U$ such that $d(A,C)\geq r$ and $G_C\subset (\bigcup_{i=2}^{k} G_{B_i})^{2k-1}$. More precisely $C$ is the translate of $B_1$ by an element of $G_{B_2}G_{B_3}\dots G_{B_k}$.
  \end{lem}
  \begin{proof}
  The proof for a general $k$ follows by a simple induction argument from case $k=2$, whose proof we now present. Take $C$ such that $C\cong_{A_{1}B_{2}} B_{1}$ and  $C\ind_{A_{1}B_{2}}A_{2}$. Since $C\cong_{A_{1}} B_{1}$ we have $d(A_{1},C)=d(A_{1},B_{1})\geq r$. We claim that $d(C,A_{2})\geq r$. Indeed, take any
  $c\in C$ and $a\in A_{2}$. There exists $e\in A_{1}B_{2}$ such that $d(c,a)=d(c,e)\oplus d(e,a)$. There are two possibilities. If $e\in B_{2}$, then $d(c,a)\geq d(e,a)\geq r$, by the choice of $A_{2}$ and $B_{2}$. If $e\in A_{1}$ then
  $d(c,a)\geq d(c,e)\geq r$.
  \end{proof}
  \vspace{5pt}

  Given an $\rr$-metric space $(X,d)$, a point $x\in X$ and $\epsilon\in R\setminus\{0\}$ let $N_{x}(\epsilon):=\{g\in \Isom{(X,d)}\;|\;d(gx,x)\leq\epsilon\}$. The following claim is easy to check. See Lemmas \ref{is group topology} and \ref{ladder} below for a more detailed explanation.
  \begin{claim}
  \label{pointwise convergence}Suppose a distance monoid $\rr$ has the property that for any $r\in R\setminus\{0\}$ there exists $s\in R\setminus\{0\}$ with $s\oplus s\leq r$.
  Then for any $\rr$-metric space $(X,d)$ the collection $\{N_{x}(\epsilon)\;|\;x\in X,\epsilon\in R\setminus\{0\}\}$ generates a Hausdorff group topology on $G=\Isom(X,d)$ at the identity.
  \end{claim}
  We denote the topology above by $\tau_{m}$. For metric spaces this is just the usual point-wise convergence topology on $G\subseteq X^{X}$. The following theorem generalizes Uspenski's minimality result for the isometry group of the Urysohn sphere.
  
  \begin{thm*}[\ref{Urysohn thm 1}]
  
  \texturysohnthmone
  \end{thm*}
  \begin{proof}
  
  Suppose $\tau_{0}$ does not satisfy the conclusion of the theorem and let $\tau$ be a group topology that is coarser than $\tau_{st}$ but not finer $\tau_{0}$. This implies that there is $s\in R\backslash \{0\}$ such that  $N_v(s)$ is not a $\tau$-neighbourhood of $1$ for any $v\in \U$ in $\tau$.
  \begin{lem} \label{lemma-1-main}
  Given $t\in R$ with $2t\leq s$, a neighbourhood $V$ of $1$ in $\tau$, as well as $a\in \U $, $k\in \mathbb N$ and $b_1,\cdots,b_k\in\U$ there is $g\in V$ such that $ga\notin \Ba_{b_i}(t)$, for each $1\leq i\leq k$, i.e. $d(ga,\{b_{i}\}_{i=1}^{k})>t$.
  \end{lem}
  \begin{proof}
  Assume the conclusion fails. Take $h\in G_a$ such that $h(B)$ and $B$ are independent over $a$ where $B=\{b_1,\dots,b_k\}$ and consider $V\cap V^{h^{-1}}$. Take $g\in V\cap V^{h^{-1}}$. There are $1\leq i,j\leq k$ such that $ga\in\Ba_{b_i}(t)\cap \Ba_{hb_j}(t)$. This implies $d(b_i,hb_j)\leq 2t\leq s$. Independence of $B$ and $h(B)$ over $a$ implies that either $d(b_i,a)\leq t$ or $d(hb_j,a)\leq t$. As $ga\in\Ba_{b_i}(t)\cap \Ba_{hb_j}(t)$, either of the two cases implies $d(ga,a)\leq 2t\leq s$ and hence $V\cap V^{h^{-1}}\subseteq N_a(s)$  which is a contradiction.
  \end{proof}
  
  \begin{lem}
  \label{lemma-3-main} Given $t\in R$ with $2t\leq s$,  $V\in\mathcal{N}_{\tau}(1)$ and $A=\{a_1,a_2,\dots,a_k\}$ a finite subset of $\U$ there is a finite subset $C$ of $\U$ such that $G_C\subset V$ and $d(A,C)> t$.
  \end{lem}
  \begin{proof}
  Consider $W$ be a neighbourhood of $1$ in $\tau$ with $W=W^{-1}$ such that $W^{6k-3}\subset V$. Let $C_0$ be a finite subset of $\U$ such that $G_{C_0}\subseteq W$. By Claim \ref{lemma-1-main} there is $g_i\in W$ such that $d(g_{i}a_i,C_0)> t$ for each $1\leq i\leq k$. Then $d(g_i^{-1}(C_0),a_i)> t$ for each $i$ and by applying Lemma \ref{separation lemma} with $B_i=g_i^{-1}(C_0)$ and $A_{i}=\{a_{i}\}$, we find a finite subset $C$ such that $G_C\subseteq (\bigcup_i G_{B_i})^{2k-1}\subseteq W^{6k-3}\subseteq V$ as $G_{B_i}\subseteq W^3$.
  \end{proof}
  
  \begin{lem}
  \label{lem-final-main}
  For any $V\in\mathcal{N}_{\tau}(1)$, any finite $C\subseteq\U$ and $r\in R\setminus\{0\}$ there is a finite $D$ with $d(D,C)\geq r$ and $G_{D}\subseteq V$.
  \end{lem}
  \begin{proof}
  \newcommand{\ceil}[1]{\lceil #1 \rceil}
  Recall that $s$ is fixed before Lemma \ref{lemma-1-main}. We claim that there exists $t_{0}\in R$ such that $2t_{0}\leq s$ and for any $r\in R$ there is $m\in\N$ such that
  $mt'\geq r$ for any $t'>t_{0}$. Indeed, either there is $t>0$ with $2t\leq s$, in which case we can take $t_{0}=t$, or else $2t'>s$ for all $t'>0$ and we can take $t_{0}=0$.
  
  Fix now $V\in\mathcal{N}_{\tau}(1)$ and $r\in R$ and let $m$ be as above. Let $k=\ceil{\frac{m-1}{2}}$.
  Choose $W\in\mathcal{N}_{\tau}(1)$ with $W=W^{-1}$
  and $W^{2m+5}\subseteq V$.
  
  By Lemma \ref{lemma-3-main} applied to $C$ and $W$ there are $A$ and $B$ with $d(A,B)>t$ such that $G_A,G_{B}\subseteq W$. Lemma \ref{new distancing lemma} then implies the existence of $g\in G_{B}(G_{A}G_{B})^k\subseteq W^{2k+1}\subseteq W^{m+2}$ such that $d(C,g(A))\geq m\cdot d(A,B)\geq r$, which in turn implies $G_{g(A)}=gG_{A}g^{-1}\subseteq W^{2m+5}$. Take $D=A$.
  \end{proof}

  We are now ready to finish the proof of Theorem \ref{Urysohn thm 1}. Pick any neighbourhood $W$ of $1$ in $\tau$ with $W=W^{-1}$ and $g\in G\backslash\{1\}$ where $g\notin W^4$. Since $W$ is a neighbourhood of identity in $\tau$ it must contain $G_{A}$ for some finite subset $A $ of $\U$. Lemma \ref{lem-final-main} implies the existence of some finite $B\subset\U$ such that $G_{B}\subset W$ and
  $ d(A\,g(A),B)\geq \diam{A\, g(A)}$. By Lemma   \ref{equal distance lemma} there is an isomorphic copy $A'$ of $A$ over $B$ and $s\in R$ such that $d(a,a')=s$  for all $a'\in A'$ and $a\in A\, g(A)$.
  
  In particular, $A\cong_{A'}g(A)$, which implies there is $B'$ such that $g(A) B'\cong A'B'$. By Lemma \ref{lem-zig-used} the chain $A,B,A',B',g(A)$ witnesses $g\in (G_AG_B)^2\subseteq W^{4}$, contradicting the choice of $W$.

  \end{proof}

\section{Group topologies on $\Isom(\mathcal{U})$ coarser than $\tau_{st}$}
  
  \newcommand{\cut}[2]{Ct_{#1}(#2)}
  \newcommand{\cutr}[3]{Ct^{#3}_{#1}(#2)}
  \newcommand{\bp}[0]{\boxplus}
  \newcommand{\ali}[2]{Al_{#1}(#2)}
  
  In the light of Theorem \ref{Urysohn thm 1} one might conjecture there is a gap between $\tau_{st}$ and the point-wise convergence topology $\tau_{m}$ in those cases in which the latter exists. This turns out to be false.
  
  Fix a distance monoid $\rr$, an $\rr$-metric space $(X,d)$ and let $G=\Isom(X,d)$. For any distinct $x,y\in X$ write
  \begin{align*}
  \nb_{x,y}:=\{g\in G\,|\,d(gx,y)\leq d(x,y)\}.
  \end{align*}
  The following is easy to check; see Lemmas \ref{is group topology} and \ref{ladder} of the following section.
  
  \begin{claim}
  \label{spherical topology}Suppose that $\rr$ has the property that for any $r\in R\setminus\{0\}$ there exists $s,s'\in R\setminus\{0\}$ with $s\oplus s'=r$.
  Then the collection $\{\nb_{x,y}\;|\;x,y\in X,d(x,y)>0\}$ generates a Hausdorff group topology $\tau_{0^{+},0}$ on $G=\Isom(X,d)$ at the identity. Moreover, if $(X,d)$ is an $\rr$-Urysohn space, then the inclusions $\tau_{m}\subset\tau_{0^{+},0}\subset\tau_{st}$ are proper.
  \end{claim}
  
  Due to the following obstruction the topology $\tau_{0,0^{+}}$ is not eliminated by an application of Lemma \ref{jump lemma} to the pair $(\tau_{m},\tau_{st})$. Take for instance $\U=\U_{\R}$ and two disjoint sets
  $A,B\subset\U_{\R}$ of size $k\geq 1$ that lie entirely on a common line (i.e. any triangle spanned by three points in $AB$ is degenerate) and alternate on said line. Let us say that the ``leftmost'' point is $\alpha\in A$ and the ``rightmost'' point is $\beta\in B$.
  Then for any chain $A_{0}=A,B_{0}=B,A_{1},B_{1},\dots A_{k-1},B_{k-1}$ for which
  $A_{i}B_{i}\cong A_{i+1}B_{i}$ and $A_{i+1}B_{i}\cong A_{i+1}B_{i+1}$ for all $0\leq i\leq k-2$  we have the constraint
  $d(\alpha,\beta')\leq d(\alpha,\beta)$, where $\beta$ and $\beta'$ are components of the same index in $B$ and $B_{k}$ respectively.
  The main content of Theorem \ref{3 topologies theorem} is the existence of a gap between $\tau_{0^{+},0}$ and $\tau_{m}$ (Proposition \ref{second gap}), which involves a series of small technical intermediate Lemmas collected in subsections \ref{subsection alignment} and \ref{subsection downwards}. In contrast, the existence of a gap between $\tau_{st}$ and $\tau_{0^{+},0}$  (Lemma \ref{first gap}) is a direct consequence of Lemma \ref{jump lemma}.

  The Lemmas in Subsection \ref{subsection alignment} highlight different aspects of the obstruction mentioned above. In particular, Lemma \ref{small disturbance} can be read as saying that this is in fact the only obstruction for the assumptions of  Lemma \ref{jump lemma} to hold for the pair $(\tau_{m},\tau_{st})$.
  
  Subsection \ref{subsection downwards} gathers Lemmas allowing one to move downwards: we are given a group topology
  $\tau$ and we know that there exists $W\in\idn$ such that all $g\in W$ preserves a certain property and we want to replace it with $W'\in\idn$ such that all $g\in W'$ preserve some different (stronger) property.

  \subsection{Point alignment}
    \label{subsection alignment}
    
    For future reference we state Lemma \ref{small disturbance} and other lemmas in this section in greater generality than required by Theorem \ref{3 topologies theorem}. The reader might as well take $\rr=\mathcal{S}_{b}$ where $S$ is some dense subgroup of $\RR$ and $b\in S_{>0}\cup\{\infty\}$.
    In the definitions below $\rr=(R,0,\leq,\oplus)$ is a distance monoid and $(X,d)$ an $\rr$-metric space.
    
    Given $\epsilon\in R$ we say that an unordered triple $r_{1},r_{2},r_{3}\in R$ is \emph{$\epsilon$-flexible} if
    $r_{i}\oplus\epsilon\leq r_{j}\oplus r_{k}$ where $r_{i}=\max\{r_{1},r_{2},r_{3}\}$ and $\{j,k\}=\{1,2,3\}\setminus\{i\}$. We say that it is \emph{strongly $\epsilon$-flexible} if moreover for any $r'_{j}$ and $r'_{k}$ such that
    $r'_{j}\oplus\epsilon\geq r_{j}$ and $r'_{k}\oplus\epsilon\geq r_{k}$ we have
    $r_{i}\leq \min\{r'_{j}\oplus r_{k},r_{j}\oplus r'_{k}\}$.

    A $0$-flexible triple will be called simply {\it triangular}. We say that a triangular triple is \emph{(strongly) flexible } if it is (strongly) $\epsilon$-flexible for some $\epsilon\in R\setminus\{0\}$.

    We say that an unordered triple of points $u,v,w$ in $X$ is ($\epsilon$-)\emph{flexible} if $d(u,v)$, $d(u,w)$, $d(v,w)$  is ($\epsilon$-)flexible. We say it is \emph{tight} if it is not $\epsilon$-flexible for any $\epsilon>0$·  We say that an ordered triple of points $(u_{1},u_{2},u_{3})\in\U^{3}$ in some $\rr$-metric space is \emph{aligned} if it is tight as an unordered triple and $d(u_{1},u_{3})=\max\, \{d(u_{i},u_{j})\,|\,{i,j\in\{1,2,3\}}\}$. Given $u,v\in\U$ we let $[u,v]$ stand for the collection of points $x$ such that $(u,x,v)$ is aligned.
    
    We say that a set of three distinct points $\{u_{i}\}_{i=1}^{3}$ is \emph{in} ($\epsilon$-)\emph{general position} if the triple $d(u_{1},u_{2})$,$d(u_{2},u_{3})$,$d(u_{2},u_{3})$ is strongly ($\epsilon$-)flexible.

    \begin{obs}
    \label{independence obs} If $\rr$ is a basic distance monoid, $\U$ an $\rr$-Urysohn space and $u,v,w\in \U$ such that $u\ind_{v}w$ and in the ambient group $d(u,v)+d(v,w)<\sup\, R\in R\cup\{\infty\}$ holds, then $(u,v,w)$ is aligned. If $\rr$ is standard then the opposite implication is also true: the triple of points $(u,v,w)$ is aligned only if $u\ind_{v}w$.
    \end{obs}

    Given two finite subsets $A,B\subset X$ and $r\in R^{*}$ we say that $B$ \emph{$r$-cuts} $A$ if there exists $a,a'\in A$ with $d(a,a')\leq r$ such that $B\cap[a,a']\neq\emptyset$. We say that $B$ \emph{cuts} $A$ if it $r$-cuts $A$ for some $r\in R$.
    We say that $B$ is {\it in ($\epsilon$-)general position} relative to $A$ if for any distinct $a_{1},a_{2}\in A$ and any $b\in B$ the triple $a_{1},a_{2},b$ is in ($\epsilon$-)general position. Notice that in particular this implies that $B$ does not cut $A$.

    The following Lemma is the main source of motivation of the definitions above.
    \begin{lem}
    \label{small disturbance}Let $\rr$ be any distance monoid, $\U$ an $\rr$-Urysohn space, $G=\Isom(\U)$ and $A,B$ finite subsets of $\U$ such that $B$ is in $\epsilon$-general position relative to $A$. Then $G_{A}G_{B}G_{A}\supset\bigcap_{a\in A}N_{a}(\epsilon)$. In particular, $G_{A}G_{B}G_{A}\in\mathcal{N}_{\tau_{m}}(1)$ in case $\tau_{m}$ exists (see Claim \ref{pointwise convergence}).
    \end{lem}
    \begin{proof}
    \label{metrical neighbourhood}
    By virtue of Lemma \ref{lem-zig-used} the result can be rephrased as follows.
    Given any finite metric space $(D,\bar{d})$ whose underlying set consists of $A$ and an isometric copy of $A'$ with the property that $\bar{d}(a',a)\leq\epsilon$ for conjugate points $a\in A,a'\in A'$ the extension of $\bar{d}$ to $(D\coprod B)^{2}$ given by $\bar{d}(a,b)=\bar{d}(a',b)=d(a,b)$ for $a, a'\in A$ and $b\in B$ defines an $\rr$-metric space. It suffices to check the triangular inequality for triples of points of the form $(a_{1},a_{2}',b)$. 
    On the one hand for any triples of points $a_{1},a_{2},b$ where $a_{1},a_{2}\in A$ and $b\in B$  we have:
    \begin{align*}
    \bar{d}(a_{1},a'_{2})\leq \epsilon\oplus d(a_{1},a_{2})\leq 	d(a_{1},b)\oplus d(b,a_{2})=\bar{d}(a_{1},b)\oplus\bar{d}(b,a_{2}),
    \end{align*}
    where the second inequality comes from $\epsilon$-flexibility of $\{a_{1},a_{2},b\}$.
    On the other hand $\bar{d}(a_{1},a_{2}')\oplus\epsilon\geq d(a_{1},a_{2})$ so strong $\epsilon$-flexibility of $\{a_{1},a_{2},b\}$
    yields:
    \begin{align*}
    \bar{d}(a_{1},b)=d(a_{1},b)\leq \bar{d}(a_{1},a'_{2})\oplus d(a_{2},b)=\bar{d}(a_{1},a'_{2})\oplus\bar{d}(a'_{2},b).
    \end{align*}
    \end{proof}

    We say that $\rr$ has \emph{no gaps} if for any $r<s$
    there exists $\epsilon\in R\setminus\{0\}$ such that $r\oplus\epsilon\leq s$.
    
    \begin{lem}
    \label{independence cutting}Let $\rr$ be a distance monoid with no gaps and $A,B,C$ finite subsets of some $\rr$-metric space $(X,d)$ satisfying $A\ind_{B}C$, where $B=B_{1}\cup B_{2}$. If $A$ $r$-cuts $C$, then at least one of the following holds:
    \begin{itemize}
    \item $B_{1}$ $r$-cuts $C$;
    \item $A$ $r$-cuts $B_{2}$.
    \end{itemize}
    \end{lem}
    \begin{proof}
    Pick $a\in A$ and $c_1,c_2$ such that $d(c_{1},c_{2})\leq r$ and $a\in[c_{1},c_{2}]$. For $i=1,2$ there exists $b_{i}\in B$ such that $d(a,c_{i})=d(a,b_{i})\op d(b_{i},c_i)$. We claim  $b_1,b_2\in [c_1,c_2]$. Otherwise for some $\epsilon>0$ and $i\in\{1,2\}$ we have:
    \begin{align*}
    d(c_{1},c_{2})\op\epsilon
    \leq d(c_{i},b_{i})\op d(b_{i},c_{3-i})\leq
    d(c_{i},b_{i})\op(d(b_{i},a)\op d(a,c_{3-i}))=\\
    =(d(c_{i},b_{i})\op d(b_{i},a))\op d(a,c_{3-i}) =
    d(c_{1},a)\op d(a,c_{2}),
    \end{align*}
    contradicting $a\in [c_{1},c_{2}]$. We claim that $a\in [b_{1},b_{2}]$
    and $d(b_{1},b_{2})\leq d(c_{1},c_{2})$. On the one hand,
    if $a\nin [b_{1},b_{2}]$, then for some $\epsilon>0$ we have:
    \begin{align*}
    d(c_{1},c_{2})\op\epsilon\leq d(c_{1},b_{1})\op (d(b_{1},b_{2})\op\epsilon)\op d(b_{2},c_{2})\leq\hspace{3 cm}\\
    \leq d(c_{1},b_{1})\op (d(b_{1},a)\op d(a,b_{2}))\op d(b_{2},c_{2})=d(c_{1},a)\op d(a,c_{2}),
    \end{align*}
    contradicting $a\in [c_{1},c_{2}]$. On the other hand $d(b_{1},b_{2})\leq d(c_{1},c_{2})$, since otherwise
    $$d(c_{1},c_{2})\op\epsilon\leq d(b_{1},b_{2})=d(b_{1},a)\oplus d(a,b_{2})\leq d(c_{1},a)\op d(a,c_{2})$$
    for some $\epsilon>0$, since $\rr$ has no gaps.
    So in case $\{b_{1},b_{2}\}\subseteq B_{2}$, then the second alternative in the statement holds, while if $\{b_{1},b_{2}\}\cap B_{1}\neq\emptyset$, then the first one must hold.
    \end{proof}

    \begin{cor}
    \label{non-cutting by parts}Let $\rr$ be a distance monoid with no gaps and $\U$ an $\rr$-Urysohn space.
    Let $r\in R$ and $A_{j},B_{j},\,\,1\leq j\leq k$ be finite subsets of $\U$ such that $A_{j}$ does not $r$-cut $B_{j}$ for any $1\leq j\leq k$. Then there exists $g\in G_{B_{k}}G_{B_{k-1}}\cdots G_{B_{2}}$ such that $A:=\bigcup_{1\leq j\leq k }A_{j}$ does not $r$-cut $g B_{1}$.
    \end{cor}
    \begin{proof}
    The argument is analogous to the one in the proof of Claim \ref{separation lemma}. We will restrict to the case $k=2$, since the general case can be deduced from it by an easy induction argument.
    Take $B'_{1}\cong_{A_{1}B_{2}}B_{1}$ such that $B'_{1}\ind_{A_{1}B_{2}}A_{2}$.
    If $A_2$ cuts $B'_{1}$ then by Lemma \ref{independence cutting} we have
    $A_{1}$ cuts $B'_{1}$ or $A_{2}$ cuts $B_{2}$. Both alternatives are ruled out by our assumptions on $A_{i},B_{i}$ and the fact that $B'_{1}\cong_{A_{1}}B_{1}$.
    \end{proof}
    
    \begin{lem}
    \label{not tight} Let $\rr$ be a standard distance monoid and suppose we are given a triangular triple $r_{1},r_{2},r_{3}\in R\setminus\{0\}$ with $r_{1}\leq r_{2}\leq r_{3}$ as well as $\epsilon_{1},\epsilon_{2},\epsilon_{3},\delta\in R$ such that:
    \begin{itemize}
    \item $2\cdot\epsilon_{i}\leq r_{1}$ for every $1\leq i\leq 3$;
    \item $\epsilon_{3}\leq\max\{\epsilon_{1},\epsilon_{2}\}$;
    \item  $\delta\leq \max\{\epsilon_{1},\epsilon_{2}\}-\epsilon_3$.
    \end{itemize}
    Then the triple $r_{i}\oplus\epsilon_{i},\,\,1\leq i\leq 3$ is strongly $\delta$-flexible.
    \end{lem}
    \begin{proof}
    On the one hand:
    \begin{align*}
    r_{3}\op\epsilon_{3}\oplus\delta\leq r_{1}\oplus r_{2}\oplus\epsilon_{3}\op\delta\leq
    r_{1}\oplus r_{2}\oplus\epsilon_{1}\oplus\epsilon_{2}=
    (r_{1}\oplus\epsilon_{1})\oplus (r_{2}\oplus\epsilon_{2}),
    \end{align*}
    where use the fact that $\epsilon_{3}\oplus\delta\leq\epsilon_{1}\oplus\epsilon_{2}$.
    On the other hand, for $i\in\{1,2\}$ we have:
    \begin{align*}
    r_{i}\op\epsilon_{i}\op\delta\leq r_{3}\op 2\max\{\epsilon_{1},\epsilon_{2}\}\leq r_{3}\oplus r_{3-i} \leq (r_{3}\oplus\epsilon_{3})\op (r_{3-i}\op\epsilon_{3-i}),
    \end{align*}
    This shows that the triple is $\delta$-flexible. Moreover,
    for $i\in\{1,2\}$ we have
    \begin{align*}
    (r_{i}+\epsilon_{i}-\delta)+ (r_{3-i}+\epsilon_{3-i})\geq r_{1}+r_{2}+\max\{\epsilon_{1},\epsilon_{2}\}-\delta\geq r_{3}+\epsilon_{3}
    \end{align*}
    which implies that $r'_{i}\oplus(r_{3-i}\oplus\epsilon_{3-i})\geq r_{3}\oplus\epsilon_{3}$ for any $r'_{i}$ for which $r_{i}\leq r'_{i}\oplus \delta$. Likewise for $i\in\{1,2\}$ we have
    \begin{align*}
    (r_{3}-\delta)+(r_{i}\oplus\epsilon_{i})\geq r_{3}-\delta+r_{i}\geq r_{3}-\delta+2\max\{\epsilon_{1},\epsilon_{2}\}\geq\\
    \geq r_{3}-\delta+(\delta+\max\{\epsilon_{1},\epsilon_{2}\})\geq r_{3-i}+\epsilon_{3-i}
    \end{align*}
    so the tuple $r_{1}\oplus\epsilon_{1},r_{2}\oplus\epsilon_{2},r_{3}\oplus\epsilon_{3}$
    is strongly $\delta$-flexible.

    \end{proof}

    \begin{lem}
    \label{misalignment}Let $\rr$ be a standard distance monoid with no gaps, no minimal positive element and such that for any $r\in R$ there exists
    $s\in R$ such that $s\oplus s\leq r$.
    Let $\U$ an $\rr$-Urysohn space and $A,B\subset\U$ finite sets such that $B$ does not cut $A$. Then there exist
    $A'\cong_{B}A$ such that $A'$ is in general position relative to $A$.
    \end{lem}
    \begin{proof}
    Fix $\epsilon\in R\setminus\{0\}$ such that $d(a,a')\oplus\epsilon\leq d(a,b)\oplus d(b,a')$ for all $a,a'\in A$ and $b\in B$ (we include the case $a=a'$).
    Fix some symmetric injective function  $f:A \times A\to (0,\epsilon)\subset R$ such that:
    \begin{itemize}
    \item $2\max(im(f))\leq\min\{d(a,a')\,|\,a,a'\in A,a\neq a'\}$;
    \item $d(a_{1},a_{2})<d(a'_{1},a'_{2})$ implies $f(a_{1},a_{2})>f(a'_{1},a'_{2})$ for $a_{1},a_{2},a'_{1},a'_{2}\in A$.
    \end{itemize}
    \begin{claim}
    \label{auxiliary existence} There exists $h\in G_{B}$ such that $d(a,ha')=d(a,a')\oplus f(a,a')$, for all $a,a'\in A$.
    \end{claim}
    
    Let us first show how to prove the Lemma using the claim above. We need to show that for any $a_{1},a_{2},a_{3}\in A$, $a_{1}\neq a_{2}$, the triple $a_{1},a_{2},ha_{3}$ is strongly $\delta$-flexible, where $\delta$ is the minimum between the minimum value of $f$ and the smallest absolute difference between two values in the image of $f$.
    
    If $a_{3}=a_{i}$ for some $i=1,2$, let's say $a_{3}=a_{1}$, then strong
    $\delta$-flexibility of $\{a_{1},a_{2},ha_{3}\}$ follows from  $f(a_1,a_2)+\delta\leq  f(a_{1},a_{1})=d(a_{1},ha_{3})$, $d(a_2,ha_1)=d(a_1,a_2)\oplus f(a_1,a_2)$ and $f(a_{1},a_{1})\leq d(a_{1},a_{2})$.
    

    If $a_{1},a_{2},a_{3}$ are all different we apply Lemma \ref{not tight} to the three distances between $a_{1},a_{2},a_{3}$, which we name in increasing order $r_{1}\leq r_{2}\leq r_{3}$.
    Here we take $\epsilon_{i}=0$ exactly for one value of $i$ for which $r_{i}:=d(a_{1},a_{2})$, while for
    $i=1,2$ if $r_{j}=d(a_{i},a_{3})$, we take $\epsilon_{j}=f(a_{i},a_{3})$.

    The first condition in Lemma \ref{not tight} follows immediately from the first property of $f$.
    On the other hand, our second condition on $f$ guarantees that whichever element in $\{\epsilon_{1},\epsilon_{2}\}$ is non-zero must be also larger than $\epsilon_{3}$ so the second assumption of Lemma \ref{not tight} also holds. The third one follows from the choice of $\delta$.

    It follows from \ref{not tight} for any $a_{1},a_{2},a_{3}\in A$, $a_{1}\neq a_{2}$, the triple $a_{1},a_{2},ha_{3}$ is strongly flexible. This in turn implies that $hA$ is in general position relative to $A$.
    
    \newcommand{\bd}[0]{\bar{d}}
    Let us now prove the Claim. For $i=1,2$ let $A^{i}=\{a^{i}\,|\,a\in A\}$ be a copy of $A$ and $D:=A^{1}\coprod A^{2}\coprod B$.
    It suffices to check that the function $\bd:D^{2}\to R$ given by:
    \begin{itemize}
    \item $\bd(a_{1}^{i},a_{2}^{i})=d(a_{1},a_{2})$ for $i=1,2$ and $a_{1},a_{2}\in A$;
    \item  $\bd(a^{i},b)=d(a,b)$ for $i=1,2$, $a\in A$ and $b\in B$;
    \item $\bd(a^{1}_{1},a^{2}_{2})=d(a_{1},a_{2})\op f(a_{1},a_{2})$ for any $a_{1},a_{2}\in A$;
    \end{itemize}
    satisfies the triangle inequality. For triples of points
    contained in $A^{1}\cup A^{2}$ this is part of the conclusion of Lemma \ref{not tight}. All that is left to check is the triangle inequality for triples of the form $(a^{1}_{1},a^{2}_{2},b)$, where
    $a_{1},a_{2}\in A$ and $b\in B$.
    We have $\bd(a^{1}_{1},a^{2}_{2})\leq d(a_{1},a_{2})\oplus\epsilon\leq  d(a_{1},b)\op d(b,a_{2})=\bd(a_{1}^{1},b)\op\bd(b,a_{2}^{2})$ by the choice of $\epsilon$ and the fact that $im(f)\subseteq(0,\epsilon)$, while the remaining inequalities are straightforward from the inequality $d(a_{1},a_{2})\leq \bd(a_{1}^{1},a^{2}_{2})$.
    \end{proof}

    \begin{lem}
    \label{better non-cutting constant} Let $\rr$ be a standard distance monoid and $\U$ an $\rr$-Urysohn space. Assume we are given $r\in R$ and finite $A,B\subset\U$ such that $B$ does not $r$-cut $A$. Take $A'$ such that $AB\cong A'B$ and $A'\ind_{B}A$. Then $A'$ does not $(r\oplus r)$-cut $A$.
    \end{lem}
    \begin{proof}
    As usual, let $a'$ stand for the conjugate of any given $a\in A$ in $A'$. Aiming for contradiction,
    suppose that there exist $a_{1},a_{2},a_{3}\in A$
    such that $d(a_{1},a_{2})\leq r\oplus r$ and $a'_{3}\in[a_{1},a_{2}]$.
    This implies, in particular, that $d(a_{1},a_{2})<\sup R$.
    
    Notice that by independence for all $a,\tilde{a}\in A$ we have $d(a,\tilde{a}')\geq d(a,\tilde{a})$, with equality only if
    $B\cap[a,\tilde{a}]\neq\emptyset$ or $d(a,\tilde{a})=\sup R$ (in the bounded case).

    Since $d(a_{1},a_{2})= d(a_{1},a'_{3})+d(a'_{3},a_{2})$, as $d(a_{1},a_{2})<sup\,R$,
    the triangle inequality implies $d(a_{j},a'_{3})=d(a_{j},a_{3})$
    for $j=1,2$. By the previous paragraph the intersections
    $B\cap [a_{1},a_{3}]$ and $B\cap [a_{2},a_{3}]$ are both non-empty.
    However, $d(a_{i},a_{3})\leq r$ for at least one $i\in\{1,2\}$: a contradiction.
    \end{proof}

  \subsection{Downward lemmas}\label{subsection downwards}
    
    Throughout this subsection $\rr$ will be a standard distance monoid, $\U$ an $\rr$-Urysohn space and
    $G=\Isom(\U)$.

    \begin{lem}
    \label{lambda lemma}For any finite $a,c\in\U$ and finite $B\subseteq\U$ we have $G_{B}\subseteq \nb_{a,c}$ if and only if $B\cap[a,c]\neq\emptyset$.
    \end{lem}
    
    \begin{proof}
    We may assume that $d(a,c)<sup\,R$, since otherwise both the condition on the right and that on the left are trivially satisfied.
    
    For the only if direction assume $B\cap [a,c]=\emptyset$ and choose $h\in G_{B}$ such that $ha\ind_{B}c$. We then have $d(ha,c)=\min \{d(a,b)\oplus d(b,c)\, | \, {b\in B}\}>d(a,c)$, so that $h\in G_{B}\setminus\nb_{a,c}\neq\emptyset$. The opposite direction is a direct consequence of the triangle inequality together with the existence for $b\in B$ such that $d(a,c)=d(a,b)\oplus d(b,c)$.
    \end{proof}

    \begin{lem}
    \label{sphere union}
    Let $\tau$ be a non-trivial group topology on $G$ strictly coarser than $\tau_{st}$ and assume that there exists $W\in\mathcal{N}_{\tau}(1)$ and distinct points
    $a,b_{1},b_{2},\dots, b_{k}\in\U$ such that $W\subseteq\bigcup_{1\leq j\leq k}\nb_{a,b_{j}}$ and
    $d(a,b_{j})+d(a,b_{j})\in R$ for all $1\leq j\leq k$.
    Then there exists $1\leq j_{0}\leq k$ such that
    $\nb_{a,b_{j_{0}}}\in\idn$.
    \end{lem}
    \begin{proof}
    Assume the conclusion does not hold. This means $U\nsubseteq N^{sp}_{c,d}$ for any $U\in\idn$ and any $c,d\in\U$ with $d(c,d)=d(a,b_{j})$ for $1\leq j\leq k$. Notice also that since $\tau$ is not trivial and coarser than $\tau_{st}$ by Theorem \ref{Urysohn thm 1} it must be finer than $\tau_{m}$ so that $N_{x}(\epsilon)\cap U\in\idn$ for any $x\in\U$ and $\epsilon>0$.
    
    Let $B=\{b_{j}\}_{j=1}^{k}$ and take $h\in G_{a}$ such that $h^{-1}B\ind_{a}B$. Let $W'=W\cap W^{h}\subseteq(\bigcup_{j=1}^{k}N^{sp}_{a,b_{j}})\cap(\bigcup_{j=1}^{k}N^{sp}_{a,h^{-1}b_{j}})$.
    For any $g\in W'$ we have $g\in N^{sp}_{a,b_{j}}$ for some $1\leq j\leq k$. Combining this with the assumption on $d(a,b_{l})$, $1\leq l\leq k$, independence and the triangular inequality we get for any $1\leq i\leq k$:
    \begin{align*}
    d(a,b_{i})+d(a,h^{-1}b_{j})=d(a,b_{i})\oplus d(a,h^{-1}b_{j})=d(b_{i},h^{-1}b_{j})\leq\\
    \leq d(ga,b_{i})\oplus d(ga,h^{-1}b_{j})\leq d(ga,b_{i})\oplus d(a,h^{-1}b_{j}),
    \end{align*}
    which implies that $d(ga,b_{i})\geq d(a,b_i)$.
    
    Pick some $U\in\idn$ such that $U^{k}\subset W'$.
    We construct a sequence of elements
    $g_{1},g_{2},\dots, g_{k}\in U$ and $\epsilon_{1}\geq \epsilon_{2}\geq \dots \epsilon_{k}\in R\setminus\{0\}$ in the inductive fashion described below. We use the notation $\bar{g}_{j}=g_{j}g_{j-1}\dots g_{1}$ for $1\leq j\leq k$.

    To start with, we choose $g_{1}\in U\setminus\nb_{a,b_{1}}$. Fix $\epsilon_{1}>0$ such that $(k-1)\epsilon_{1}<d(g_{1}a,b_{1})-d(a,b_{1})$. Now, suppose that for some $1\leq j\leq k$ the elements $g_{1},g_{2},\dots g_{j}$ and $\epsilon_{1},\epsilon_{2},\dots\epsilon_{j}$ have already been chosen.
    
    We distinguish two cases. If $d(\bar{g}_{j}a,b_{j+1})>d(a,b_{j+1})$, then let $g_{j+1}=1$.
    
    Otherwise, the choice of $W'$ together with the fact that $\bar{g}_{j}\in U^{j}\subseteq W'$ implies that $d(\bar{g}_{j}a,b_{j+1})=d(a,b_{j+1})$.
    Since by the second observation in the first paragraph we have $U':=U\cap N_{a}(\epsilon_{j})\in\idn$, it follows from the first observation in that same paragraph that
    we can choose $g_{j+1}\in U'\setminus\nb_{\bar{g}_{j}a,b_{j+1}}$.
    Notice that in both cases we get $d(\bar{g}_{j+1}a,b_{j+1})>d(a,b_{j+1})$.
    Finally, we choose $\epsilon_{j+1}\in (0,\epsilon_{j})$ such that $(k-j-1)\epsilon_{j+1}<d(\bar{g}_{j+1}a,b_{j+1})-d(a,b_{j+1})$.
    
    We claim that $\bar{g}_{k}\nin\nb_{a,b_{j}}$ for any $1\leq j\leq k$. Since $\bar{g}_{k}\in U^{k}\subseteq W$, this contradicts the initial hypothesis. For $j=k$ this has already been shown. For $j<k$ we have
    \begin{align*}
	    d(\bar{g}_{k}a,b_{j})\geq d(\bar{g_{j}}a,b_{j})-\sum_{l=j}^{k-1}d(\bar{g}_{l}a,\bar{g}_{l+1}a)\geq\\
	     d(\bar{g_{j}}a,b_{j})-\sum_{l=j}^{k-1}\epsilon_{l}>d(\bar{g_{j}}a,b_{j})-(k-j)\epsilon_{j}>d(a,b_{j})
    \end{align*}
    \end{proof}

    \begin{lem}
    \label{alignment to sphere} Let $\tau$ be a group topology on $G$.
    Suppose we are given $W\in\mathcal{N}_{\tau}(1)$ and points $a,b_{1},c_{1},b_{2},c_{2},\dots, b_{k},c_{k}\in\U$ such that $d(b_{i},c_{i})+d(b_{i},c_{i})\in R$\footnote{In the sense of $+$ not $\oplus$} for all $1\leq i\leq k$ and for all $g\in W$ there is some $1\leq j\leq k$ such that $ga\in [b_{j},c_{j}]$.
    Assume $k'\leq k$ is such that $a\in[b_{j},c_{j}]$ precisely for $1\leq j\leq k'$. Then
    \begin{align*}
    \bigcup_{1\leq j\leq k'}(\nb_{a,b_{j}}\cap \nb_{a,c_{j}})\in\idn.
    \end{align*}
    \end{lem}
    \begin{proof}
    Let $(b_{i}',c'_{i})_{1\leq i\leq k}$ be an isometric copy of $(b_{i},c_{i})_{1\leq i\leq k}$ independent from the latter over $a$. Pick $g\in G_{a}$ such that $gb_{i}=b'_{i}$ and $gc_{i}=c'_{i}$ and consider $W'=W^{g^{-1}}\cap W$.
    
    Let $h$ be any element in $W'$. We need to show that there exists some $1\leq i\leq k'$ such that
    $h\in\nb_{a,b_{i}}\cap\nb_{a,c_{i}}$. We know there exist $1\leq i,j\leq k$ such that
    $ha\in [b_{i},c_{i}]\cap [b'_{j},c'_{j}]$. Let $\lambda_{l}:=d(a,b_{l})=d(a,b'_{l})$, $\mu_{l}:=d(a,c_{l})=d(a,c'_{l})$ and let also $\bar{\lambda}_{l}:=d(ha,b_{l})$, $\bar{\mu}_{l}:=d(ha,c_{l})$
    and  $\bar{\lambda}'_{l}:=d(ha,b'_{l})$, $\bar{\mu}'_{l}:=d(ha,c'_{l})$ for $1\leq l\leq k'$.
    
    Our assumption on $h$ translates into equations
    \begin{equation}
    \label{eq choice h}
    \bar{\lambda}_{i}+\bar{\mu}_{i}=d(b_i,c_i)\leq \lambda_{i}+\mu_{i}\;\;\;\bar{\lambda}'_{j}+\bar{\mu}'_{j}=d(b'_j,c'_j)=d(b_j,c_j)\leq\lambda_{j}+\mu_{j}
    \end{equation}
    while independence of $\{b_{i},c_{i}\}$ and $\{b'_{j},c'_{j}\}$ over $a$ (and the assumption on the distances $d(b_{l},c_{l})$) implies
    \begin{equation}
    \label{eq independence}
    d(b_{i},b'_{j})=\lambda_{i}+\lambda_{j}\;\;\;
    d(c_{i},c'_{j})=\mu_{i}+\mu_{j}.
    \end{equation}
    By the triangular inequality $\bar{\lambda}_{i}+\bar{\lambda}'_{j}\geq d(b_{i},b'_{j})$ and  $\bar{\mu}_{i}+\bar{\mu}'_{j}\geq d(c_{i},c'_{j})$. Putting this together with \ref{eq choice h} and \ref{eq independence} yields:
    \begin{equation}
    \label{eq 7.3}
    \lambda_{i}+\lambda_{j}+\mu_{i}+\mu_{j}\geq\bar{\lambda}_{i}+\bar{\lambda}'_{j}+\bar{\mu}_{i}+\bar{\mu}'_{j}\geq d(b_{i},b'_{j})+d(c_{i},c'_{j})=\lambda_{i}+\lambda_{j}+\mu_{i}+\mu_{j}.
    \end{equation}
    Thus, we must have
    $\bar{\lambda}_{i}+\bar{\lambda}'_{j}= d(b_{i},b'_{j})$ and  $\bar{\mu}_{i}+\bar{\mu}'_{j}= d(c_{i},c'_{j})$
    and similarly $\bar{\lambda}_{i}+\bar{\mu}'_{j}= d(b_{i},c'_{j})$ and  $\bar{\mu}_{i}+\bar{\lambda}'_{j}= d(c_{i},b'_{j})$. Using this we get:
    \begin{align*}
    \bar{\lambda_{i}}-\bar{\mu}_{i}=(\bar{\lambda_{i}}+\bar{\lambda}'_{j})-(\bar{\mu_{i}}+\bar{\lambda}'_{j})&=
    d(b_{i},b'_{j})-d(c_{i},b'_{j})\\
    &=(\lambda_{i}+\lambda_{j})-(\mu_{i}+\lambda_{j})=\lambda_{i}-\mu_{i}.
    \end{align*}
    Notice that here $d(c_{i},b'_{j})=\mu_{i}+\lambda_{j}$ follows from $c_{i}\ind_{a}b'_{j}$. Since all inequalities involved in \ref{eq 7.3} are equalities, so must be those involved in \ref{eq choice h}, so that $\lambda_{i}+\mu_{i}=d(b_i,c_i)=\bar{\lambda}_{i}+\bar{\mu}_{i}$, which implies $i\leq k'$. Together with the previous equation it also gives  $\bar{\lambda}_{i}=\lambda_{i}$ and $\bar{\mu}_{i}=\mu_{i}$.
    Thus any $h\in W'$ belongs to $\bigcup_{1\leq j\leq k'}(\nb_{a,b_{j}}\cap\nb_{a,c_{j}})$, as needed.
    \end{proof}
    
  \subsection{Proof of Theorem \ref{3 topologies theorem}}
    
    \begin{prop}
    \label{second gap}Let $\rr$ be a standard archimedean distance monoid with no least positive element. Let $\U$ be an $\rr$-Urysohn space and $G=\Isom(\U)$. Then any group topology $\tau$ strictly coarser than $\tau_{0^{+},0}$ is coarser than $\tau_{m}$.
    \end{prop}
    \begin{proof}
    
    Fix some group topology $\tau$ on $G$ coarser than $\tau_{0^{+},0}$.
    Denote by $\Delta$ the collection of all $r\in R\setminus\{0\}$ such that $\nb_{u,v}\in\mathcal{N}_{\tau}(1)$ for some (equivalently, any) pair $u,v\in\U^{2}$ with $d(u,v)=r$.

    Let $\Gamma$ be the collection of all $r\in R$ such that there exist $a\in\U$, $W\in\idn$ and some finite
    $B\subset\U$ such that $\{ga\}$ $r$-cuts $B$, for each $g\in W$.
    
    The fact that $\Delta$ is upper-closed, i.e., that $s\leq r$ and $s\in\Delta$ implies $r\in\Delta$, follows from the fact that $\rr$ is closed under taking positive differences and the following  observation:
    \begin{obs}
    \label{upper closedness} Let $\tau$ a group topology on $G=\Isom(\U)$ where $\U$ is an $\rr$-Urysohn space. Suppose that we are given $u,v,w\in\U$ with
    $d(u,w)=d(u,v)\oplus d(v,w)$. Then $\nb_{u,v}\subseteq\nb_{u,w}$.
    \end{obs}
    \begin{proof}
    Indeed, if $g\in G$ is such that $d(gu,v)\leq d(u,v)$ it follows that
    $$d(gu,w)\leq d(gu,v)\oplus d(v,w)\leq d(u,v)\oplus d(v,w)=d(u,w).$$
    \end{proof}
    
    On the other hand $\Gamma$ is upper-closed by definition. Notice as well that $\Delta=R\setminus\{0\}$ implies $\tau=\tau_{0^{+},0}$.
    
    Lemmas \ref{sphere union} and \ref{alignment to sphere} come into play through the following lemma (notice $s+s$ in the statement, rather than $s\oplus s$).
    \begin{lem}
    \label{gamma in delta}If $s+s+s+s\in R$ and $s+s\in\Gamma$ then $s\in\Delta$
    \end{lem}
    \begin{proof}
    Let $a\in\U$, $B$ a finite subset of $\U$ and $W\in\idn$ be such that $\{ga\}$ $r$-cuts $B$ for any $g\in W$.
    Let $\{b_{j},c_{j}\}_{1\leq j\leq N}$ be an enumeration of all the pairs of points in $B$ at distance at most $s+s$.
    If $s\notin \Delta$, then $ga\in\bigcup_{1\leq j\leq N}[b_{j},c_{j}]$, for each $g\in W$. Lemma \ref{alignment to sphere} then tells us that
    $\bigcup_{1\leq j\leq k'}(\nb_{a,b_{j}}\cap \nb_{a,c_{j}})\in\idn$
    where (up to reindexing) $a\in[b_{j},c_{j}]$ if and only if $j\leq k'$.
    We may assume that $d(a,b_{j})\leq d(a,c_{j})$, so that $d(a,b_{j})\leq s$.
    By Lemma \ref{sphere union}, there exists some $1\leq j_{0}\leq k'$ such that
    $\nb_{a,b_{j_{0}}}\in\idn$ and we are done.
    \end{proof}

    \begin{lem}
    \label{discutting}  If $r\nin\Gamma$, then for any $U\in\idn$ and finite subsets $A,B\subset\U$ there is $h\in U$ such that $hA$ does not $r$-cut $B$.
    \end{lem}
    \begin{proof}
    Take $A$ and $B$ as above, write $A=\{a_{j}\}_{j=1}^{k}$ and pick some $U_{0}=U_{0}^{-1}\in\idn$ such that $U_{0}^{3k-2}\subseteq U$. Let $C$ be finite such that $G_{C}\subseteq U_{0}$. We may assume $B\subseteq C$. Assume that $r\notin\Gamma$. This implies that for each $1\leq j\leq k$ there exists some $h_{j}\in U_{0}$ such that $\{a_{j}\}$ does not $r$-cut $h_{j}C$. Let $C_{j}=h_{j}B$. By Lemma \ref{non-cutting by parts} there exists $g\in G_{C_{k}}G_{C_{k-1}}\cdots G_{C_{2}}\subseteq U_{0}^{3k-3}$ such that $A$ does not $r$-cut $gC_{1}=gh_{1}B$. Equivalently, $A':=h_{1}^{-1}g^{-1}A$ does not $r$-cut $B$.
    \end{proof}
    \begin{rem}   In the previous proof we are using  only the fact that $\tau$ is coarser than $\tau_{st}$ (as opposed to $\tau_{0^{+},0}$ ).
    \end{rem}

    
    \begin{lem}
    \label{delta in gamma}$\Delta\subseteq\Gamma$.
    \end{lem}
    \begin{proof}
    Assume for the sake of contradiction that $r\in\Delta\setminus\Gamma$ for some $r$. On the one hand, given $u,v\in\U$ with $d(u,v)=r$ we have $\nb_{u,v}\in\idn$, so there exists some $W\in\idn$ such that $W^{3}\subseteq\nb_{u,v}$. Pick some $A\subseteq\U$ such that $G_{A}\subset W$.
    
    On the other hand, by Lemma \ref{discutting} there must be
    some $g\in W$ such that $gA$ does not cut $\{u,v\}$. By Lemma \ref{lambda lemma}  this implies that $G_{gA}\nsubseteq\nb_{u,v}$. However $G_{gA}\subseteq W^{3}\subseteq\nb_{u,v}$: a contradiction.
    \end{proof}
    
    We are now ready to finish the proof of Proposition \ref{second gap}. Let $S=\{s\in R\,|\,s+s\in R\}$.
    Observation \ref{upper closedness}, Lemmas \ref{delta in gamma} and \ref{gamma in delta}, together with the fact that $\rr$ is archimedean imply that either $\Gamma=\Delta=R\setminus\{0\}$ or $\Gamma\cap S=\emptyset$. The former implies $\tau=\tau_{0^{+},0}$ so from now on assume $\Gamma\cap S=\emptyset$.
    
    We need to show that $U\in\mathcal{N}_{\tau_{m}}(1)$ for any arbitrary $U\in\idn$. Pick $U_{0}\in\idn$ such that $U_{0}^{17}\subseteq U$ and $A\subset \U$ such that
    $G_{A}\subseteq U_{0}$.
    By Lemma \ref{discutting}, there exists some $g\in U_{0}$ such that $A':=gA$ does not $s$-cut $A$ for any $s\in S$. By Lemma \ref{better non-cutting constant}, there exists $A_1$ such that
    $A_{1}\cong_{A'}A$ such that $A_{1}$ does not cut $A$. Lemma \ref{misalignment} implies the existence of $A_2$ such that
    $A_{2}\cong_{A_{1}}A$ such that and $A_{2}$ is in general position relative to $A$.

    Notice that $G_{A'}\subseteq G_{A}^{g^{-1}}\subseteq U_{0}^{3}$ and
    $G_{A_{1}}\subseteq G_{A'}G_{A}G_{A'}\subseteq U_{0}^{7}$. Likewise, $G_{A_{2}}\subseteq U_{0}^{15}$
    and thus $G_{A}G_{A_{2}}G_{A}\subseteq U_{0}^{17}\subseteq U$.
    On the other hand, by Lemma \ref{small disturbance} we have  $G_{A}G_{A_{2}}G_{A}\in\mathcal{N}_{\tau_{m}}(1)$, hence $U\in\mathcal{N}_{\tau_{m}}(1)$ and we are done.
    \end{proof}
    \begin{lem}
    \label{first gap}Let $\rr$ be a standard distance monoid, $\U$ an $\rr$-Urysohn space and $G=\Isom(\U)$. Then, given finite sets $A_{1},A_{2}\subset\U$ and $V\in\mathcal{N}_{\tau_{0+,0}}$ there exists $W\in\mathcal{N}_{\tau_{0+,0}}$ such that
    $G_{A_{1}\cap A_{2}}\cap V\subseteq(G_{A_{2}}\cap W)(G_{A_{1}}\cap W)(G_{A_{2}}\cap W)$.
    \end{lem}
    \begin{proof}
    \newcommand{\nnb}[0]{\hat{N}^{sp}}
    We may assume $W=\nb_{B}$ for some finite $B\subseteq\U$ containing $A_{1}\cup A_{2}$, where
    $\nb_{B}=\bigcap_{\sst{b,b'\in B\\b\neq b'}}\nb_{b,b'}$. Given $c,c'\in\U$ with $d(c,c')\neq sup\,R$ denote by $\nnb_{c,c'}=\{g\in G\,|\,d(gc,c')=d(c,c')\}$. It is easy to see that $\nnb_{c,c'}\in\tau_{0^{+},0}$, since $\nb_{c,c'}\cap\nb_{c,c''}\subseteq\nnb_{c,c'}$ for any $c''$ such that $d(c',c'')=d(c',c)+d(c,c'')$. Since we also know that $\tau_{m}\subseteq\tau_{0^{+},0}$, it follows that $V(\epsilon):=\bigcap_{(b,b')\in X}\nnb_{b,b'}\cap N_{B}(\epsilon)\in\mathcal{N}_{\tau_{0^{+},0}}$ for any $\epsilon\in R\setminus\{0\}$, where $X=\{(b,b')\in B^{2}\,|\,b\neq b',\,d(b,b')\neq sup\,R\}$.
    
    Fix some  $\epsilon>0$  smaller than the distance between two distinct points of $B$. We will check that $V=V(\epsilon)$ satisfies the desired property.
    
    Fix $g\in V\cap G_{A_{1}\cap A_{2}}$. We let $C=A_{1}\cap A_{2}$, $D_{i}=A_{i}\setminus C$ and $E=B\setminus(A_{1}\cup A_{2})$. Write $D''_{i}=gD_{i}$, $E''=gE$ and so on.
    The proof of the Lemma reduces to that of the following claim:

    Let $F=(B\cup B'')\coprod D'_{1}\coprod E'$, where where $D'_{1},E'$  as two copies
    of $D_{1}$ and $E$ as abstract sets. As usual, given $d\in D_{1}$ we will use $d',d''$ to denote the elements (entries) of $D'_{1}$ and $D''_{1}$ respectively corresponding to $d$ and similarly for $e\in E$.
    
    \newcommand{\dd}{\bar{d}}
    In order to prove the Lemma it suffices to show that the following claim:
    \begin{claim}
    The function $\bar{d}:F\times F$ given as follows satisfies the triangle inequality:
    
    \begin{enumerate}
    \item \label{a1}$\bar{d}(f,f)=0$ for any $f\in F$;
    \item \label{a2}$\dd_{\restriction(B\cup B'')^{2}}=d_{\restriction (B\cup B'')^{2}}$;
    \item \label{a3}$\dd(b_{1}',b_{2})=\dd(b_{1}',b''_{2})=d(b_{1},b_{2})$ for any distinct $b_1,b_2\in D_{1}\cup E$.
    \item \label{a4} $\dd(b'_{1},b'_{2})=d(b_{1},b_{2})$ for $b_{1},b_{2}\in D_{1}\cup E$.
    \item \label{a5}$\dd(b',b)=\dd(b',b'')=\epsilon$ for any $b\in D_{1}\cup E$.
    \end{enumerate}
    \end{claim}
    
    Indeed, if the claim holds then we can embed a copy of $(F,\bar{d})$ isometrically in $\U$ over $BB''$. The chain $CD_{1}D_{2}E\cong CD'_{1}D_{2}E'\cong CD'_{1}D''_{2}E'\cong CD''_{1}D''_{2}E''$ then witness the desired inclusion in a fashion to the proof of \ref{lem-zig-used}. Our conditions imply
    any $h_{1}$ mapping $CD_{1}D_{2}E$ to $CD'_{1}D_{2}E'$ is in $G_{A_{2}}\cap \hat{N}^{sp}_{B}$, any $h_{2}$ mapping
    $h_{1}^{-1}(CD'_{1}D_{2}E')=B$  to $h_{1}^{-1}(CD'_{1}D''_{2}E')$  is in $G_{A_{1}}\cap \hat{N}^{sp}_{B}$
    and $h_{2}^{-1}h_{1}^{-1}g\in G_{A_{2}}\cap\hat{N}^{sp}_{B}$.
    
    As for the proof of the claim, we begin by noticing that the triangle inequality is automatically satisfied by triples of distinct points contained in $B\cup B''$. By points \ref{a3}, \ref{a4} the same is true for triples contained in
    $B\cup D_{1}'\cup E'$ and $B''\cup D'_{1}\cup E'$ which do not contain pairs of points of the form
    $\{b,b'\}$  or $\{b'',b'\}$ for $b\in D_{1}\cup E$.
    
    If our triple is of the form $\{b,b',c\}$ with $b\in D_{1}\cup E$, $c\in B\cup D'_{1}\cup E'$, then by \ref{a3}, \ref{a4} and \ref{a5}, and the choice of $\epsilon$ yields $\dd(b,b')=\epsilon<\dd(b,c)=\dd(b',c)$, from which the triangle inequality easily follows.
    The same argument works for triples of the form $\{b',b'',c\}$ with $c\in D'_{1}\cup E'\cup B''$, $b\in D_{1}\cup E$.
    So we are left with triples of the form $\{b_{1},b'_{2},b''_{3}\}$ for $b_{1},b_{3}\in B$, $b_{2}\in D_{1}\cup E$.

    If $b_{1}=b_{3}\neq b_{2}$, then the triangle inequality can be established by a similar argument to that of the previous paragraph using the fact that
    $g\in N_{B}(\epsilon)$ and $\epsilon<\inf_{\substack{b_{1},b_{2}\in B\\b_{1}\neq b_{2}}}\,\,\dd(b_{1},b_{2})$ instead of \ref{a5}.
    
    If $b_{1}=b_{2}\neq b_{3}$, then the argument is similar. We have
    \begin{align*}
    \dd(b_{1},b'_{2})=\epsilon & <\dd(b'_{2},b''_{3})=d(b_{2},b_{3})=d(b_{1},b_{3})\leq\\ & \dd(b_{1},b''_{3})\oplus\epsilon 
     =\dd(b_{1},b''_{3})\oplus\dd(b_{1},b'_{2}),
    \end{align*}
    where in the last inequality we are using the fact that $g\in \bigcap_{b\in B}N_{b}(\epsilon)$ and $b''_{3}=gb_{3}$. If $b_{2}=b_{3}\neq b_{1}$ the argument is entirely analogous.

    If $b_{1}=b_{2}=b_{3}$ we have $\dd(b_{1},b''_{3})\leq\epsilon=\dd(b_{1},b'_{2})=\dd(b'_{2},b''_{3})$  and the result follows as well.
    
    Finally we have the case in which $b_{1},b_{2},b_{3}$ are all distinct.
    First of all, we have:
    $$
    \dd(b_{1},b''_{3})\leq \dd(b_{1},b_{3})\leq \dd(b_{1},b_{2})\oplus \dd(b_{2},b_{3})=\dd(b_{1},b'_{2})\oplus \dd(b'_{2},b''_{3}).
    $$
    
    If the maximum distance between two points in $\{b_{1},b_{2},b_{3}\}$ is witnessed by
    $d(b_{1},b_{3})$, then we also have
    \begin{align*}
    \dd(b_{1},b'_{2})=d(b_{1},b_{2})\leq d(b_{1},b_{3})&\leq \dd(b_{1},b''_{3})\oplus\epsilon \leq \\ & \dd(b_{1},b''_{3})\oplus d(b_{2},b_{3})=\dd(b_{1},b''_{3})\oplus \dd(b_{2}',b''_{3})
    \end{align*}
    and the same argument yields $\dd(b'_{2},b''_{3})\leq \dd(b'_{2},b_{1})\oplus \dd(b_{1},b''_{3})$.

    We are left with the situation in which the maximum distance between two points in $\{b_{1},b_{2},b_{3}\}$ is witnessed by
    $d(b_{1},b_{2})$ or $d(b_{2},b_{3})$. We might as well assume it is witnessed by $d(b_{1},b_{2})$, since the other sub-case is entirely analogous.
    
    One possibility is that $\dd(b_{1},b_{3})<sup\,\,R$, in which using $g\in \hat{N}^{sp}_{B}$ we get:
    $$
    \dd(b_{1},b''_{2})=d(b_{1},b_{2})\leq d(b_{1},b_{3})\oplus d(b_{2},b_{3})=\dd(b_{1},b''_{3})\oplus \dd(b'_{2},b''_{3})
    $$
    This suffices to settle the triangle inequality in this case, as we also have:
    $$
    \dd(b'_{2},b''_{3})=d(b_{2},b_{3})\leq d(b_{1},b_{2})=\dd(b_{1},b'_{2}).
    $$
    The other possibility left is that $d(b_{1},b_{2})=d(b_{1},b_{3})=sup\,\,R$. The triangle inequality is then clear:
    \begin{align*}
    \dd(b'_{1},b''_{3})\leq\dd(b_{1}',b''_{2})=d(b_{1},b_{2})=d(b_{1},b_{3})\leq\dd(b_{1}',b''_{3})\oplus\epsilon\leq \\  \dd(b_{1}',b''_{3})\oplus d(b_{2},b_{3})=\dd(b_{1}',b''_{3})\oplus \dd(b'_{2},b''_{3}), \mbox{ and }
    \end{align*}
    $$
    \dd(b'_{2},b''_{3})=d(b_{2},b_{3})\leq d(b_{1},b_{2})=\dd(b_{1},b_{2}).
    $$
    This concludes the proof of the Claim and with it the proof of the Lemma.
    
    \end{proof}
    
    \begin{thm*}[\ref{3 topologies theorem}]
    
    Let $\rr$ be a standard archimedean distance monoid with no minimal positive element, $\U$ an $\rr$-Urysohn space and $G=\Isom(\U)$. Then there are exactly $4$ group topologies on $G$ coarser than $\tau_{st}$
    $$
    \tau_{st}\supsetneq\tau_{0^{+},0}\supsetneq\tau_{m}\supsetneq\{\emptyset,G\}.
    $$
    \end{thm*}
    \begin{proof}
    The fact that $\{\emptyset,G\}$ is the only group topology strictly coarser than $\tau_{m}$ is a particular case of Theorem \ref{Urysohn thm 1} and the fact that there is no group topology $\tau$ with $\tau_{0^{+},0}\supsetneq\tau\supsetneq\tau_{m}$ is the content of Proposition \ref{second gap}.
    
    The fact that any group topology strictly coarser than $\tau_{st}$ is coarser than $\tau_{0^{+},0}$ follows from
    a combination Lemma \ref{jump lemma} with Lemma \ref{first gap}. We are applying Lemma
    \ref{jump lemma} with $\tau^{*}=\tau_{0^{+},0}$. The first alternative in its conclusion leads to
    $\tau=\tau_{st}$, while in the second $X'=\emptyset$, since in this case the action is transitive.
    \end{proof}
    
    We will explain the system behind the notation $\tau_{0^{+},0}$ in the next section. As the reader might have guessed, the true identity of $\tau_{st}$ and $\tau_{m}$ will be $\tau_{0,0}$ and $\tau_{0^{+},0^{+}}$.

    

\section{Parametrizing topologies of isometry groups of generzalized Urysohn spaces}
  
  \label{Urysohn parametrization section}
  We borrow the following construction from Conant \cite{conant2015model}. Let $\rr$ be a distance monoid. By an \emph{end segment} of $\rr$ we mean a subset $\alpha\subset R$ with the property that $t\in\alpha$ whenever $s\in\alpha$ for some $s\leq t$. Let $R^{*}$ be either the collection of end segments of $R$ in case $R$ has no maximal element or the collection of non-empty end segments in case $R$ has a maximal element. There is a natural order $\leq^{*}$ on the set $R^{*}$ given by $\alpha\leq^{*}\beta$ if and only if $\beta\subseteq\alpha$.
  One can endow $R^{*}$ the operation $\oplus^{*}$, defined as $\alpha\oplus^{*}\beta=\inf_{R^{*}}\{r\oplus s\,|\,r\in \alpha,s\in\beta\}$ (see \cite{conant2015model}[2.6.4., 2.6.5.]). This gives $R^*$ the structure of a distance monoid $\rr^*$.

  The natural embedding $\nu$ from $R$ into $R^{*}$ sending $r\in R$ to $\{s\in R\,|\,s\geq r\}$ respects the linear order and the operations on both sides: $\nu(s\oplus t)=\nu(s)\oplus^{*}\nu(t)$. From now on we identify $R$ with $\nu(R)$ and write $\oplus$ instead of $\oplus^{*}$. If $R$ contains no minimal element greater than $r\in R$, then we denote the successor $\{s\,|\,s>r\}\in R^{*}$ as $r^{+}$. Of particular interest for us will be $0^{+}$.
  Notice that provided $0^{+}$ exists, the condition $\forall r\in R\setminus\{0\}\,\,\exists s\in R\setminus\{0\}\,s\oplus s\leq r$ is equivalent to $0^{+}\oplus 0^{+}=0^{+}$.
  
  \newcommand{\nn}[0]{\mathcal{N}}
  \newcommand{\ip}[0]{IP} 
  
  Let $\U$ be an $\rr$-Urysohn space and $G$ its group of isometries. By an ideal of $R$ we mean a non-empty closed subset of $R$ closed under addition and such that $s\leq r\in R$ implies $s\in R$.
  Given an ideal $\mu$ of $R$ let:
  \begin{align*}
  G^{b}_{\mu}:=\{g\in G\,|\,\forall u\in\,\,\U\,\,d(gu,u)\in\mu\}.
  \end{align*}
  \begin{obs}
  \label{hausdorfness} If $\mu$ is an ideal of $\rr$ then $G^{b}_{\mu}$ is a normal subgroup of $G$. If $\rr$ is countable or
  archimedean then $G^{b}_{\mu}=G^{b}_{\mu'}$ only if $\mu=\mu'$.
  \end{obs}
  This is true in other situations as well, but we will skip the discussion at this point. \\ Given $u,v\in \U$ and $r\in R$, let $N_{u,v}(r)=\{g\in G\,|\,d(gu,v)\leq r\}$. Given a function $f:R\to R^{*}$, let $S_{f}=\{N_{u,v}(r)\,|\,r\in f(d(u,v))\}$ and $\mu_{f}=\{r\in R\,|\,\forall s\in R\,,\,\,s\oplus r\leq f(s)\}$.
  
  Let also  \begin{align*}
  T(r)=\{(s,t)\in R^{2}\,|\,r\leq s\oplus t, s\leq r\oplus t, t\leq r\oplus s\}.
  \end{align*}
  
  \begin{lem}
  \label{is group topology} Let $f:R\to R^{*}$ be a function such that:
  \begin{enumerate}[(a)]
  \item \label{consistent bound}$f(d)\geq d$, for each $d\in R$;
  \item \label{f+f} $f(d)\geq\inf_{R^*}\{f(s)\oplus f(t)\,|\,(s,t)\in T(d)\}$, for each $d\in R$.
  \end{enumerate}
  Then  $S_{f}$ generates the base of neighbourhoods of the identity of a  group topology $\tau_{f}$. Moreover, $\mu_{f}$ is an ideal and the closure of $1$ in $\tau_{f}$ coincides with the group $G^{b}_{\mu_{f}}$.
  \end{lem}
  \begin{proof}
  To begin with observe that $N^{-1}_{(u,v)}(r)=N_{(v,u)}(r)$, all sets $N_{(u,v)}(r)$ contain the identity map and the collection $S_{f}$ is invariant under conjugation.
  Take now any two points $u,v\in\U$. Let $d=d(u,v)$. Since $r\geq f(d)\geq\inf_{R^*}\{f(s)\oplus f(t)\,|\,(s,t)\in T(d)\}$, there must be $(s,t)\in T(d)$ such that
  $f(s)\oplus f(t)\leq r$. Since $(s,t)\in T(d)$, there must exist some point $w\in\U$ such that $d(w,u)=s$ and $d(w,v)=t$.
  Given any $h,g\in W:=N_{u,w}(s)\cap N_{w,v}(t)$ we have
  $d(gu,w)\leq s$, $d(hw,v)\leq t$. Hence $d(hgu,v)\leq d(hgu,hw)\oplus d(hw,v)\leq s\oplus t\leq r$. We conclude  $W^{2}\subseteq N_{u,v}(r)$ and we are done.
  
  If $r,r'\in\mu_{f}$ then for all
  $q\in R$ and $(s,t)\in T(q)$ we have $f(s)\oplus f(t)\geq s\oplus t\oplus r\oplus r'\geq q\oplus r\oplus r'$.
  This implies that $f(q)=\inf_{R^*}\{f(s)\oplus f(t)\,|\,{(s,t)\in T(q)}\}\geq q\oplus r\oplus r'$. Thus $r\oplus r'\in\mu_{f}$. Hence, $\mu_{f}$ is an ideal of $\rr$.
  Let $K$ be the closure of $1$ in $\tau_{f}$, i.e., $K=\bigcap_{W\in S_{f}}W$.
  Clearly $G^{b}_{\mu_{f}}\subset K$. For the opposite inclusion consider $g\nin G^{b}_{\mu_{f}}$ there exists
  some $u\in\U$ such that $r=d(gu,u)\nin\mu_{f}$, which means that there exists $s\in R$ such that $s\oplus r>f(s)$. Universality of $\U$ implies the existence of $v\in\U$ such that $d(gu,v)=s\oplus r$, so that $g\nin N_{(u,v)}(r)\in S_{f}$.
  
  \end{proof}
  
  An additional condition is needed to ensure the faithfulness of the parametrization $f\mapsto\tau_{f}$.
  \begin{defn}
  We say that $f:R\to R^{*}$ is a \emph{$\rr$-modulus of continuity} if it satisfies Conditions \ref{consistent bound} and \ref{f+f} of Lemma \ref{is group topology} together with the following:
  \begin{enumerate}[(a)]
  
  \setcounter{enumi}{2}
  \item \label{f+id}  $f(d)\leq\inf_{R^*}\{s\oplus f(t)\,|\,(s,t)\in T(d)\}$, for each $d\in R$.
  \end{enumerate}
  \end{defn}
  \noindent
  Notice that this implies the inequalities in \ref{f+f} and \ref{f+id} are actually equalities.
  \begin{lem}
  \label{monotonicity moduli}If $f,g:R\to R^{*}$ are $\rr$-moduli of continuity, then $\tau_{f}\subseteq\tau_{g}$ if and only if
  $g(r)\leq f(r)$ for all $r\in R$.
  \end{lem}
  This is an easy consequence of Condition \ref{f+id} together with the following fact:
  \begin{lem}
  \label{neighbourhood containment}Suppose we are given $v,w\in\U$, $r\in R$ and a finite collection $X\subseteq\U^{2}\times R$, where $(u,u',s)\in X$ implies $d(u,u')\leq s$. Then $\bigcap_{(u,u',s)\in X}N_{u,u'}(s)\subseteq N_{v,w}(r)$ if and only if there exists
  $(u,u',s)\in X$ such that either
  \begin{itemize}
  \item $d(w,u)\oplus d(u,u')\oplus d(u',v)\leq r$ or
  \item $d(v,u)\oplus d(u,u')\oplus d(u',w)\leq r$.
  \end{itemize}
  \end{lem}
  \begin{proof}
  The `if' part is clear. For the only part, assume that neither of the two cases above holds.
  We want to show that $\bigcap_{(u,u',s)\in X}N_{u,u'}(s)\nsubseteq N_{v,w}(r)$
  We may assume there is a finite set $Y=\{u_{i}\}_{i=1}^{q}$ such that
  $v=u_{1},w=u_{2}$ (assume without loss of generality that $v\neq w$) and $X$ contains exactly one triple $(u_{i},u_{j},s_{i,j})$ for any
  $1\leq j\leq q$ and that $s_{i,j}=s_{j,i}$. We construct a new finite $\rr$-metric space as follows. As the underlying set $Z$ we take the disjoint union of two copies $Y^{j}=\{u^{j}_{i}\}_{i=1}^{q}$, $j=1,2$ of $Y$ after identifying $u^{1}_{i}$ and $u^{2}_{i}$ in case $s_{i,i}=0$.
  \newcommand{\bd}[0]{\bar{d}}
  \newcommand{\ts}[0]{\tilde{s}}
  Consider the map $\bd:Y \times Y\to\rr$ given by $\bd(u^{l}_{j},u^{l}_{k})=d(u_{j},u_{k})$ (abbreviated as $d_{j,k}$) and
  $\bd(u^{1}_{j},u^{2}_{k})=\min\{d_{j,j'}\oplus s_{j',k'}\oplus d_{k',k}\,|\,1\leq j',k'\leq q\}$,
  (abbreviated as $\ts_{j,k})$.
  We claim that $(Z,\bd)$ is an $\rr$-metric space. Since our starting assumption translates as $\tilde{s}_{1,2}>r$, this will witness $\bigcap_{(u,u',s)\in X}N_{u,u'}(s)\nsubseteq N_{v,w}(r)$.
  By symmetry, all we need to check is $\ts_{i,j}\leq\ts_{i,j'}\oplus\ts_{j',j}$ as well as the symmetric inequality for all $1\leq i,j,j'\leq q$. This follows easily from the definition and
  the inequality $d_{j',l}\oplus d_{l,j}\geq d_{j',j}$.
  \end{proof}

  
  Given any distance monoid $\rr$, let $Id(\rr^*)$ stand for the collection of all idempotents of $\rr^{*}$.
  The following claim follows easily from the definitions and
  the fact that $\alpha\oplus\beta=\inf_{\rr^{*}}\{s\oplus t\,|\,s\in\alpha,t\in\beta \}$
  \begin{lem}
  \label{ladder}Let $g:R\to Id(\rr^{*})$ satisfy the following properties:
  \begin{enumerate}[(i)]
  \item $g$ is constant on any archimedean class;
  \item $g$ is non-increasing;
  \item \label{f+id proxy} If $[r]<[s]$  and
  $g(s)<g(r)$, then $g(r)\leq s$;
  \item \label{weird condition} For any $r\in R$ there exist $s,t\in [r]$ such that
  $s\oplus t= r$.
  \end{enumerate}
  Then the function $\hat{g}:R\to R^{*}$ given by $\hat{g}(r)=r\oplus g(r)$ is an $\rr$-modulus of continuity.
  \end{lem}
  \begin{proof}
  Let us check condition \ref{f+id} first. If it does not hold, then there exists
  $r\in R$ such that $r\oplus g(r)> s\oplus t\oplus g(t)$ for some $(s,t)\in T(r)$.
  Since $s\oplus t\geq r$ this implies $g(t)<g(r)$  and thus $[t]>[r]$, since $g$ is non-increasing and constant on archimedean classes. Since $(s,t)\in T(r)$, this implies in turn $[s]=[t]$. By condition \ref{f+id proxy} it also implies $g(r)\leq t$.
  We cannot have $[g(r)]<[t]$ so necessarily $[g(r)]=[t]>[r]$.
  Notice that in general if an archimedean class $[q]$ contains an idempotent $q_{0}$, then $q_{0}=\max[q]$ and $q_{0}\oplus p=q_{0}$ for any $p\leq q_{0}$. It follows that $t=g(r)$ and both left and right-hand side of the inequality we started with are actually equal to $g(r)$; a contradiction.
  
  Condition \ref{f+f} can be proved using \ref{weird condition} in a similar way.
  \end{proof}
  
  We will refer to any map $g:R\to Id(\rr^{*})$ as above as an $\rr$-ladder of idempotents.
  
    \subsubsection*{Example.}
      Given any $\alpha\in Id(\rr^{*})$ it is easy to check that the function $g$ with constant value $\alpha$ satisfies the definition above. If $\alpha=0$, then $\tau_{\hat{g}}=\tau_{st}$. If $\alpha=0^{+}$, then $\tau_{\hat{g}}$ coincides with the
      generalized point-wise convergence topology $\tau_{m}$.  The system of generating sets at the identity given this way is larger than the one in the definitions above, but it can be easily checked the extra generators are redundant.
      Notice that $\tau_{\hat{g}}$ is Hausdorff if $\inf\{im(g)\,|\,g\in G\}\in\{0,0^{+}\}$.
      This is an only if in case $\rr$ is countable or $\U$ the completion of a countable Urysohn space.\\

      If $\rr$ has only finitely many archimedean classes, we can think of a ladder as the non-increasing sequence $\sigma$ of its values on $R/\sim$ and accordingly write $\tau_{\sigma}$ instead of $\tau_{\hat{g}}$. In particular, in the archimedean case we can wirte $\tau_{0,0}$ for $\tau_{st}$ and $\tau_{0^{+},0^{+}}$ for $\tau_{m}$.
      Notice that the condition in Claim \ref{spherical topology} is precisely what one need in order to get \ref{weird condition}

    \subsubsection*{Example}
      The archimedeanity assumption  of Theorem \ref{Urysohn thm 1} is essential.
      Let $(\Lambda,0,+,<)$ be the abelian group $\Q\times\Q$ equipped with the lexicographical order. Let $\rr=(R,0,\leq,\oplus)$ be the distance monoid given by the 	restriction of $(\Lambda,0,+,<)$ to the non-negative part of $\Lambda$. Let $\alpha\in R^{*}$ be the upper closed set $\{(a,b)\in\Lambda^{\geq 0}\,|\,a>0\}$.
      The four idempotents of $\rr^{*}$ are $0<0^{+}<\alpha<\emptyset$. Here we have ladders $(0,0,0)$, $(0^{+},0,0)$ and $(0^{+},0^{+},0^{+})$ whose associated topologies give analogues to $\tau_{st},\tau_{0^{+},0}$ and $\tau_{m}$ in the archimedean case.
      
      But we get also $(0^{+},0^{+},0)$, witnessing the failure of Lemma \ref{second gap}, as well as $(\alpha,\alpha,0^{+})$ and $(\alpha,\alpha,0)$, witnessing the fact that $\tau_{m}$ is not minimal. Notice that by condition \ref{f+id proxy} there aren't any other ladders whose minimal value is either $0$ or $0^{+}$. Indeed, $g(0)=\emptyset=\max R^{*}$ implies $g([r])=\emptyset$, for each $r\in R$, while $g((0,\alpha)\cap R)<\alpha$ implies also $g(0)=\alpha$.
      
      \begin{lem}
      Assume $\rr=\mathcal{S}$ where $S=\Lambda^{\geq 0}$ for some ordered abelian group $(\Lambda,0,\leq,+)$. Then any modulus of continuity comes from some $\rr$-ladder of idempotents.
      \end{lem}
      \begin{proof}
      Consider any $\rr$-modulus of continuity $f:R\to R^{*}$ and let $g:R\to R^{*}$ be given by $g(r)=f(r)-r$.
      Notice that since $R$ is closed under differences and $f(r)\geq r$ the right hand side is a well defined element of $R^{*}$.
      Given any $t,'t\in R$ with $t\leq t'$. Property \ref{f+id} implies that $f(t')\leq f(t)+(t'-t)$
      from which it follows that $g(t')\leq g(t)$. The same property applied to $(t',t'-t)\in T(t)$ also yields $g(t)\leq g(t')+ (t'-t)$.
      
      We now claim that $im(g)\subseteq Id(\rr^{*})$.
      Property \ref{f+f} yields:
      \begin{align*}
      g(r)\geq\inf_{R^*}\{f(s)+f(t)-r\,|\,(s,t)\in T(r) \}=\hspace{0.5 cm}\\
      =\inf_{R^*}\{g(s)+ s+ g(t)+ t-r\,|\,{(s,t)\in T(r)}\}.
      \end{align*}
      It thus suffices to show that $g(s)+ g(t)+ ((s+ t)-r)\geq g(r)+ g(r)$, for each $(s,t)\in T(r)$.
      This is clear in case $s,t\leq r$, since $g(r)$ is non-increasing and $s+t-r\geq 0$. If $s$ or $t$ are larger than $r$, then
      the result follows from inequalities $g(t)+ (t-r)\geq g(r)$, $g(s)+ (s-r)\geq g(r)$.
      Since $g(r+r)\leq r+g(r)\leq g(r)+g(r)=g(r)$ we conclude that $g$ is constant on archimedean classes.
      The second property of the definition follows easily from \ref{f+id}.
      \end{proof}

      \begin{rem}
      In general not all moduli of continuity need come from a ladder.
      Take $a\in\R^{+}\setminus\Q$ and consider $R=\Q^{\geq 0}+\Q^{\geq 0}a$. The sum and order inherited from $\R$ make $R$ into a distance monoid. Let
      $f:R\to R^{*}$ evaluate to $r$ on any $r\nin\Q\cap R$ and to $r^{+}$ on any $r\in\Q\cap R$. It is easy to check that $f$ is a $\rr$-modulus of continuity.
      \end{rem}
      
      \begin{question}
      Does some distance monoid $\rr$ admitting moduli of continuity whose range contains non-idempotent elements?
      \end{question}
      
      \begin{problem}
      Classify the collection of $\rr$-moduli of continuity associated with an arbitrary $\rr$.
      \end{problem}
      
      An alternative (and in the long run better way) of thinking about moduli of continuity is in terms of parameters of a generalized version of bi-Katetov maps as described in \cite{Usp}, or types of pairs of copies of $\U$. Any function $f:R\mapsto R^{*}$ satisfying \ref{f+id} can be associated to a $G$-invariant bi-Katetov map that assigns distance $f(r)$ to any pair $u'$,$v''$
      where $u'$ and $v''$ are copies of $u\in\U$ and $v\in\U$ respectively in the two copies of $\U$. Condition \ref{f+f} on the other hand states that the type is idempotent. So the following conjecture seems natural from that point of view as well.
      
      \begin{conjecture}
      Given any distance monoid $\rr$ and any $\rr$-Urysohn space $\U$ any group topology on $\Isom(\U)$ strictly coarser than
      $\tau_{st}$ is of the form $\tau_{f}$ for some $\rr$-modulus of continuity $f$.
      \end{conjecture}

\section{Zariski topology}
  \label{sec-Zariski}
  
  Given a group $G$, an {\it equation} ({\it inequality}) over $G$ in one variable $x$ is an
  expression of the form $w(x,\alpha)=1$ ($w(x,\alpha)\neq 1$), where $w$ is
  a term over $x\cup\alpha$ in the language of groups with inversion.
  We can think of $w$ as an expression of the form
  \begin{align*}
  \alpha_{0}x^{\epsilon_{0}}\alpha_{1}\cdots x^{\epsilon_{m-1}}\alpha_{m};
  \end{align*}
  where $j_{l}\in\{1,\dots, r\}$ and $\epsilon_{l}\in\{1,-1\}$ for $0\leq l\leq m-1$. This represents an element of the group $G\frp\subg{x}$, where $\subg{x}$ is the cyclic free group over $x$. It is easy to check that if $w$ and $w'$ correspond to the same group element then the equations $w=1$ and $w'=1$ have the same set of solutions.
  Hence, one can always assume that the above word is reduced, i.e. $\alpha_{l}\neq 1$ whenever $\epsilon_{l}+\epsilon_{l+1}=0$.
  We say that an equation is trivial if $w$ represents the trivial element in $G\frp\subg{x}$.
  
  A system of equations (inequalities) is just the conjunction of finitely many equations (inequalities).
  It can be checked that the collection of all sets of solutions of systems of inequalities over $G$ is the base of a topology on $G$ known as the Zariski topology, which we will denote by  $\tau_Z$.
  
  As mentioned in the introduction, according to a result of Gaughan in \cite{Gua} for the group $S_\infty$ the Zariski topology is the same as the standard topology and hence a group topology. Later, in \cite{BGP} the same is shown for every subgroup of $S_\infty$, containing all permutations of finite support. Here, we investigate the Zariski topology for the automorphim groups of some homogeneous countable structures including the automorphism groups of Fra\"iss\'e limit structures of free amalgamation classes and rational Urysohn spaces (bounded and unbounded).
  
  \begin{defn}
  \label{def-st-ub} Let $\alpha\in \Aut{\M}$. We say $\alpha$ is {\it strongly unbounded} if for every finite subset $A$ of $M$ and $b\in M\backslash \acl(A)$, there is a realization $c\models \typ(b/A)$ such that $\alpha(c)\notin \acl(cA)$.
  \end{defn}
  \begin{rem}
  Notice that being strongly unbounded is a strictly weaker notion than moving maximally in the sense of \cite{TentZUry} (and almost moving maximally in \cite{EGT}).
  \end{rem}
  
  The following is a generalization of a classical argument for finding embeddings of free groups into automorphism groups of $\omega$-categorical structures (see \cite{MacS}, Prop. 4.2.3).
  \begin{lem}
  \label{lem-comeager}
  Suppose  $\M$ is a countable $\omega$-saturated first order structure in which acl is locally finite and put $G:=\Aut{\M}$. Assume $\alpha$  is a finite tuple of automorphisms of $\M$ where  $\alpha_i$ is either $1$ or strongly unbounded and $w(x,\alpha):=\alpha_0x^{\epsilon_0}\alpha_1\cdots x^{\epsilon_m}\alpha_{m+1}$ a reduced word in one variable. Then the set of solutions of the equation $w=1$ is meager in $(G,\tau_{st})$. 
  \end{lem}
  \begin{proof} As remarked above, the set of solutions of $w(x,\alpha)=1$ is closed in any group topology. We want to show it has empty interior. Aiming for contradiction suppose that is not the case. Up to performing a change of variable of the form $x\mapsto x\gamma$ we can assume that there is a finite subset $B$ such that $w(G_{B},\alpha)=1$.
  
  We will construct inductively a chain of elementary maps $\id_B=f_{m+1}\subseteq f_{m}\subseteq \dots \subseteq f_0$ together with $a\in M$ and $c_{k}:= \alpha_k f_k^{\epsilon_k}\alpha_{k+1}\cdots f_m^{\epsilon_m}\alpha_{m+1}(a)$ for $0\leq k \leq m+1$ with the property that:
  \begin{equation}
  \tag{$\dagger$} \label{ck cond}   c_{k}\notin \acl (\dom(f_{k}^{\epsilon_{k-1}}));
  \end{equation}
  for $1\leq k \leq m+1$ and $c_{0}\neq a$. This finishes the proof. Indeed, given any extension $\beta\in G$ of $f_0$, clearly $\beta\in G_B$ but $w(\beta,\alpha)(a)=c_{0}\neq a$.
  
  We start by choosing any  $a\in \alpha_{m+1}^{-1}(M\backslash \acl{(B)})$ so that
  $c_{m+1}=\alpha_{m+1}(a)\nin\acl{(B)}=\acl{(\dom{(f_{m+1})})}$.
  For the induction step, assume $f_{k+1}$ has been successfully constructed for some $2\leq k\leq m$. We want to extend it to a map $f_{k}$ satisfying (\ref{ck cond}). Let $D_k=\dom (f_{k+1}^{\epsilon_{k}})$ and $q(x,y)=\typ(c_{k+1},D_{k})$.
  Let $p(x):=q(x,D_{k})$ and $p'(x):=q(x,f_{k+1}^{\epsilon_{k}}(D_{k}))$.
  For any realization $e\models p'(x)$ the map $g_{e}$ defined by $g_{e}^{\epsilon_{k}}=f_{k+1}^{\epsilon_{k}}\cup\{(c_{k+1},e)\}$
  is elementary by construction. Our goal is thus to show that for some such $e$ if we let $f_{k}=g_{e}$ then the resulting  $c_{k}\in M$ satisfies both (\ref{ck cond}) and $c_{k}\neq a$. Notice that by the induction hypothesis
  $c_{k+1}\nin \acl( D_{k})$, i.e. $p(x)$ is non-algebraic and hence so is $p'(x)$. Since by assumption $\M$ is $\omega$-saturated, $p'(x)$ has infinitely many realizations in $M$. There are two different scenarios to consider.
  
  If $\epsilon_{k}=\epsilon_{k-1}$, then take $e\models p'$ with $e\nin \{a\}\cup\alpha_{k}^{-1}(\acl(D_{k}c_{k+1}))$. This is possible by the observation of the last paragraph and the local finiteness of $\acl{-}$. Taking $f_{k}:=g_{e}$ we obtain:
  \begin{align*}
  c_{k}=\alpha_{k}(f_{k}^{\epsilon_{k}}(c_{k+1}))=\alpha_{k}(e)\nin \acl(D_{k}c_{k+1})=\acl(\dom(f_{k}^{\epsilon_{k}}))=\\=\acl(\dom(g_{e}^{\epsilon_{k}}))=\acl(\dom(f_{k}^{\epsilon_{k}}))=\acl(\dom(f_{k}^{\epsilon_{k-1}})).
  \end{align*}
  Consider now the case $\epsilon_{k}=-\epsilon_{k-1}$. Since $w$ is reduced, this implies that $\alpha_{k}\neq 1$ and thus, by assumption, that $\alpha_{k}$ is unbounded. The type $p'(x)$ is non-algebraic with parameters in $D':=f_{k+1}^{\epsilon_{k}}(D_{k})$.
  Unboundedness implies there exists a realization $e$ of $p'(x)$ such that
  $\alpha_{k}(e)\nin\acl(D'e)$. But $D'e=im\,f_{k}^{\epsilon_{k}}=\dom f_{k}^{\epsilon_{k-1}}$, hence condition (\ref{ck cond}) follows for $c_{k}=\alpha_{k}(e)$ as well.
  In the last step all we have to do is to choose $e\models p'$ such that $\alpha_{0}(e)\neq a$. This is clearly possible by the fact that $p'(x)$ is non-algebraic.
  \end{proof}
  \begin{rem}
  The actual sufficient condition given by the proof is that no two occurrences of opposite sign of $x$ in $w$ are separated by non-strongly unbounded element from the group. In particular, words involving only positive powers of $x$ have always meager sets of solutions.
  \end{rem}

  \begin{lem}
  \label{using meagerness}	 Suppose $\M$ is a countable first order structure such that the solution sets of all non-trivial equations of the form $w(x,\alpha)$ are meager in $G=\Aut{\M}$ with respect to the standard topology. Then $\tau_Z$ is not a group topology for $G$.
  \end{lem}
  \begin{proof}
  Indeed, fix $\alpha\in G$ and consider the equation $z\alpha^{-1}=1$ in $G$, where $\alpha\in G$.
  Now, suppose we are given two systems of inequalities in one variable:
  \begin{align*}
  \Sigma(x,\beta)\neq 1;\\
  \Pi(y,\beta)\neq 1;
  \end{align*}
  where $\beta=(\beta_{1},\dots,\beta_{k})\in G^{k}$ is the tuple of parameters appearing in the two systems, i.e., the non-trivial elements of $G$ appearing in the corresponding normal forms.
  Consider the system of inequalities:
  \begin{align*}
  \{\Sigma(x,\beta)\neq 1\}\cup\{\Pi '(x,\beta')\neq 1\};
  \end{align*}
  where $\Pi'(x,\beta')$ is the system obtained from $\Pi(y,\beta)\neq 1$ by replacing $y$ with $x^{-1}\alpha$ (the substitution corresponds to an automorphism of $G\frp\subg{x}$, so this is still a non-trivial system of equations)  and $\beta'$ the updated superset of parameters.
  Given a solution $x_{0}$ of $\Sigma(x,\beta)\neq 1$ and $\Pi '(x,\beta')\neq 1$ the pair $(x_{0},x_{0}^{-1}\alpha)$ belongs to the neighbourhood defined by the systems $\Sigma\neq 1$ and $\Pi\neq 1$ but their product satisfies the initial equation $z\alpha^{-1}=1$.
  Note that in a topological group any finite conjugation of group action is continuous and the pre-image of a nowhere dense set is nowhere dense. Hence the conclusion above finishes the proof. \end{proof}
  Combining Lemma \ref{lem-comeager} and Lemma \ref{using meagerness} one gets the following:
  \begin{cor}
  \label{cor-zariski}
  Suppose  $\M$ is a countable homogeneous first order structure in which algebraic closure is locally finite.  Assume  all non-trivial automorphims of $\M$ are strongly unbounded. Then $\tau_Z$ is not a group topology for $\Aut{\M}$.
  \end{cor}
  There is another consequence of the meagerness of solution sets of equations worth mentioning. We start with the observation that the multivariate case follows from the univariate case.
  \begin{lem}
  Let $(G,\tau)$ be a non-meager Polish group. If the set of solutions of any non-trivial equality in one variable with parameters in $G$ is meager in $G$ then the same holds for non-trivial equalities with parameters in any number of variables.
  \end{lem}
  \begin{proof}
  Let $w(x,\alpha)=1$ be the equation in question, where $x=(x_{0},x_{1},\dots, x_{k})=(x_{0},y)$ by induction. For each value of $y_{0}:=(x_{1}^{0},\dots, x_{k}^{0})$ consider the term in $x_{0}$ obtained by replacing  each $x_{j}$ by the element $x_{j}^{0}$ for $j\geq 1$ in $w(x,\alpha)$ and then merging together all consecutive constants.
  If all the resulting products that lay between two consecutive occurrences of $x_{0}^{\epsilon}$ with opposite exponents are non-zero then the resulting expression is already reduced and is a non-trivial inequality in $x_{0}$ (which without loss of generality appeared in the original expression). Therefore for such $y_{0}$ comeagerly many values of $x_{0}$ satisfy the equation, by the single variable case. Now, the condition above can be expressed as a system of finitely many non-trivial inequalities in the variable $y$ and hence holds for comeagerly many values of $y$ by the induction hypothesis.
  \end{proof}
  
  Using the Baire Category Theorem one can derive the following corollary:
  \begin{cor}
  Let $(G,\tau)$ be a non-meager Polish group such that the set of solutions of any non-trivial equation in one variable with parameters in $G$ is meager. Then for any countable subgroup $A\leq G$ there exists some free group $F\leq G$ over a countable base such that $\subg{A,F}\cong A\ast F$.
  \end{cor}

  \subsection{Fra\"iss\'e constructions}
    Suppose $\M$ is a countable first order structure.  Let $G=\Aut{\M}$ and for any $\alpha\in G$ define  $\mov(\alpha):=\{m\in M \,|\, \alpha(m)\neq m\}.$
    
    Recall the setting from Subsection \ref{subsec-Fra}; namely $\mathfrak L$ is a relational signature and $\mathcal K$ is a class of finite $\mathfrak L$-structures.
    Suppose $A,B,C$ are $\mathfrak L$-structures with $A,B\subseteq C$.
    Let $B'\subseteq B\backslash A$. By $\Delta^C(B';A)$ 
    we mean the set of all positive Boolean combinations of all $\phi(b,a)$ where $b\subseteq B'$, $a\subseteq A$ and $\phi$ is an atomic $\mathfrak L$-formula. 

    \subsubsection{Free amalgamation classes}
      \label{subsubsec-free-AP}
      For a definition of free amalgamation classes see Subsection \ref{subsec-Fra}. First, we remind the reader the following fact about strong amalgamation classes (hence also about free amalgamation classes).
      
      \begin{fact}     \label{fact-trivial}  Suppose $\mathcal  K $ is a Fra\"iss\'e class with strong amalgamation. Then $\acl_M(A)=A$ for all $A\subseteq M$ where $\M=\fl{\mathcal K}$.      \end{fact}
      
      From Corollary 2.10 in \cite{MacTent} one can easily conclude the following.
      
      \begin{fact}     \label{lem-free-unbound-1}   Suppose $ \mathcal K $ is a Fra\"iss\'e class with the free amalgamation property and $\M=\fl{K}$. Assume $\Aut{M}$ is transitive and $\Aut{M}\neq S_\infty$. Then every non-trivial automorphism of $\M$ where $\M=\fl{\mathcal K}$ is strongly unbounded.      \end{fact}
      
      Here we introduce a more general setting for Fra\"iss\'e classes that include the setting of \cite{MacTent}. Then Lemma \ref{lem-free-unbound} is a mild generalisation of Corollary 2.10 in \cite{MacTent} which we give a complete proof here.

      \begin{defn} Suppose $\mathcal K$ is a Fra\"iss\'e class. We say $\mathcal K$ is {\it non-discrete} (ND) if there is $m\in \N$ (where  $\min\{n_R:R^{n_R}\in \mathfrak L\}\leq m-2\leq \max\{n_R:R^{n_R}\in \mathfrak L\}$)   such that for every $A\in \mathcal K$ with $|A|\geq m$, and $ a_1,a_2\in A$  with $a_1\neq a_2$  and $A'\subset A\backslash\{a_1,a_2\}$ with $|A'|=m-2$ there is $B\in \mathcal K$ such that    \begin{enumerate}
      \item $B=A\coprod\{b\}$;
      \item There is $\phi(ba_1,s)\in \Delta^B(ba_1,A')$ such that $B\models \phi(ba_1,s)$ but $B\models\neg \phi(ba_2,s')$ for all $s'\subseteq A$ where $s'\cong s$.
      \end{enumerate}
      \noindent
      We call $B$ a {\it non-discrete one-point extension} of $A$ of the form $a_1\triangleright_{A'}^d a_2$. Occasionally, we use $m$-ND when we want to specify the cardinally of $|A|$.

      \end{defn}
      Notice that in the definition above if $B\subseteq \fl{\mathcal K}$, $\alpha\in \Aut{\fl{\mathcal K}}$ where $\alpha^{\pm 1}(a_{1})=a_{2}$ and $\alpha$ fixes $A$ setwise, then $d\in\mov(\alpha)$.
      \begin{lem}     \label{lem-free-unbound}   Suppose $ \mathcal K $ is an ND Fra\"iss\'e class with the free amalgamation property. Then every non-trivial automorphism of $\M$ where $\M=\fl{\mathcal K}$ is strongly unbounded.      \end{lem}
      
      \begin{lem}
      \label{lem-tech-rich}
      Suppose $\alpha$ is an automorphism of a first order structure $\M$ where the algebraic closure in $M$ is trivial. Assume for every finite subset $A$ of $M$ and $b\in M\backslash A$ there is a realisation $c$ of $\typ(b/A)$ such that $c\in \mov(\alpha)$. Then $\alpha$ is strongly unbounded.
      \end{lem}
      \begin{proof}
      Given a finite subset $A$ of $M$ and $b\in M\backslash A$, we prove the set $T_b:=\{r\in M\,|\, r\in \mov(\alpha),  r\models \typ(b/A)\}$ is infinite. Assuming $T_b$  is  infinite, because the algebraic closure is trivial, there is a realisation $c\in T_b$ such that $\alpha(c)\notin cA$. This shows $T_b$ is strongly unbounded.
      
      Now we show $T_b$ is infinite. Put $p_0=\typ(b/A)$. By the assumption there is $r_0\models p_0$ where $r_0\in \mov(\alpha)$. Now consider $p_1:=\typ(s_1/r_0A)$ where $s_1\models p$ and $s_1\neq r_0A$ (and such $s_1$ exists because $r_0\notin A=\acl(A)$). Let $r_1\models p_1$ where $r_1\in \mov(\alpha)$ again using the assumption. Clearly $r_1\models p_0$. Inductively, we can build $p_i$'s and find realisation $r_i$'s for $i\in \mathbb N$ which $r_i\models p_j$ and $r_i\in \mov(\alpha)$ when $i\geq j$. Hence we have infinitely many realisation of $p_0$ in $\mov(\alpha)$ and hence $T_b$ is infinite.
      
      \end{proof}

      By  applying Corollary \ref{cor-zariski}, we conclude the following.
      \begin{cor}
      \label{zar-free-AP} Assume
      $\mathcal{K}$ is a non-trivial free amalgamation class. Let $\M=\fl{\mathcal K}$ and $G=\Aut{\M}$ and assmue $G$ acts transitively on $M$. Then $\tau_Z$ is not a group topology.
      \end{cor}
      \begin{proof}
      Since $\mathcal{K}$ is non-trivial there exists some $m\geq 1$ such that the action $G$ on $M$ is $(m-1)$-transitive (let us say every action is $0$-transitive) but not $m$-transitive. Our second hypothesis implies that in fact $m\geq 2$.
      
      This implies that any substructure of size at most $m-1$ in $\mathcal{K}$ is essentially a pure set in the sense that the type of a tuple enumerating all of its members without repetitions does not depend on the order, but that this pure set of size $m-1$ extends to some $A\in\mathcal{K}$ which this is not the case anymore \footnote{We may assume no relation holds for a constant tuple. No positive atomic formula can hold for all tuples of distinct elements of size $\geq 2$ because of free amalgamation.}.
      
      Consider now any $B\in \mathcal K$ with $B\geq m-1$ and any $b_1,b_2\in B$.
      Pick $B'\subset B$ with $|B'|=m-1$  where $b_1\in B',b_2\notin B'$ and let $A=B'\cup \{c\}$ be a non-pure extension of $B'$. Then $A\otimes_{B'}B\in \mathcal K$ is a non-discrete extension of $B$ of the form $b_1\triangleright^c_A b_2$, since no relation holds for a tuple containing both $b_{2}$ and $c$. From Lemma \ref{lem-tech-rich} follows that $\mathcal K$ is an ND \fraisse class. The final conclusion then follows from Corollary \ref{cor-zariski}.
      
      \end{proof}

      \begin{rem}
      
      The assumption of ND in Lemma \ref{lem-free-unbound} is a necessary condition in order to conclude that every non-trivial automorphisms of a \fraisse limits of free amalgamation classes is  strongly unbounded. Notice that by the result of Gaughan in \cite{Gua}, for the   Fra\"iss\'e class all finite  (in empty signature) $\tau_Z=\tau_{st}$ and hence $\tau_Z$ is a group topology on $S_\infty$.
      
      \end{rem}
      
    \subsubsection{Rational Urysohn spaces} Consider the distance monoids $\mathcal Q=(\mathbb Q^{\geq 0},+,\leq,0)$ and $\mathcal Q_b=(\mathbb Q\cap [0,b],+_b,\leq,0)$ for $b\in \mathbb Q^{>0}$ where $+_b$ is addition truncated at $b$. Let $\mathcal U_{\mathcal Q}$ and $\mathcal U_{\mathcal Q_b}$ be the corresponding Urysohn space respectively -- see Section \ref{subsec-setting-Ury}.
      They are precisely the classical rational Urysohn space and rational Urysohn $b$-spheres (or sometimes bounded rational Urysohn space). Here we prove $\tau_Z$ for the automorphism groups of $\mathcal U_{\mathcal Q}$ and $\mathcal U_{\mathcal Q_b}$ are not group topologies.
      
      We briefly discuss how rational Urysohn space and rational Urysohn spheres  are constructed in first order logic as Fra\"iss\'e limits.
      Fix $\mathcal R\in \{\mathcal Q,\mathcal Q_b\,|\,b\in \mathbb Q^{>0}\}$. Let $\mathfrak L$ be the first-order language with a binary relation $R_q(x,y)$ for each $q\in \mathbb \mathcal \mathcal R$. A metric space $(A,d)$ with $\mathcal R$-rational distances is an $\mathfrak L$-structure in the following manner: for $x,y\in A$ and $q\in \mathcal R$ we have
      $R_q(x,y)$ iff $d(x,y)\leq q$. 
      Let $\mathcal C_{\mathcal R}$ be the class of all finite metric spaces with $\mathcal R$-rational distances as $\mathfrak L$-structures. 
      \begin{prop}
      The class $\mathcal C_\mathcal R$  where $\mathcal R\in\{ \mathcal Q, \mathcal Q_b\,|\,b\in \mathbb Q^{>0}\}$ has the amalgamation property.
      \end{prop}
      Let $\mathcal U_{\mathcal R}$ be the corresponding Fra\"iss\'e limit. On easy fact is the following:
      
      \begin{lem}
      \label{lem-ND-Ury}
      The class $\mathcal C_{\mathcal R}$  where $\mathcal R\in\{ \mathcal Q, \mathcal Q_r\,|\,r\in \mathbb Q^{>0}\}$  is non-discrete. 
      \end{lem}
      \begin{proof}
      We prove $\mathcal C_{\mathcal R}$ is $2$-ND.
      Suppose $A\in \mathcal C_{\mathcal R}$ where $|A|\geq 2$. Let $a_1,a_2\in A$ be two elements which we want to separate by a one-point extension. Let $q=d(a_1,a_2)$ and consider $B=\{a_1,a_2,b\}$  to be an $\mathfrak L$-structure with $d(a_1,a_2)=q$ and $d(a_1,b)=\frac{q}{2}$ and $\frac{q}{2}+\epsilon$ where $\epsilon\in(0,\frac{q}{2})$. It is easy to check $B$ is a metric space with rational distance and its diameter is $q$ hence $B\in \mathcal C_\rr$. Note that $B$ is a one-point non-discrete extension of $a_1a_2$ that separates $a_1$ and $a_2$. Now the amalgamation of  $A$ and $B$ over $a_1a_2$ is the one-point extension of $A$ which we are looking for.
      \end{proof}
      \begin{lem}
      \label{zar-Urysohn}
      Non-trivial automorphisms of $\mathcal U_{\mathcal Q}$ and $\mathcal U_{\mathcal Q_b}$ for $b\in \mathbb Q^{>0}$ are strongly unbounded.
      \end{lem}
      \begin{proof}
      By Lemma \ref{lem-ND-Ury} the class $\mathcal C_{\mathcal R}$ is $2$-ND when $\mathcal R\in\{ \mathcal Q, \mathcal Q_b\,|\,b\in \mathbb Q^{>0}\}$. Let $\alpha$ be a non-trivial automorphism of $\mathcal U_{\mathcal R}$. One can easily check that the algebraic closure is trivial in $\mathcal U_{\mathcal R}$. Hence based on Lemma \ref{lem-tech-rich} it is enough to show for a given finite subset $A$ of $\mathcal U_{\mathcal R}$ and $r\in \mathcal U_{\mathcal R} \backslash A$ there is a realization of $p:=\typ(r/A)$ in $\mov(\alpha)$.
      
      First note that $\mov(\alpha)$ is infinite and hence one can find distinct $s,s_1\in U_{\mathcal R}$ such that $s_1=\alpha(s)$ and $ss_1\cap A=\emptyset$. Consider the metric space $ss_1rA$. If $d(r,s)\neq d(r,s_1)$ then $r$ separates $s$ and $s_1$ and hence $r\in \mov(\alpha)$.  Assume now $d(r,s)=d(r,s_1)$. Consider a one-point extension $C=Ass_1\cup c$ of $Ass_1$ with the following properties:
      \begin{enumerate}
      \item $d(a,c)=d(c,s)=l$ for all $a\in A$ where $l\in \mathcal R$ and $\frac{\diam{Arss_1}}{2}<l<\diam{\mathcal U_R}$;
      \item $d(c,s_1)=l'$ where $l'\in \mathcal R$, $l'>l$ and $l'-l\leq \min\{d(s,s_1),d(s_1,a)\,|\,a\in A\}$.
      \end{enumerate}
      It is easy to check $C\in \mathcal C_{\mathcal R}$ and any realisation of $\typ(c/Ass_1)$  is in $\mov(\alpha)$.
      Since there are infinitely many realisations of $\typ(c/Ass_1)$ we can assume $c,\alpha(c)\notin Ar$.
      Rename $c$ and $\alpha(c)$ to $s$ and $s_1$, respectively. Now looking at $D_1:=sA$ and $D_2:=rA$ we can amalgamate them over $A$ if the following distance between $r$ and $s$ in $D_1D_2$ holds:
      $$\max\{|l-d(r,a)|\,|\,a\in A\}\leq d(s,r)\leq \min\{l+d(r,a),\diam{\mathcal U_{\mathcal R}}\,|\,a\in A\}.$$
      Let $l_1:=\max\{|l-d(r,a)|\,|\,a\in A\}$ and  $l_2:=\min\{l+\mu,\diam{\mathcal U_{\mathcal R}}\}$. Clearly $l_1<l_2$.
      In $D_1D_2$ assigning $d(r,s)=t$ for $t\in [l_1,l_2]\cap \mathcal R$ satisfies the triangle inequality.
      
      Choose $t\in  (l_1,l_2)\cap \mathcal R$ and let $D_1D_2$ be the structure such that $d(r,s)=t$.  With abuse of notation call one of its isomorphic copies over $A$ again $D_1D_2$ and $b$ be the element that has the same type as $p$. If $d(r,s)\neq d(r,s_1)$ we are done.
      
      Suppose again $d(r,s)= d(r,s_1)=t$. We claim the following structure is a metric structure: Let $E=eAss_1$ where $\typ(e/A)=p$ and assign $d(e,s_1)=t$ and $d(e,s)=t'$ where  $t< t'$ , $t'\in (l_1,l_2)\cap \mathcal R$. In order to have $eAss_1$ as a metric space we only need to have the triangle inequality for  $(e,s,s_1)$ and for that we choose $t'\in (l_1,l_2)\cap \mathcal R$ in such a way that $t'-t<d(s,s_1)$. That is always possible and hence we conclude every non-trivial automorphism is strongly unbounded.
      \end{proof}
      Now by applying Corollary  \ref{cor-zariski} we conclude
      \begin{cor}
      \label{cor-zar-ury} The Zariski topology $\tau_Z$ for $ \Aut{\mathcal U_{\mathcal Q}}$ and $ \Aut{\mathcal U_{\mathcal Q_b}}$ for $b\in \mathbb Q^{>0}$ is not a group topology.
      
      \end{cor}
      
    \subsubsection{Random tournament}
      
      A {\it tournament} is a digraph where there is exactly one edge between every pair of vertices. According to Lachlan's classification there are only three countable homogeneous tournaments up to isomorphism: the dense linear
      order on the rational numbers $\mathbb Q$, the dense local order $S(2)$, and the tournament $T^\infty$ that is universal for the set of all finite tournaments. It is easy to check that the classes of finite substructure of homogeneous tournaments (also homogeneous digraphs) are $2$-ND.  Moreover, the algebraic closure in all three cases is trivial. One can show the following in this case:
      \begin{lem}
      Given a non-trivial automorphism $\alpha\in \Aut{T^\infty}$ and any finite subset $A$ of $T^\infty$ and $b\in T^\infty\backslash A$ there is a realisation $c\in \typ(b/A)$ such that $c\in \mov(\alpha)$.
      \end{lem}
      \begin{proof}
      One can easily show $\mov(\alpha)$ is infinite when $\alpha$ is non-trivial. Then consider $t\in \mov(\alpha)$ where $t,\alpha(t)\notin A$. Let $B=A\cup \{t,\alpha(t)\}$ and consider the tournament $C=B\cup \{c\}$ where $c\notin B$, and $c$ and $t$ has the same qf-type over $A$ and moreover $c$ relates to $t$ and $\alpha(t)$ with an opposite direction. Then $C$ is a tournament and $B\subseteq C$ and by the theorem of Fra\"iss\'e one can find a copy of $C$ in $M$ over $B$. With abuse of notation denote it again by $C$. Then $c\in \mov(\alpha)$.
      \end{proof}
      Then from Lemma \ref{lem-tech-rich} and Corollary \ref{cor-zariski} we conclude
      \begin{cor}
      \label{cor-zar-rant} The Zariski topology $\tau_Z$ in $\Aut{T^\infty}$ is not a group topology.
      \end{cor}
      However,
      it is not hard to see that if  $\M\in \{\mathbb Q,S(2)\}$ then there are non-trivial automorphisms  of $\M$ which are not strongly unbounded.
      
      \newcommand{\AR}[0]{\Aut R}
      
  \subsection{Products of Fra\"iss\'e classes}
    
    \newcommand{\tp}[0]{\otimes}
    \newcommand{\moving}[0]{discriminating } 

    Given two distinct elements $a,b$ of an $\omega$-categorical structure $\M$ such that $\typ(a)=\typ(b)$ and any
    $k\geq 1$ denote by $\Delta_{a,b}^{k}(x)$ the formula over $x=(x_{i})_{i=1}^{k}$ stating that  $\typ(x/a)\neq \typ(x/b)$ and $x\cap\{a,b\}=\emptyset$. Notice that for any     $\alpha\in \Aut{\M }$ if $b=\alpha(a)$ then at least one component of any realisation of     $\Delta_{a,b}^{k}(x)$ must belong to $\mov(\alpha)$. We denote by $\Delta^{k,\mathfrak L'}_{a,b}$ the result of calculating $\Delta$ in $\M \restriction_{\mathfrak L'}$, where $\mathfrak L'\subset \mathfrak L$.
    We say that a \fraisse class $\mathcal{K}$ in a relational signature $\mathfrak L$  is \emph{\moving} if for each pair of distinct elements $a,b\in \M=\fl{\mathcal{K}}$ there exists $k\geq 1$ such that the formula $\Delta_{a,b}^{k}(x)$ is non-algebraic.
    
    \begin{obs}
    \label{discriminating}Let $\mathcal{K}$ be a non-trivial \fraisse class in a finite relational signature and $\M=\fl{\mathcal{K}}$.
    Then $\mathcal{K}$ is \moving provided one of the following holds:
    \begin{itemize}
    \item $M$ has trivial algebraic closure and the action of $\Aut{\M}$ on $M^{2}\setminus\{(a,a)\}_{a\in M}$ is transitive; or,
    \item $\mathcal{K}$ has free amalgamation and $\Aut{\M}$ acts transitively on $M$.
    \end{itemize}
    \end{obs}
    \begin{proof}
    For the first part, take some $R^{(k)}\in\L$ holding for some $k$-tuples but not for others.
    This implies the existence of $a,b\in M^{k}$ (with pairwise distinct coordinates) differing only in one coordinate $a_{i}\neq b_{i}$ and such that $\typ(a)\neq \typ(b)$. This implies $\Delta_{a_{i},b_{i}}^{k-1}(x)$ is non-algebraic. By our transitivity condition, it follows that
    $\Delta_{c,d}^{k-1}(x)$ is non-algebraic for any distinct $c,d\in M$.
    The second part follows from the proof of Corollary \ref{zar-free-AP}.
    \end{proof}
    
    
    We say $\mathcal{K}$ is \emph{dense} if for any distinct $a,b\in\M=\fl{\mathcal{K}}$ and any non-algebraic 1-type $p$ over finitely many parameters of $\M$ isolated by a formula $\phi(x)$ the formula $\Delta^{1}_{a,b}(x)\wedge p(x)$ is not algebraic.
    
    \begin{defn}
    \label{def: tensor product of fraisse classes}Given two Fra\"iss\'e classes $\mathcal K_1$ and $\mathcal K_2$ over finite relational languages $\L_1$ and $\L_2$, respectively, define $\mathcal K_1\otimes K_2$ to be the class of $\L$-structures $A$ where $\L=\L_1\coprod \L_2$ and $A\restriction_{ \L_{1}}\in \mathcal K_1$ and $A\restriction_{ \L_{2}}\in \mathcal K_2$.
    \end{defn}
    
    \renewcommand{\L}[0]{\mathfrak L}
    \begin{lem}
    \label{products lemma}For $i=1,2$ let $\mathcal{K}_{i}$ be a Fra\"iss\'e class over a finite relational language $\L_{i}$ such that
    $\mathcal{K}=\mathcal{K}_{1}\tp\mathcal{K}_{2}$ is a \fraisse class. Assume that:
    \begin{itemize}
    \item $\mathcal{K}_{1}$ is \moving;
    \item $\mathcal{K}_{2}$ is dense;
    \end{itemize}
    and let $\M=\fl{\mathcal{K}}$. Then any non-trivial element of $G=\Aut{\M}$ is strongly unbounded.
    \end{lem}
    \begin{proof}
    Take $\alpha\in G\setminus\{1\}$ and some non-algebraic formula $\phi(z,a)$ in one variable over a finite tuple $a$ of parameters which isolates a non-algebraic type. We can write $\phi=\phi_{1}\wedge\phi_{2}$, where $\phi_{i}$ is a quantifier free formula in the language $\L_{i}$.
    Suppose that $\alpha(c)=c'$ for distinct $c,c'\in M$. The fact that $\mathcal{K}_{1}$ is \moving implies there is some $k$ such that the formula $\Delta_{c,c'}^{k,\mathfrak L_{1}}$ is non-algebraic.
    In particular, it can be realized in $\M\restriction_{ \L_{1}}$ by some tuple $d$ disjoint from
    $a$. Since the $\L_{2}$ formula $\phi_{2}(x,a)$ is non-algebraic, it is possible to find such $d$ in $M$
    with the property that all of its entries satisfy $\phi_{2}(z,a)$. On the other hand, $\alpha(d_{i})\neq d_{i}$ for some $d_{i}$. Density of $\mathcal{K}_{2}$ then implies $\phi_{2}(z,a)\wedge\Delta^{1,\L_{2}}_{d,d'}(z)$ is not algebraic.
    Therefore, neither is $\phi(x)\wedge\Delta_{d,d'}^{1,\L_{2}}(z)$ (again, by quantifier elimination and the definition of $\mathcal{K}_{1}\otimes \mathcal{K}_{2}$). This implies that $\phi(M)\cap \mov(\alpha)$ is non-empty.
    \end{proof}
    We collect below a handful of particular cases of Lemma \ref{products lemma}.
    \begin{cor}
    \label{cor-products}Let $\mathcal{K}_{1}$ and $\mathcal{K}_{2}$ be two \fraisse classes with strong amalgamation and $\mathcal{K}=\mathcal {K}_{1}\otimes\mathcal{K}_{2}$. Assume $\mathcal{K}_{1}$ is non-trivial and satisfies one of the following:
    \begin{itemize}
    \item The action of $\Aut{\fl{\mathcal{K}_{1}}}$ on the set $M^{2}\setminus\{(a,a)\}_{a\in M}$ is transitive;
    \item $\mathcal{K}_1$ has free amalgamation and the action of $\Aut{\fl{\mathcal{K}_{1}}}$ on $\fl{\mathcal{K}_1}$ is transitive.
    \end{itemize}
    Assume also $\fl{\mathcal{K}_{2}}$ one of the following:
    \begin{itemize}
    \item $(\mathbb{Q},<)$;
    \item The countable dense meet tree;
    \item The cyclic tournament $S(2)$.
    \end{itemize}
    Then the solution set of any non-trivial equation with parameters in $G$ is meager.
    \end{cor}

  \subsection{Hrushovski's pre-dimension construction}
    \label{subsec-Hru}
    Recall the setting in subsection \ref{sec-Hrush-new}.
    Suppose $s\geq  2$ and $\eta\in (0,1]$. Let $\mathcal C_\eta: = \{B \in \mathcal  C \,|\, \emptyset \leqslant B\}$ and $\M^\eta$ be the countable structure that one obtains from Proposition \ref{prop:sgen}.

    Suppose $A$ is a finite subset of $M^\eta$. Using the pre-dimension function $\delta$ one can define the \emph{dimension} of $A$ as $\dim A:=\delta\left(\cl A\right)$, where $\cl A$ is the smallest $\leqslant$-closed finite subset of $\M^\eta$ that contains $A$. Given $b\in M^\eta$ and  $A$  a finite subset of $M^\eta$, we denote $\dim{b/A}$ for $\dim{bA}-\dim A$.   From part (2) of Lemma \ref{lem-predim} and part (2) Proposition $\ref{prop:sgen}$ it follows that $\cl{A}$ is well-defined. Similarly to Lemma \ref{lem-comeager} we prove the following
    \begin{lem}
    \label{thm-Zar}
    Suppose $\alpha$  is a finite tuple of automorphisms of $G=\Aut{\M^\eta}$. Then the set of solutions of a non-trivial equation $w(x,\alpha):=\alpha_0x^{\epsilon_0}\alpha_1\cdots x^{\epsilon_m}\alpha_{m+1}=1$  is meager in $G$. 
    \end{lem}
    In order to prove Lemma \ref{thm-Zar} we recall some facts about non-trivial automorphisms of $\M^\eta$. Recall that $\alpha\in \Aut{ \M^\eta}$ is $\gcl$-bounded if there exists a finite subset $B$ of $ \M^\eta$ such that $m\in \gcl(mB)$ for all $m\in  M^\eta$ where $\gcl(X):=\{m\in M^\eta\,|\,\dim{m/X}=0\}$. One can define an independence notion $\ind^d$ between finite subsets of $M^\eta$ using the dimension function; namely $A\ind^d_BC$ iff $\dim{A/B}=\dim{A/BC}$ where $A,B$ and $C$ are finite subsets of $\M^\eta$. It turns out $\ind^d$ is indeed the forking-independence in $\M^\eta$ and for simplicity we denote it by $\ind$ in $\M^\eta$  and remove the superscript $d$.  From Lemma 3.2.27  and Theorem 3.2.29 in \cite{Zthesis} follows:
    \begin{prop}
    \label{lem-tech-old}
    For every non-trivial automorphism $\alpha\in \Aut{\mathcal M^\nu}$ and $X,Y\in \mathcal C_\eta$ where $X\leqslant Y$ and $Y\cap \gcl(X)=X$, there is $Y'$ where $\typ(Y'/X)=\typ(Y/X)$ and $Y'\ind_X\alpha(Y')$.
    \end{prop}
    
    One can modify the definition of strongly unbounded to strongly $\gcl$-unbounded such that for an automorphism of a structure that $\gcl$ is well-defined. Namely $\alpha $ is {\it strongly $\gcl$-unbounded} if for every finite set $A$ and $b\in M\backslash\gcl(A)$ there is a realisation $c\in M$ of $\typ(b/A)$ where $\alpha(c)\notin \gcl(cA)$.  Proposition \ref{lem-tech-old} implies immediately the following.
    
    \begin{fact}
    \label{fact-strong}
    All non-trivial automorphisms of $\M^\eta$ are strongly $\gcl$-unbounded.
    \end{fact}
    It has to be remarked that Proposition \ref{lem-tech-old} is proving something stronger than just that non-trivial automorphisms are strongly $\mbox{gcl}$-unbounded.
    
    \begin{proof}[Proof of Lemma \ref{thm-Zar}]
    Let $G=\Aut{ \M^\eta}$. 
    We want to show the set of solutions of a non-trivial equation $w(x,\alpha)=\alpha_0x^{\epsilon_0}\alpha_1\cdots x^{\epsilon_m}\alpha_{m+1}=1$ has empty interior in $G$. From Fact \ref{fact-strong} all the non-trivial automorphisms are strongly $\mbox{gcl}$-unbounded.
    
    Essentially the same arguments of the proof of Lemma \ref{lem-comeager} works and we only need to replace $\acl$ by  $\mbox{gcl}$ and apply Fact \ref{fact-strong} when $\alpha_i$'s are non-trivial.
    
    In order to show the starting point of the argument we provide some details and leave the rest (avoiding a repetition).
    We follow closely the proof of Lemma \ref{lem-comeager}.
    Aiming for contradiction suppose that is not the case. Again up to performing a change of variable of the form $x\mapsto x\gamma$ we can assume that there is a finite $\leqslant$-closed subset $B$ such that $G_B\subseteq w(G,\alpha)$.
    We will construct inductively a chain of partial isomorphisms $\id_B=f_{m+1}\subseteq f_{m}\subseteq \dots \subseteq f_0$ together with $a\in \M^\eta$ with the property that: for $0\leq k \leq m+1$
    $$c_k:= \alpha_kf_k^{\epsilon_k}\alpha_{k+1}\cdots f_k^{\epsilon_m}\alpha_{m+1}(a)\notin \gcl(\dom (f_{k}^{\epsilon_{k}}));$$ where $c_k\neq a$ when $k\neq m+1$.
    
    The starting point is choosing $a$ any element in $ \M^\eta\backslash \gcl(B)$. We have two possibilities: If $\alpha_{m+1}$ is $1$, then let $c_{m+1}=a$. Suppose $\alpha_{m+1}$ is strongly $\gcl$-unbounded. Using Fact \ref{fact-strong} and Proposition \ref{lem-tech-old} set $c_{m+1}:=\alpha_{m+1}(c)$ where $c\models \typ(a/B)$ and  $\alpha_{m+1}(c)\notin \gcl(cB)$ and rename $c$ to $a$. 
    Assume now we have successfully constructed $f_{k}$. Then the same arguments of the proof of Lemma \ref{lem-comeager} work only replacing $\acl$ by $\gcl$ and apply Fact \ref{fact-strong} when $\alpha_i$'s are non-trivial.
    \end{proof}
    Then from Lemma \ref{thm-Zar} and Theorem \ref{thm-main-zariski} follows
    \begin{cor}
    \label{Hrus-zar-main-hh}
    The Zariski topology
    $\tau_Z$ for $\Aut{\M^\eta}$ is not a group topology. \end{cor}

  \subsection{Some cases when the Zariski topology is a group topology. }
    
    \newcommand{\I}[0]{\mathcal{I}}

    By a family of generalized intervals in $\M $ we mean a $G$-equivariant collection of data consisting of  a collections $\I$ of infinite subsets of $M$ and a map  $\lambda$ assigning to each $I\in\I$ a set of $2$ elements of $\M $ with the following properties, where we write $I^{*}=M\setminus(I\cup\lambda(I))$:
    \begin{enumerate}[(a)]
    \item \label{c} for $I,J \in \mathcal I$, if $\lambda(I)\subset J$ then either $I\subset J$ or $I^{*}\subset J$. The same holds with $J^{*}$ in place of $J$.
    \item \label{d}for each $I\in\mathcal{I}$ and every $K\in\{I,I^{*}\}$ there exists some $\alpha_{K}\in Aut(\M)$  such that $\mov(\alpha_{K})=K$;
    \item \label{separation} for any distinct $x,y\in M $ there exists some $I\in\mathcal{I}$  such that $x\in I$ and $y\notin I$;
    \item  \label{three intervals}for any $I\in\mathcal{I}$ and $x\in I$ there exist $I',I'',I_{1},I_{2}\in\I$ and contained in $I$ with the following properties.
    \begin{enumerate}[(i)]
    \item \label{1}$x\in I''\subset I'$;
    \item \label{2}$I_{1},I'',I_{2}$ are disjoint;
    \item \label{3}$\lambda(I')\cap I_{i}\neq\emptyset$ for $i=1,2$.
    \end{enumerate}
    
    \end{enumerate}
    
    Given a structure $\M$, $x\in M$  and $I\subseteq M$ we write $[x:I]=\{g\in Aut(\M)\,|\,gx\in I\}$.

    \begin{lem}
    Let $\M $ be a structure and $G=\Aut{\M }$. For any family of generalized intervals $\mathcal{I}$ in $\M $ the collection  $\{[x,I]\,|\,x\in I\in\mathcal{I}\}$ forms a sub-base of neighbourhoods of $1$ for the Zariski topology on $G$ and the latter is a group topology.
    \end{lem}
    \begin{proof}
    Given $I_{1},I_{2}\in\mathcal{I}$ let $\Lambda_{I_{1},I_{2}}$ be the set of solutions in $G$ of the inequality   $u_{I_{1},I_{2}}(x)=[\alpha_{I_{1}}^{x},\alpha_{I_{2}}]\neq 1$. Clearly $g^{-1}I_{1}\cap I_{2}\neq\emptyset$ for any $g\in\Lambda_{I_{1},I_{2}}$.
    
    Let $\Gamma_{I_{1},I_{2}}$ be the intersection of all sets of the form
    $\Lambda_{J_{1},J_{2}}$, where $J_{i}\in\{I_{i},I_{i}^{*}\}$.
    Observe that if $g\in\Lambda_{I_{1},I_{2}}$ then both $g^{-1}\lambda(I_{1})\cap I_{2}$ and $g^{-1}\lambda(I_{1})\cap I_{2}^{*}$ must be non-empty by \ref{c}..
    
    We will now show that every set of the form $[x:I]$ for $x\in M$ and $I\in\I$ is a neighbourhood of the identity in the Zariski topology.
    Let $x\in M$ and $I\in\mathcal{I}$ with $x\in I$. Choose $I',I'',I_{1},I_{2}$ as in \ref{three intervals} and consider the set   		$\Theta=\Lambda_{I',I_{1}}\cap\Lambda_{I',I_{2}}\cap\Gamma_{I'',I'}$.
    Clearly $1\in\Theta$. Now let $g\in\Theta$.
    Since $g\in\Lambda_{I',I_{1}}\cap\Lambda_{I',I_{2}}$ we must have $g^{-1}\lambda(I')\cap I_{i}\neq \emptyset$ for $i=1,2$ and thus $g^{-1}\lambda(I')\subseteq I_{1}\cup I_{2}$. Since $(I_{1}\cup I_{2})\cap I''=\emptyset$, then by \ref{c} either    $g^{-1}I'$ contains $I''$ or has empty intersection with $I''$.
    Since $g\in\Gamma_{I'',I'}$, the former must be the case, which implies that $x\in g^{-1}I'$, i.e. $gx\in I'\subset I$. Hence $\Theta\subseteq[x:I]$.

    All that is left to show is that the neighbourhoods $[x:I]$ for $I\in \mathcal I$ form a basis of a Hausdorff group topology. Given $x$ and $I\in\I$, let $I',I''$, $I_{1}$ and $I_{2}$ be as given by \ref{three intervals}.
    Suppose that $\lambda(I')=\{a_{1},a_{2}\}$ with $a_{i}\in I_{i}$. Since $I^{*}\cap I'=\emptyset$, the argument above implies that for any $h\in[a_{1}:I_{1}]\cap[a_{2}:I_{2}]$ we have
    $h(I')\subset I$. So $([a_{1}:I_{1}]\cap[a_{2}:I_{2}])\cdot[x:I']\subseteq [x:I]$ and thus multiplication is continuous. Continuity of inversion can be checked in a similar way. Hausdorffness is straightforward from property \ref{separation}.
    \end{proof}

    \begin{defn}
    A {\it tree} is a partial order $(T,\leq )$ where for each $t\in T$ the set $\{s\in T\,|\,s\leq t\}$ is a linear order. We say $(T,\leq)$ is a {\it meet tree} if for every $t_1,t_2\in T$ the set $\{s\in T\,|\,s\leq t_1,t_2 \}$ has a greatest element which we denote it by $meet(t_1,t_2)$.
    We say a meet tree $(T,\leq)$ is {\it dense} if \begin{itemize}
    \item for any $t$ the set $\{s\in T\ |\ s\leq t\}$ is dense and has no first element;
    \item every point $t$ is a meet of infinitely many  pairs.
    \end{itemize}
    \end{defn}
    
    \begin{thm}
    The Zariski topology $\tau_{Z}$ on $\Aut{\M }$ is a group topology in the following cases:
    \begin{itemize}
    \item \label{case q}Any non-trivial reduct of $(\mathbb{Q},<)$.
    \item \label{case cyclic tournament}The cyclic tournament $S(2)$.
    \item \label{case tree}The countable dense meet tree. In this case $\tau_{Z}=\tau_{st}$.
    
    \end{itemize}
    \end{thm}
    
    \begin{proof}
    In the case of a non-trivial reduct of $(\mathbb{Q},<)$ the collection $\mathcal{I}$ consists of all bounded intervals in the structure and $\lambda(I)$ are the two endpoints of $I$.
    
    For $S(2)$ we can take as $\mathcal{I}$ the collection of all the sets of the form    $\Delta_{a,b}=\{x\,|\,x\neq a\wedge x\neq b\wedge\neg E(x,a)\leftrightarrow E(x,b)\}$, where $E(a,b)$. Notice that this consists of the union of an interval $(a,b)$ on the circle with its antipodal interval, which lacks a corresponding closed interval in the structure. Here we let $\lambda(\Delta_{a,b})=\{a,b\}$.
    This can be described as the union of two intervals on opposite sides of the circle.
    
    For the dense meet tree as $\I$ we take the collection of all sets the form $I(a,b)=\{a<meet(b,x)<b\}$ where $a<b$ and we let $\lambda(I(a,b))=\{a,b\}$.
    Take incomparable $b_{1},b_{2}$ and let $c=meet(b_{1},b_{2})$. Take $a<b$. Then $I(a,b_{1})\cap I(a,b_{2})=\{x\in M\,|\,a<x\leq c\}$. Note that $W:=[c:I(a,b_{1})\cap I(a,b_{2})]\in\mathcal{N}_{\tau_{Z}}(1)$. Let now $g\in W\cap W^{-1}$. Since $g\in W$ the previous discussion implies that $gc\leq c$ and since $g^{-1}\in W$ we know that $g^{-1}c\leq c$, that is $c\leq gc$ so in fact $gc=c$. It follows that $W\cap W^{-1}\leq G_{c}$ and thus $\tau_{Z}=\tau_{st}$.
    \end{proof}

  \subsection{$a$-minimality}
    
    Let $G$ be a group. The intersection of all Hausdorff topological groups structures on $G$ is called the {\it Markov topology}, denoted by $\tau_M$. The topology $\tau_M$ is always $\mathrm{T_1}$ but not necessarily a group topology.
    
    We say that a group $G$ is  {\it $a$-minimal} if $(G,\tau_M)$ is a topological group.
    Notice that if $\tau_{Z}$ is a group topology, then $\tau_{M}=\tau_{Z}$.
    \begin{question}
    For which (sufficiently homogeneous) structures $\M$ is $\Aut{\M}$ $a$-minimal?
    \end{question}
    

\section{Topologies and types}
  \label{sec-typ}
  
  \newcommand{\subspace}[0]{subspace }   
  Let $\M$ be a first order structure and $T=Th(\M)$. Consider two tuples of variables  $x=(x_{m})_{m\in M}$ and $y=(y_{m})_{m\in M}$ indexed by the elements of $M$.
  Given some finite tuple $a=(a_{1},a_{2},\dots, a_{k})\subset M$ we write $x_{a}$ in lieu of $(x_{a_{1}},x_{a_{2}},\dots, x_{a_{k}})$.
  Let $p_{M}(x)=\typ(M)$, where variable $x_{m}$ is made to correspond with $m\in M$.
  Let $R(\M)$ stand for the collection of all $T$-complete types in variables $x,y$
  containing $p_{M}(x)\cup p_{M}(y)$ and let $R^{pa}(\M)$ for the collection of
  partial types in variables $x,y$ in $T$ containing $p_{M}(x)\cup p_{M}(y)$ (i.e., of all closed subsets of $R(\M)$ in the logic topology). Here we assume types are deduction closed. Given any partial type $p(x,y)$ we will denote the deduction closure of $p(x,y)\cup p_{M}(x)\cup p_{M}(y)$ in $T$  as $\subg{p}$.
  The set $R^{pa}(\M)$ can be endowed with the so-called logic topology, which we denote by $\tau_{L}$, generated by neighbourhoods of the form $[\phi]=\{p\in R(\M)\,|\,\phi\in p\}$, where $\phi$ is any formula in $(x,y)$. The result is  Stone space. \\
  
  Given $p^{1},p^{2}\in R^{pa}(\M)$ we let
  $(p^{1}\frp p^{2})(x,y)\in R^{pa}(\M)$ denote the collection of all formulas
  $\psi(x,y)$ such that there exist $\phi^{i}(x,y)\in p^{i}(x,y)$, $i=1,2$
  such that
  \begin{align*}
  \phi^{1}(x,z)\wedge\phi^{2}(z,y)\vdash \psi(x,y).
  \end{align*}
  It can be checked that $*$ endows $R^{pa}(\M)$ with a semigroup structure.
  If we let $0=\subg{\emptyset}\in R^{pa}$ then clearly $p\frp 0=0$ for any $p\in R^{pa}$.
  We write $p\leq q$ for $p\vdash q$.
  
  Given $p\in R^{pa}$, let $\bar{p}\in R^{pa}$ be defined by $\theta(x,y)\in\bar{p}\leftrightarrow\theta(y,x)\in p$.
  Every  $g\in \Aut \M$ is associated to some $\iota(g)=\subg{\{x_{gm}=y_{m}\}_{m\in M}}\in R^{pa}$. It can be easily checked that $\iota$ is a continuous homomorphic embedding of $(G,\tau_{st})$ into $(R^{pa}(\M),\tau_{L})$ whose image is contained in $R(\M)$. From now on we will write simply $g$ instead of $\iota(g)$.
  Notice that $p^{g}:=g^{-1}\frp p\frp g=\{\phi(x_{a},y_{b})\,|\,\phi(x_{g\cdot a},y_{g\cdot b})\in p\}$
  for any $p\in R^{pa}$ and $g\in G$.
  Notice that $*$ is a continuous map $R^{pa}(\M)\times R^{pa}(\M)\to R^{pa}(\M)$ and $p\mapsto\bar{p}$ continuous with respect to $\tau_{L}$.  For the first, notice that given $p_{1},p_{2}\in R^{pa}(\M)$ and a formula $\phi(x,y)$ with $p_{1}\frp p_{2}\in N_{\phi(x,y)}$, the definition of $\frp$ together with compactness implies the existence of $\phi_{1}(x,z)\in p_1$ and $\phi_{2}(z,y)\in p_2$ such that $T\cup\{\phi_{1}(x,z),\phi_{2}(z,y)\}\vdash\phi(x,y)$, which implies that $N_{\phi_{1}}\frp N_{\phi_{2}}\subseteq N_{\phi}$.  \\
  
  \begin{defn}
  \label{def: invariant idempotents} Suppose $\M$ is an  $\L$-structure  and $G=\Aut \M$. We say that $q\in R^{pa}$ is an invariant idempotent if the following conditions are satisfied:
  \begin{enumerate}
  \item  \label{idp1}$1_{G}\leq q$;
  \item  \label{idp2}$q=\bar{q}$;
  \item \label{idp3}$q\frp q=q$; and,
  \item \label{idp4}$q=q^{g}$ for any $g\in G$.
  \end{enumerate}
  \end{defn}
  Notice that assumption \ref{idp1} implies $q=1\frp q\leq q\frp q$, so that item \ref{idp3} could be replaced by the a priori weaker condition  $q*q\leq q$.
  
  Given a formula $\phi(x,y)$, let $N_{\phi}:=\iota^{-1}([\phi])=\{g\in G\,|\,\M\models\phi(ga,b)\}$.
  Given an invariant idempotent $q\in R^{pa}(\M)$ let 	$\mathcal{N}_{q}=\{N_{\phi}\,|\,\phi(x,y)\in q\}$. The following can be seen as a generalization of Lemmas \ref{is group topology} and \ref{monotonicity moduli}.
  
  \begin{lem}
  \label{properties of type topologies}Given any structure $\M$ the following statements hold, where $G=\Aut \M$:
  \begin{enumerate}
  \item \label{topology defined} Any invariant idempotent $q\in R^{pa}(\M)$ the family $\mathcal{N}_{q}$ forms a basis of neighbourhoods of a (unique) group topology $\tau_{q}$ on $G$.
  \item \label{kernel of the topology}The closure of $1$ in $\tau$ coincides with
  the collection of all $g\in G$
  such that $g\leq q$.
  \item \label{functoriality} Given invariant idempotents $p,q\in R^{pa}(\M)$ such that $p\leq q$
  we have $\tau_{p}\supseteq\tau_{q}$ and then the implication from right to left holds as well
  if $\M$ is countable and $\omega$-saturated.
  \end{enumerate}
  \end{lem}
  \begin{proof}
  On the one hand for any $\phi(x_{A},y_{B})\in q$ we have:
  \begin{align*}
  N_{\phi(x,y)}^{-1}&=\{g\in G\,|\,\M\models\phi(g^{-1}a,b)\}\\ &=\{g\in G\,|\,\M\models\phi(a,gb)\} =N_{\phi(y,x)}\in\mathcal{N}_{\bar{q}}=\mathcal{N}_{q}.
  \end{align*}
  On the other hand, the condition $q\frp q=q$ is equivalent to the following: for any finite $A$ and $B$ there is $C\subset M$ and formulas $\psi(x_{A},z_{C}),\psi'(z_{C},y_{B})\in q$ such that modulo $T$ we have:
  \begin{align}
  \label{implication} p_{M}(x)\cup p_{M}(y)\cup p_{M}(z)\cup\{\psi(x_{A},z_{C})\wedge\psi'(z_{C},y_{B})\}\vdash\phi(x_{A},y_{B}).
  \end{align}
  Let $N=N_{\psi(x_{A},y_{C})\wedge\psi(x_{C},y_{B})}$.
  Given $h,g\in N$ we have $\M\models\psi(gA,C)\wedge\psi'(hC,B)$. Formulas are of course $h$ invariant, hence $\M\models\psi(hgA,hC)$.
  Likewise, $hgA\models p_{A}$ and $hC\models p_{C}$ and thus by \ref{implication} we conclude that $\M\models\phi(hgA,B)$ and therefore $hg\in N_{\phi}$. This settles part
  \ref{topology defined}.
  Part \ref{kernel of the topology} follows easily from the fact that $\iota(g)$ is a complete type for $g\in G$ and left to the reader.
  As for \ref{functoriality}, the implication from left to right is trivial.
  Assume now $\mathcal{M}$ is $\omega$-saturated and we are given $p,q$ such that $p\nleq q$ . Then there exists some $\phi(x_{a},y_{a})\in q$ for $a\in [M]^{<\omega}$ such that $p\nleq\phi$.
  
  This implies there exists some $g\in G$ such that $\M\models\psi(ga,a)$, for  each $\psi(x,y)\in p$ but $\M\models\neg\phi(ga,a)$.
  This implies that $N_{\phi}$ cannot contain $N_{\psi}$ for any $\psi\in p$.
  \end{proof}
  
  \begin{rem}
  In particular, $1\in G\subset R^{pa}$ is a invariant idempotent. The associated topology $\tau_{1}$  is just the standard topology. It can be checked by inspection that all topologies on automorphism groups that feature in this paper are of the form $\tau_{q}$ for some invariant idempotent $q$.
  \end{rem}
  
  \begin{question}
  Let $\M$ be a countable $w$-categorical (homogeneous) structure. Is it true that any group topology on $Aut(\M)$ is of the form $\tau_{q}$ for some invariant idempotent $q\in R^{pa}$?
  \end{question}
  
  \subsection{Non-minimality in the trivial $\acl$ case}
    
    Fix some structure $\M$ in a finite relational language in which $\acl$ is {\it trivial}, i.e. $\acl(A)=A$ for any finite $A\subset M$.
    Consider the type $q_{inf}\in R^{pa}(\M)$ consisting of all the formulas of the form $\phi_{A,B}(x,y)\in\typ(A,B)$, where $A\cap B=\emptyset$.
    Notice that $q_{inf}$ is clearly invariant under the action of $\Aut \M$ on $x_{M}$ and $y_{M}$.
    
    \begin{defn}
    We say that $\M$ has the separation property if for any two disjoint finite tuples $a,b\in [M]^{<\omega}$ there exists 	$c\in [M]^{<\omega}$ disjoint with both $a$ and $b$ such that $\typ^{x,z}(a,c)\cup \typ^{z,y}(c,b)\vdash \typ^{x,y}(a,b)$.
    \end{defn}
    
    \begin{lem}
    \label{density}The type $q_{inf}$ is an invariant idempotent in $R^{pa}(\M)$ if and only if $\M$ has the separation property. Moreover, $q_{inf}\not\vdash 1$ so that $q_{inf}$ is strictly coarser than $\tau_{st}=\tau_{1}$ if $\M$ is countable and $\omega$-saturated.
    \end{lem}
    \begin{proof}
    Properties \ref{idp1}, \ref{idp2} and \ref{idp4} are immeidate from the definition. For property \ref{idp2} all we need to check is that $q*q\leq q$ as remarked after Definition \ref{def: invariant idempotents}, but this is precisely the content of the separation property, as in its definition
    $\typ^{x,z}(a,c)\cup \typ^{z,y}(c,b)\vdash \typ^{x,y}(a,b)$ we have
    $\typ^{x,y}(a,c)\cup\typ^{x,y}(c,b)\subseteq q_{inf}$ and thus $\typ^{x,y}(a,b)\subseteq q_{inf}*q_{inf}$ for the arbitrary fragment $\typ^{x,y}(a,b)\subseteq q_{inf}$ we started with.
    
    If $q_{inf}=q_{1}$, then for any $b\in M$ there must be some finite $A\subseteq M\setminus\{b\}$ such that $tp^{x_{A},y_{b}}(A,b)\vdash y_{b}=x_{b}$, which can only be the case if $b\in dcl(A)$. The final claim then follows from last point of Lemma \ref{properties of type topologies}.
    
    \end{proof}
    
    
    Distal theories are a particular class of NIP theories introduced in \cite{simon2013distal}. One main feature is the following fact (Theorem 21 in \cite{chernikov2015externally}):
    \begin{fact}
    \label{distality} Let $T$ be distal. Then for any formula $\phi(x,y)$ there is a formula $\theta(x,z)$ such that for any $\typ^{\phi}(a/C)$ over a finite set of parameters $C$ there is a tuple $d\subset C$ such that $\theta(a,d)$ holds, and $\theta(x,d)\vdash \typ^{\phi}(a/C)$, i.e $\theta(x,y)\cup \typ^{y}(d,C)\vdash \typ^{\phi}(x/C)$, where $|y|=|d|$.
    \end{fact}

    \begin{lem}
    Let $\M$ be any distal \fraisse limit in a finite relational language with trivial algebraic closure. Then $\M$ has the separation property.
    \end{lem}
    \begin{proof}
    Consider any two disjoint finite tuples $a,b\in M$. Since $\M$ has quantifier elimination, there exists some formula $\phi(x,y)$ such that for any $C\subset M$ the full type $\typ(a/C)$ is equivalent to the $\phi$-type $\typ^{\phi}(a/C)$ ($|a|=|x|$). Let $\theta(x,z)$ be the formula provided by Fact \ref{distality} and let $s=|z|$.
    Take a sequence $b_{-s},b_{-s+1},\dots, b_{0}=b,b_{1},\dots, b_{s}$ of instances of $\typ(b/a)$ indiscernible over $a$, where $b_{i}$ and $b_{j}$ are disjoint for $i\neq j$. Let $C=b_{-s}b_{-s+1}\dots, b_{s}$ and $d$ be the tuple obtained from applying \ref{distality} to $\typ(a/C)$. Let $J$ be the set of indices $j\in\{-s,-s+1,\dots, s\}$ such that $d\cap b_{j}\neq\emptyset$. Now, there must be some $j_{0}\in\{-s,-s+1,\dots, s\}\setminus J$ and some order preserving bijection $\phi:J\cup\{j_{0}\}\to J'\subseteq\mathbb{Z}$ sending $j_{0}$ to $0$. Since $(b_{i})_{i}$ is indiscernible,  the fact that $\typ(a/b_{l})_{l\in J}$ isolates
    $\typ(a/b_{l})_{l=-s}^{s}$ implies that $\typ(a/b_{l})_{l\in J'\setminus\{0\}}$ isolates $\typ(a/b_{l})_{l\in J'}$ so that the tuple
    $C=(b_{l})_{l\in J'\setminus\{0\}}$ witnesses the separation property for the pair $(a,b)$.
    \end{proof}

    \begin{cor}
    \label{non-minimality}Let $\M$ be any distal \fraisse limit in a finite relational language with trivial algebraic closure. Then the type $q_{inf}$ defines a group topology on $G=\Aut{\M}$ strictly coarser than $\tau_{st}$.
    \end{cor}
    
\section*{Acknowledgements}\label{ackref}
  The authors would like to thank Dugald Macpherson, David M. Evans, Itay Kaplan and Todor Tsankov for the encouraging comments and thoughtful suggestions.


\end{document}